%% file: main_arxiv.tex
\documentclass[11pt]{article}
\newcommand{\blind}{1}
\renewcommand{\baselinestretch}{1}

\def\vt{{\vartheta}}
\def\cs{{Cauchy–Schwarz }}

\def\Id{\operatorname{I}}

\def\tzt{{\tilde \zeta}}

\usepackage[colorlinks,citecolor=blue,urlcolor=blue]{hyperref}       
\hypersetup{
  colorlinks,
  linkcolor={blue},
  citecolor={blue},
  urlcolor={blue}
}
\usepackage{url}            
\usepackage{booktabs}       
\usepackage{amsfonts}       
\usepackage{nicefrac}       
\usepackage{xcolor}         
\usepackage{color}
\usepackage{bm,multirow,xspace,todonotes}
\usepackage{enumerate,mathtools}
\usepackage[inline]{enumitem}
\usepackage{graphicx}
\usepackage{epstopdf,amsmath}
\usepackage{bbm}
\usepackage{rotating}
\usepackage{amssymb, amsthm}

\usepackage{mathrsfs}
\usepackage{makecell}
\usepackage[authoryear,round]{natbib}

\renewcommand{\baselinestretch}{1.1}

\allowdisplaybreaks

\usepackage{amsmath}    
\usepackage{algorithm}

\usepackage{color}

\usepackage{caption}

\usepackage{algpseudocode}
\usepackage{graphicx}
\usepackage{epstopdf}

\allowdisplaybreaks

\newcommand*{\colorboxedAux}[3]{%
  \begingroup
    \colorlet{cb@saved}{.}%
    \color#1{#2}%
    \boxed{%
      \color{cb@saved}%
      #3%
    }%
  \endgroup
}

\DeclareMathOperator*{\argmax}{argmax}
\DeclareMathOperator*{\argmin}{argmin}

\newcommand{\abs}[1]{\left\lvert #1 \right\rvert}
\DeclareMathOperator*{\iid}{\text{i.i.d.}}
\newcommand{\ceil}[1]{\lceil #1 \rceil}
\newcommand{\floor}[1]{\lfloor #1 \rfloor}

\newcommand{\expec}[2]{\mathbb{E}_{#2}[ #1 ] }
\newcommand{\norm}[1]{\lVert #1 \rVert}
\newcommand\numberthis{\addtocounter{equation}{1}\tag{\theequation}}  

\def\fn[#1]#2{{f_{#1}\left(x_{#2}\right)}}

\newcommand{\kb}[1]{{\color{cyan}\bf[KB: #1]}}
\newcommand{\ar}[1]{{\color{red}\bf[AR: #1]}}

\newtheorem{lemma}{Lemma}[section]
\newtheorem{theorem}{Theorem}[section]

\newtheorem{assumption}{Assumption}[section]
\newtheorem{remark}{Remark}

\newcommand*{\colorboxed}{}
\def\colorboxed#1#{%
  \colorboxedAux{#1}%
}

\usepackage[]{mdframed}

\def\exp{{\rm exp}}

\def\cI{{\cal I}}
\def\cF{{\cal F}}

\def\tr{{\rm tr}}

\def\te{{\tilde e}}
\def\tS{{\tilde S}}

\def\slinv{(\textstyle \sum_{i=1}^n l_i)^{-1}}
\def\xik{\xi_k(\theta_{k-1},x_k)}
\def\tzk{\tilde{\zeta}_k}

\newenvironment{talign*}
 {\csname align*\endcsname}
 {\endalign}

 \numberwithin{equation}{section}

\begin{document}

\def\spacingset#1{\renewcommand{\baselinestretch}%
{#1}\small\normalsize} \spacingset{1}


\if1\blind
{
  \title{\bf Online covariance estimation for stochastic \\gradient  descent under Markovian sampling}
  \author{Abhishek Roy\thanks{
    This author gratefully acknowledges \textit{support from NSF via grant CCF-1934568, when the author was affiliated with UC Davis  where  part of the work was done.}}\hspace{.2cm}\\
    Halicioğlu Data Science Institute, University of California, San Diego\\
    and \\
    Krishnakumar Balasubramanian\thanks{
    This author gratefully acknowledges \textit{support from NSF via grant DMS-2053918.}}\\
    Department of Statistics, University of California, Davis}
  \maketitle
} \fi

\if0\blind
{
  \bigskip
  \bigskip
  \bigskip
  \begin{center}
    {\LARGE\bf Online covariance estimation for stochastic \\gradient  descent under Markovian sampling}
\end{center}
  \medskip
} \fi

\bigskip

\begin{abstract}
We investigate the online overlapping batch-means covariance estimator for Stochastic Gradient Descent (SGD) under Markovian sampling. Convergence rates of order $O\big(\sqrt{d}\,n^{-1/8}(\log n)^{1/4}\big)$ and $O\big(\sqrt{d}\,n^{-1/8}\big)$ are established under state-dependent and state-independent Markovian sampling, respectively, where $d$ is the dimensionality and $n$ denotes observations or SGD iterations. These rates match the best-known convergence rate for independent and identically distributed ($\iid$) data. Our analysis overcomes significant challenges that arise due to Markovian sampling, leading to the introduction of additional error terms and complex dependencies between the blocks of the batch-means covariance estimator. Moreover, we establish the convergence rate for the first four moments of the $\ell_2$ norm of the error of SGD dynamics under state-dependent Markovian data, which holds potential interest as an independent result. Numerical illustrations provide confidence intervals for SGD in linear and logistic regression models under Markovian sampling. Additionally, our method is applied to the strategic classification with logistic regression, where adversaries adaptively modify features during training to affect target class classification.
\end{abstract}

\noindent%
\vfill

\newpage
\spacingset{1} 
\input{intro}

\newpage
\bibliographystyle{abbrvnat}
\bibliography{twocovar}
\appendix
\input{appendix}

\end{document}

%% file: intro.tex

\section{Introduction}\label{sec:intro}
Many statistics and machine learning problems could be formulated as solving an underlying optimization problem of the form
\begin{align*}
    \argmin_{\theta\in\mathbb{R}^d}f(\theta)=\argmin_{\theta\in\mathbb{R}^d}\expec{F(\theta;x)}{\pi},\numberthis\label{eq:mainprobstand}
\end{align*}
where $\theta$ is the parameter to learn and $x$ is a random vector sampled from the distribution $\pi$. Alternatively, one can think of \eqref{eq:mainprob} as minimizing $f$ over $\theta$, where $F$ is the estimate of $f$ based on an observation $x$. This problem arises in numerous statistical applications. For example, in linear regression with squared-loss, let $x\coloneqq (z,y)$ where $z$ is the predictor variable and $y$ is the response variable, and let $\theta$ be the coefficient vector (to be optimized for). Then the least-squares linear regression problem boils down to solving \eqref{eq:mainprobstand} with $F\coloneqq (y-\langle z,\theta\rangle)^2$ for the unknown parameter. Note that if the data is drawn \textcolor{black}{$\iid$} from the statistical model, $y=x^\top {\theta}^* +\epsilon$ (with $x$ having a non-degenerate covariance matrix and with $\epsilon$ being a zero-mean finite-variance  noise parameter), the minimizer of~\eqref{eq:mainprobstand} with this choice of $F$ becomes the true model parameter $\theta^*$. Logistic regression can also be formulated as \eqref{eq:mainprobstand} with $F\coloneqq \log(1+\exp(-y\theta^\top z))$, where $x\coloneqq (z,y)$, $z$ is the feature, $y$ is the label, and $\theta$ is the coefficient vector. Similar to linear regression, in the well-specified case, the solution of \eqref{eq:mainprobstand} is the true model parameter. 

Over the last few decades, with the availability of huge datasets, online optimization algorithms for solving~\eqref{eq:mainprobstand} have become increasingly popular due to their small memory requirement and higher computation efficiency. Arguably, Stochastic Gradient Descent (SGD) has been the most popular choice among the online optimization algorithm. The update step of SGD takes the following form,
\begin{align}\label{eq:vanillasgd}
    \theta_{k+1}=\theta_k-\eta_{k+1}\nabla F(\theta_k,x_{k+1}),
\end{align}
where $k$ denotes the iteration index, $\eta_{k+1}$, and $x_{k+1}$ is the step-size, and the observed sample respectively at iteration $k$.
Since SGD is a stochastic algorithm, the estimate $\hat{\theta}$ provided by the algorithm is a random vector. As a consequence a single run of SGD only provides a point estimate of $\theta^*$. From a statistical perspective, quantifying the uncertainty associated with the estimate, by constructing confidence intervals for $\theta^*$ (where with a slight overload of notation, we use $\theta^*$ to denote the minimizer of \eqref{eq:mainprobstand} generically without assuming any statistical model) is preferable. 

A first step towards uncertainty quantification is assessing the limiting distribution of the SGD iterates. Suppose we run the SGD in~\eqref{eq:vanillasgd} for a total of $n$ iterations. It has been shown that for strongly convex objective function $f(\theta)$, when the data-stream $\{x_k\}_k$ are sampled $\iid$ (e.g., \cite{polyak1992acceleration}) or from a state-dependent Markov chain (e.g., \cite{liang2010trajectory}), under suitable regularity conditions, the averaged iterates follow a Central Limit Theorem (CLT), i.e., we have
\begin{align}\label{eq:informLCLT}
    \sqrt{n}(\bar{\theta}_n-\theta^*)\overset{d}{\to}N(0,\Sigma),\qquad\text{where}\qquad \bar{\theta}_n\coloneqq \frac{1}{n}\sum_{k=1}^n\theta_k,
\end{align}
where $\Sigma$ is the limiting covariance matrix of the form $A^{-1}\,S\,A^{-1}$, where $A\coloneqq\nabla^2 f(\theta^*)$. Under $\iid$ sampling, we have $S=\expec{\nabla F(\theta^*,x)\nabla F(\theta^*,x)^\top}{}$. Under the state-dependent Markovian sampling, $S$ is expressed as a limiting covariance of a martingale-difference sequence derived from the data sequence; see Lemma~\ref{lm:poisregular} for details. \textcolor{black}{It is important to observe that under state-dependent Markovian sampling, the minimizer $\theta^*$ of the problem \eqref{eq:mainprobstand} 
corresponds to the equilibrium point or performative stable point \cite{perdomo2020performative} which is different from the \textit{true parameter} of the model from which the data is generated; see Section~\ref{sec:motivation} for details.} We also make the following remarks about the above CLT:
\begin{enumerate}[noitemsep]
    \item In case of state-dependent Markovian sampling, a truncation with re-initialization step is added on top of vanilla SGD update in \eqref{eq:vanillasgd} to enusre that the dynamics is stable, i.e., the iterates remain contained in a compact set (see Algorithm~\ref{alg:tsgd} and Lemma~\ref{lm:finitetrunc} for details).
    \item The step-size sequence $\{\eta_{k}\}_k$ needs to be chosen properly. Intuitively, the step-size can neither be too small nor be too large (see Assumption~\ref{as:liangasA4}, and Remark~\ref{rm:step}).
    \item Moreover, $\Sigma$ turns out to be the optimal covariance, appropriately defined. Under $\iid$ sampling, it achieves the Cramer-Rao lower bound and for Markovian sampling it turns out to be the smallest possible covariance achievable by a class of stochastic approximation algorithms \citep{tang1999asymptotic,liang2010trajectory}.
\end{enumerate}
As the limiting covariance depends on the unknown parameter $\theta^*$, it becomes important to estimate $\Sigma$ to construct a confidence interval. Furthermore, to preserve the above-mentioned advantages of SGD, the estimator should be constructed in an online fashion.

Covariance estimators for the SGD iterates have been explored recently by \cite{fang2018online,fang2019scalable,chen2020statistical,zhu2021online,zhong2023online} for $\iid$ data-stream $\{x_k\}_k$. To estimate $\Sigma=A^{-1}\,S\,A^{-1}$ with $S$ as defined above, \cite{chen2020statistical} proposes a plug-in estimator of the form $A_n^{-1}\hat{S}_nA_n^{-1}$, where $$\textstyle A_n=n^{-1}\sum_{i=1}^n\nabla^2 F(\theta_{i-1},x_i)\quad\text{and}\quad \hat{S}_n=n^{-1}\sum_{i=1}^n\nabla F(\theta_{i-1},x_i)F(\theta_{i-1},x_i)^\top,$$
 which has faster convergence rate than the batch-means estimator. But as also observed by \cite{chen2020statistical} and \cite{zhu2021online}, while the SGD algorithm itself works with stochastic gradients, the above plug-in estimator requires stochastic Hessians (which maybe unavailable or intractable in various problems of interest). More importantly, inverting $A_n$ at every iteration increases the per-iteration computational cost by $O(d^3)$ which can be prohibitively large even in moderate dimensions (i.e., of order fifty to hundred).

Observe that due to time-dependent step-sizes, $\{\theta_k\}_k$ evolves as a inhomogeneous Markov chain even with $\iid$ data stream. This viewpoint enables one to leverage the rich literature available on inference in Markov chains, to develop inferential procedures for SGD. Motivated by this observation, \cite{chen2020statistical} and \cite{zhu2021online} studied overlapping batch-means covariance estimator with time-varying batches developed in the context of SGD inference. Let $\{a_m\}_m$ be a strictly increasing sequence of integers with $a_1=1$. In this approach, for any $k=1,2,\cdots$, we construct a block $B_k$ consisting of the iterates $\{\theta_{t_k},\theta_{t_k+1},\cdots,\theta_k\}$ where $t_k=a_m$ for $k\in [a_m,a_{m+1})$. We use the terms block and batch interchangeably. Let $l_k=|B_k|$ denote the size of the block $B_k$. Then, after $n$ iterations, the batch-means covariance estimator is given by
\begin{align}
    \hat{\Sigma}_n=\frac{\sum_{i=1}^n\left(\sum_{k=t_i}^i\theta_k-l_i\bar{\theta}_n\right)\left(\sum_{k=t_i}^i\theta_k-l_i\bar{\theta}_n\right)^
    \top}{\sum_{i=1}^nl_i}. \label{eq:covestimator}
\end{align}
Here, one needs to evaluate the covariance over batches because of the correlation among the iterates. The goal is to select batches such that the batch means $\textstyle l_i^{-1}\textstyle\sum_{k=t_i}^i\theta_k$ have low correlation among each other. The larger the correlation, the larger is the block-size $l_k$. More intuition about $\hat{\Sigma}_n$ could be found in \cite{zhu2021online}, and reference therein. Classical batch-means estimator maintain a constant batch-size for homogeneous Markov chain \citep{chen2020statistical}. Furthermore, we will see later from Theorem~\ref{th:mainthm} that $\hat{\Sigma}_n$ is a purely-online estimator (or an any-time estimator) since the knowledge of $n$ is not needed apriori to pick the parameters $\{a_m\}_m$.

\cite{zhu2021online} used $\hat{\Sigma}_n$ as in~\eqref{eq:covestimator} to estimate the covariance of SGD with $\iid$ data. As mentioned above, even for $\iid$ data-stream, the SGD sequence forms an inhomogeneous Markov chain, which demands a time-varying batch-size. Since the correlation among the updates depends on the step-size choices, the batch-size is related closely with the step-size choice (see Theorem~\ref{th:mainthm}). As the estimator is motivated by inference for inhomogeneous Markov chain, intuitively one would expect the methodology to extend to a more general data stream. Hence, the $\iid$ assumption in the above works seems to be restrictive and made for the convenience of analysis. In this work, we study the following problem:
\begin{quote}
    What is the rate of convergence of the online batch-means covariance estimator in~\eqref{eq:covestimator} in the context of SGD under Markovian sampling?
\end{quote}

Towards answering the above question, we show in Theorem~\ref{th:mainthm} that $\hat{\Sigma}_n$ in \eqref{eq:covestimator} is a consistent estimator of the true covariance $\Sigma$ under Markovian sampling, 
with the rate of convergence matching that of the $\iid$ data-stream case. While our main focus is on the case when the Markov chain is \textit{state-dependent}, where the transition probability of the Markov chain depends on the iterates of the algorithm, for the sake of completeness, we also show a similar result for the case of state-independent Markov chain in Theorem~\ref{th:simain}.

In general, estimating the asymptotic covariance of a Markov chain requires specially designed estimators like non-overlapping batch-means estimator \citep{glynn1991estimating,kitamura1997empirical,lahiri2003resampling}, overlapping batch-means estimator \citep{politis1999subsampling}, and spectral variance methods \citep{flegal2010batch} due to serial correlation in a Markov chain. In particular, \cite{meketon1984overlapping,lahiri2003resampling,flegal2010batch} show that the asymptotic covariance of overlapping batch-means  estimator is about $33\%$ smaller than non-overlapping batch-means estimator for homogeneous geometrically ergodic Markov chains although the convergence rate is same. As shown by \cite{flegal2010batch}, for a specific choice of window size, namely Bartlett window, overlapping batch-means  estimator is equal to the spectral variance estimator. Moreover, \cite{flegal2010batch} shows that for general windows and geometrically ergodic homogeneous Markov chain, the consistency of spectral variance estimator requires stronger assumptions on the moments of the invariant distribution compared to overlapping batch-means  estimator. Although we are primarily focusing on estimating the asymptotic covariance of SGD iterates that form an inhomogeneous state-dependent Markov chain, considering that the estimators have similar asymptotic rates of convergence in the homogeneous geometrically ergodic Markov chains and that overlapping batch-means  has a slight edge over the others, we concentrate in this work on the overlapping batch-means  estimator. 
\vspace{-0.05in}
\subsection{SGD under Markovian Sampling}\label{sec:sgdmarkov}
\begin{algorithm}[t]
	\caption{Truncated Stochastic Gradient Descent} \label{alg:tsgd}
	\textbf{Input:} Truncation parameters $\{d_k\}_k$, step-size parameters $\{\eta_k\}_k$, Initial point $\theta_0\in \mathbb{R}^{d}$. 
	\begin{algorithmic}[1]
        \State \textbf{set} $\varkappa_1=0$ 
		\State \textbf{for} $k=0,\cdots n$ \textbf{do}
        \State \textbf{sample } $x_{k+1}\sim P_{\theta_k}(\cdot)$
		\State \textbf{update} 
		\begin{align}\label{eq:sgd}
        \begin{aligned}
            &\theta_{k+1}=\theta_{k}-\eta_{k+1}\nabla F(\theta_k,x_{k+1})\\
      &\varkappa_{k+1}=\varkappa_k
        \end{aligned}
		\end{align}
  \State \textbf{if } $\left(\norm{\theta_{k+1}-\theta_k}_2\geq d_k\text{ or } \theta_{k+1}\notin \mathcal{K}_{\varkappa_k}\right)$
\State\qquad\qquad $\theta_{k+1}=\theta_1$, $x_k=x_1$, and $\varkappa_{k+1}=\varkappa_k+1$
  \State \textbf{end if}
		\State \textbf{end for}
	\end{algorithmic}	
 \textbf{Output:} $\bar{\theta}_n=\frac1n\sum_{k=1}^n\theta_k$
\end{algorithm}
We now discuss the problem setup in more details. Formally, the optimization problem that we consider is given by
\begin{align}\label{eq:mainprob}
    \argmin_{\theta\in\mathbb{R}^d}f(\theta)=\argmin_{\theta\in\mathbb{R}^d}\expec{F(\theta;x)}{x\sim\pi_{\theta}},
\end{align}
where $\pi_\theta$ is the stationary distribution corresponding to the $\theta$ dependent transition operator $P_\theta$. SGD has access to a data sequence $\{x_{k}\}_k$, $P_{\theta_{k-1}}(\cdot,\cdot)$ is the Markov transition kernel dependent on the iterate ${\theta_{k-1}}$ at iteration $k-1$, and $x_k\sim P_{\theta_{k-1}}(x_{k-1},\cdot)$. Let $\nabla f(\theta)=\expec{\nabla F(\theta,x)}{}$ where $x\sim \pi_\theta$, and define the \textit{gradient noise} at $\theta_k$ as 
\begin{align}\label{eq:gradnoise}
\xi_{k+1}(\theta_k,x_{k+1})\coloneqq\nabla F(\theta_k,x_{k+1})-\nabla f(\theta_k).
\end{align}
Under $\iid$ sampling, the gradient noise sequence turns out to be a martingale-difference sequence. However, as we discuss later in Lemma~\ref{lm:poisregular}, under Markovian sampling, the gradient noise in~\eqref{eq:gradnoise} exhibits a more nuanced structure.

To handle this, truncated SGD as discussed in Algorithm~\ref{alg:tsgd} is used. Truncated SGD maintains a sequence of compact sets $\{\mathcal{K}_q\}_q$, called truncation sets, such that 
\begin{align*}
\mathcal{K}_q\subset \text{int}(\mathcal{K}_{q+1}),\qquad \text{and}\qquad \cup_{q\geq 0} \mathcal{K}_q= \Theta,
\end{align*}
where $\text{int}(\cdot)$ denote the interior of a set. Let $\{d_k\}_k$ be a decreasing sequence of thresholds. At each iteration $k$, first, an iterate $\theta_{k+1}$ is generated from $\theta_k$ using the vanilla SGD step as in \eqref{eq:sgd}. Then, if $\theta_{k+1}$ does not belong to the current truncation set $\mathcal{K}_q$ or the change in the consecutive iterates is bigger than a predefined threshold, i.e., if $\theta_{k+1}\notin\mathcal{K}_q$ or $\norm{\theta_{k+1}-\theta_k}_2\geq d_k$, the algorithm is initialized from $\theta_0$ with a bigger truncation set $\mathcal{K}_{q+1}$. In stochastic approximation literature \citep{harold1997stochastic,andrieu2005stability,benveniste2012adaptive}, the iterates are assumed to remain confined in a compact set. Instead of this assumption, Algorithm~\ref{alg:tsgd} automatically guarantees the desired stability.

\textbf{A note about truncation.} Vanilla SGD update (as in~\eqref{eq:vanillasgd}) is not be stable under state-dependent Markovian sampling, i.e., the iterates may not be contained in a compact set; see, for example \citet[Page 2]{andradottir1991projected} and  \cite{chen2002stochastic,andrieu2005stability,liang2010trajectory} more details. To ensure the convergence of the SGD under $\iid$ sampling, it is assumed that the norm of the gradient $\norm{f(\theta_k)}_2$ grows at most linearly with $\norm{\theta_k}_2$; see, e.g., \citet[Assumption 4.3]{polyak1992acceleration}. In the case of state-dependent Markovian data, even if $\norm{\nabla f(\theta)}_2=O(\norm{\theta}_2)$, the mean of the gradient-noise can have a super-quadratic dependence on $\theta$ leading to the divergence of SGD dynamics. Furthermore, state-dependent Markovian data arises in strategic classification problems which are inherently adversarial in nature. In such problems, the dependence of the transition kernel on the state allows for data sequence such that the iterate is pushed away whenever it is within a compact set containing the true parameter. In order to ensure stability, using truncations as in Algorithm~\ref{alg:tsgd}, was proposed by \cite{chen1987continuous,andrieu2005stability,liang2010trajectory}. 

However, under the regularity conditions stated  Assumption~\ref{as:liangasA3}, we also have the following result by \cite{liang2010trajectory}, showing that there exists a finite $\sigma_s\geq 1$ such that for $k\geq \sigma_s$, the truncation step is not necessary. Without loss of generality let $\sigma_s=1$. Then, for $k\geq \sigma_s$, Algorithm~\ref{alg:tsgd} becomes vanilla SGD.
\begin{lemma}[\cite{liang2010trajectory}]\label{lm:finitetrunc}
    Let Assumption~\ref{as:strongcon}, Assumption~\ref{as:liangasA2}, Assumption~\ref{as:liangasA3} and Assumption~\ref{as:liangasA4} be true. Let the number of truncations be denoted by $\mathcal{T}$. Then, there is a finite positive integer $\sigma_s$ such that almost surely: (i) no truncation is necessary after $\mathcal{T}=\sigma_s$, and (ii)  the iterates remain in a compact set for $k\geq \sigma_s$.
\end{lemma}
According to \citet[Theorem 5.4]{andrieu2005stability}, the probability of the event $\mathcal{T}>k$ decays exponentially, i.e., there exist constants $\omega\in(0,1)$, and $C>0$ such that $P(\mathcal{T}>k)\leq C\omega^k$. In our experiments, the algorithms did not require any truncation at all.
Even if truncation is necessary, after finitely many truncations, the updates are guaranteed to remain contained in a compact set by Lemma~\ref{lm:finitetrunc}, and then the updates of Algorithm~\ref{alg:tsgd} are equivalent to vanilla SGD. 
\subsection{Motivating Application}\label{sec:motivation}
Data sampled according to a state-dependent Markov chain is frequently encountered in applications such as reinforcement learning~\citep{bartlett1992learning,goldberg2013adaptive,karimi2019non,qu2020finite,li2023statistical}, algorithmic versions of adaptive inference~\citep{zhang2021statistical,khamaru2021near}, and strategic classification and performative prediction~\citep{cai2015optimum,hardt2016strategic,perdomo2020performative, mendler2020stochastic,li2022state}. We describe one such application in detail now. 

Consider the classification problem where a bank (learner) is trying to decide the eligibility of a client (agent) for a potential loan. If the features used by this classifier are made public, the clients try to adapt their features to increase the chance of being eligible for the loan \citep{li2022state}. Even if the classifier is not public, the population features, e.g. credit score, are susceptible to change by the classifier decisions \citep{drusvyatskiy2023stochastic}. A similar scenario arises in spam email filtering as well. On learning the learner's classifier information, agents often use iterative algorithms such as gradient ascent to learn the optimal perturbed feature to maximize the probability of getting classified in the target class \citep{li2022state}. From the learner's perspective, such a data sequence can be modeled as a Markov chain where the transition probability depends on the current classifier parameter $\theta_k$. Here, the goal is to learn the \textit{equilibrium point} $\theta^*$ which minimizes the expected loss when the data is being sampled from the stationary distribution $\pi_{\theta^*}$ of the Markov chain whose transition probability $P_{\theta^*}$ corresponds to $\theta^*$. In the performative prediction literature, $\theta^*$ is referred to as \textit{performatively stable} point \citep{li2022state}.

Let the classifier be $h(u,\theta)=u^\top\theta$ where $u\in\mathbb{R}^d$ is the feature and $\theta$ is the parameter to be optimized. Let the loss function be logistic loss which for a sample $(u,y)$, where $y\in\{-1,1\}$ denotes the corresponding class, is given by,
\begin{align*}
	L(\theta;x)=\log\left(1+\exp\left(-yh(u,\theta)\right)\right),
\end{align*} 	
where $x\coloneqq (u,y)$. We use $u_S$, and $u_{-S}$ to denote the parts of feature $u$ which are respectively strategically modifiable, and non-modifiable by the agents. In the bank loan example, modifiable features could be Revolving Utilization, Number of Open Credit Lines, and Number of Real Estate Loans or Lines. Then the modified feature (the best response) $u_S'$ reported by  the agent is the solution to the following  optimization problem.
\begin{align*}
	u_S'=\argmax_{u_S} \left(h(u;\theta)-c(u_S,u_S')\right), 
\end{align*}
where $c(u_S,u_S')$ is the cost of modifying $u_S$ to $u_S'$. Let the agents iteratively learn $u_S'$ similar to \cite{li2022state}. When the agents learn the best response $u_S'$ using some iterative optimization algorithm such as Gradient Ascent then, at every iteration $k$, a set $\cI_k$ of  $n_1\leq M_1$ randomly chosen agents out of $M_1$ agents modify their features as
\begin{align}\label{eq:feturemod}
	u_{S,i}^k=\begin{cases}
		u_{S,i}^{k-1}+\alpha\left(\nabla h(u_{S,i}^{k-1};\theta_{k})-\nabla c(u_{S,i}^{k-1},u_{S,i}^0)\right) & i\in \cI_k\\
		u_{S,i}^{k-1}& i\notin \cI_k,
	\end{cases}
\end{align} 
where $\alpha>0$ is the stepsize. With a little abuse of notation, we use $\nabla h(u_{S,i}^{k-1};\theta)$  in \eqref{eq:feturemod} to denote  the fact that the gradient is with respect to $u_{S,i}^{k-1}$ while $u_{-S,i}$ remains unchanged. This introduces the state-dependent Markov chain dynamics in the training data for the bank. One can readily see that this fits in the general framework of optimization under Markovian sampling as discussed in Section~\ref{sec:sgdmarkov}.

\subsection{Our Contributions}
In this section, we summarize our main contributions. We study the overlapping batch-means-based online covariance estimator to estimate the limiting covariance of SGD under Markovian sampling. 
We summarize our main theoretical result informally below. 

\noindent \textbf{Informal Statement.} \textit{Let $f(\theta)$ be strongly-convex, has Lipcshitz continuous gradient, and the data sequence $\{x_k\}_k$ is sampled from a $\theta_k$-dependent Markov chain. For the covariance estimator: in~\eqref{eq:covestimator},
\begin{enumerate}
    \item[(i)] if the chain satisfies standard regularity conditions, and a solution to the Poisson equation of the Markov chain exists (Assumption~\ref{as:liangasA3}) , then we have
    \begin{align*}
    \expec{\norm{\hat{\Sigma}_n-\Sigma}_2}{}
    \lesssim\, \sqrt{d}\,n^{-\frac{1}{8}}(\log n)^\frac14.
\end{align*}
    \item[(ii)] if the transition probability of the Markov chain does not depend on $\theta_k$, and the chain is $V$-uniformly mixing, then we have
\begin{align*}
    \expec{\norm{\hat{\Sigma}_n-\Sigma}_2}{}
    \lesssim\, \sqrt{d}\,n^{-\frac{1}{8}+\iota}, \qquad \forall\,\iota>0.
\end{align*}
\end{enumerate}} 

This provides an answer to the question posed in Section~\ref{sec:intro}. Specifically, it shows that, ignoring logarithmic factors the overlapping batch-means estimator $\hat{\Sigma}_n$ with the Markovian data-stream indeed has the same rate of convergence as $\iid$ data-stream. Establishing this result is far from trivial. In particular, our proof techniques differs from that of the $\iid$ case broadly in three ways. 
\begin{itemize}
    \item State-dependent Makrovian sampling leads to a different noise decomposition \eqref{eq:decomp} compared to the $\iid$ case, which leads to extra terms in error analysis. Specifically, the terms $\mathsf{A}_2$ and $\mathsf{A}_3$ in \eqref{eq:A1A2A3def}, the terms $\mathsf{V}$ and $\mathsf{VI}$ in \eqref{eq:IVVVIdef}, and the terms $\mathsf{K}_2$ and $\mathsf{K}_3$ in \eqref{eq:thetaithetakintermed} don't appear in $\iid$ data setting. We show, with explicit rates, that these error terms vanish. 
    \item State-dependent Markovian sampling also leads to non-trivial  dependencies between the blocks of the covariance estimator in~\eqref{eq:covestimator}, which are absent in the $\iid$ case. In particular, intricate analysis is needed to tackle the fourth-order terms which do not vanish unlike $\iid$ case. It involves constructing some auxiliary sequence of data with same asymptotic covariance and showing that these sequence is \textit{close} to the original data sequence. See \textit{point 2} in the proof outline of Theorem~\ref{th:mainthm} provided in Section~\ref{sec:proofoutline}. 
    \item Lastly, we show that $\expec{\norm{\theta_{k+1}-\theta^*}_2^p}{}=O(\eta_{k+1}^{p/2})$, $p=1,2,4$ (Lemma~\ref{lm:expecconviter}) which is required to show that the $\hat{\Sigma}_n$ defined on a linear approximation of the original SGD updates is indeed close to $\hat{\Sigma}_n$ defined on the original updates. The bound on $\expec{\norm{\theta_{k+1}-\theta^*}_2^4}{}$ for SGD with state-dependent Markovian data could be of independent interest. We establish this result using a novel technique which involves defining an auxiliary sequence of updates \eqref{eq:perturbtheta} (different from the one mentioned in the previous paragraph), establishing the expected convergence for this auxiliary updates, and showing that this sequence of updates is indeed \textit{close} to the original one \eqref{eq:sgd}.
\end{itemize}  

We illustrate the performance of the estimator on real and synthetic datasets for linear and logistic regression. We show that the convergence rate of the estimation error of the covariance matrix agrees with our theoretical result. Furthermore, the convergence plots for different dimensions seem to suggest that the dimension dependence is indeed polynomial. We show that the online confidence interval constructed using the covariance estimator $\hat{\Sigma}_n$ achieves the correct coverage probability for the one-dimensional projection $\boldsymbol{1}^\top\theta^*$ of the unknown parameter $\theta^*$, where $\boldsymbol{1}$ is a $d$ dimensional vector with all entries equal to $1$. 
\subsection{Prior works on inference for SGD}\label{sec:RW}

Following \cite{polyak1992acceleration}, under $\iid$ sampling, several works have established CLTs for variants of SGD under strong-convexity with decreasing step-sizes; see, for example,~\cite{ toulis2017asymptotic,asi2019stochastic,duchi2021asymptotic}.  Non-asymptotic rates for SGD CLTs were derived  in~\cite{anastasiou2019normal,shao2022berry}. Furthermore,~\cite{dieuleveut2020bridging} and~\cite{yu2020analysis} established asymptotic normality of constant step-size SGD in the convex and nonconvex setting respectively. More recently,~\cite{davis2023asymptotic} extended the results of~\cite{polyak1992acceleration} to certain non-smooth settings. While the above works focus on the low-dimensional setting, recently~\cite{agrawalla2023high} established high-dimensional CLTs in the context of linear regression and developed related inferential procedures.

Motivated by viewing SGD iterates as Markov chains, several works have studied estimating the limiting covariance matrix in the SGD CLT under $\iid$ sampling. 
A partially offline inference procedure based on estimating the minimum eigenvalue of the asymptotic covariance matrix is proposed by \cite{chee2022plus}. The plug-in estimator-based inference procedure for SGD has been discussed in the context of federated learning by \cite{li2022statistical}. \cite{su2023higrad} proposed a tree-based inference procedure where the tree is constructed to exploit the asymptotic independence between multiple threads of SGD. Inference for implicit SGD was studied by \cite{liang2019statistical}. Batch-means covariance estimator for zeroth-order SGD was studied in~\cite{jin2021statistical,chen2021online}. Inference for SGD applied to specific problems has been studied by \cite{chen2021statistical}, and \cite{shi2021statistical}. Online multiplier bootstrap procedures for inference of SGD estimator under $\iid$ data have been studied by \cite{fang2018online,roy2023fairness}. Recently, \cite{chen2023recursive} developed non-asymptotic confidence bounds for SGD applied to the specific problem of quantile estimation.

CLTs for SGD iterates have been established under various non-$\iid$ data sampling settings, e.g., stationary strongly mixing data \citep{solo1982stochastic}, and state-dependent Markovian data \citep{liang2010trajectory,fort2015central}. We know discuss works focusing on inference for SGD under non-$\iid$ sampling, which is the problem we focus on in this work.  \cite{ramprasad2022online} proposed a multiplier-bootstrap based approach for inference of SGD with uniformly-ergodic Markov chain data where $\nabla f(\theta)$ is a linear function of $\theta$. In comparison, our focus is on the case of state-dependent Markovian sampling with smooth and general strongly-convex functions. Moreover, bootstrap-based approach involves generating a large number of perturbed gradients at each iteration, especially in high dimensional problems, leading to higher per-iteration computational costs. \cite{li2023online} proposed an asymptotic covariance estimator similar to sample-covariance for uniformly-mixing Markov chain data, albeit with qualitative and asymptotic justification, for scale-invariant functions of one-dimensional projection of the averaged iterates. In contrast, we construct online confidence interval for $\theta^*$ by estimating the asymptotic covariance with explicit convergence rates. \cite{liu2023statistical} studied a multiplier-bootstrap based estimator for stationary, polynomially $\phi$-mixing data. Unlike our work, the method proposed in \cite{liu2023statistical} is dependent on mini-batch SGD which is incompatible with a fully-online setting where data sample arrives in a streaming manner. Moreover, compared to both \cite{li2023online} and \cite{liu2023statistical}, we assume state-dependent Markovian sampling where the data is not necessarily uniformly-mixing or $\phi$-mixing. In \citep{khamaru2021near}, the authors propose an inference procedure with debiasing for the \textit{true parameter} of the linear regression under adaptive data sampling, in a model-based setup. In contrast, our goal in this work is to develop inference procedure for the equilibrium point which is motivated by the problem of performative prediction, as discussed in Section~\ref{sec:motivation}.

\section{Main Results}
We now introduce our assumptions on the optimization problem~\eqref{eq:mainprob}. We refer to~\cite{douc2018markov} for a textbook introduction to additional details regarding several assumptions below. Let $\cF_k$ be the filtration generated by $\{\theta_0,\cdots,\theta_k,x_1,\cdots,x_k\}$. For any mapping $g:\mathbb{R}^d\to\mathbb{R}^d$ define the norm with respect to a function $V:\mathbb{R}^d\to [1,\infty)$ as
\begin{align*}
    \norm{g}_{V}=\underset{x\in\mathbb{R}^d}{\sup}\frac{\norm{g(x)}_2}{V(x)},
\end{align*}
and let $L_{V}=\{g:\mathbb{R}^d\to\mathbb{R}^d,\norm{g}_{V}<\infty\}$. 

We make the following regularity assumption on the objective function for our analysis. 
\begin{assumption}\label{as:strongcon} The objective function $f(\theta)$ is continuously differentiable and $\mu$-strongly convex where $\mu>0$, i.e., for any $\theta_1,\theta_2$
\begin{align*}
    f(\theta_2)\geq f(\theta_1)+\nabla f(\theta_1)^\top (\theta_2-\theta_1)+\frac{\mu}{2}\|\theta_2-\theta_1\|_2^2.
\end{align*}
\end{assumption}
\begin{assumption}\label{as:liangasA2}
There exists a positive-definite matrix $\tilde{Q}$, $r>0$, and a constant $c$ such that, 
\begin{align*}
    \norm{\nabla f(\theta)-\tilde Q(\theta-\theta^*)}_2\leq c\|\theta-\theta^*\|_2^{2} \quad \forall\theta \text{ such that }\norm{\theta-\theta^*}_2\leq r.
\end{align*}
\end{assumption}
Note that if a strongly-convex function $f(\theta)$ is thrice continuously differentiable in $\norm{\theta-\theta^*}\leq r$ for some $r>0$ then Assumption~\ref{as:liangasA2} is true for $f(\theta)$.
These are standard assumptions in stochastic approximation literature (see \cite{harold1997stochastic} for example) satisfied by important applications like linear regression where the predictor variable has a non-degenerate covariance matrix, and  $L_2$-regularized logistic regression. 
\begin{assumption}\label{as:liangasA3}
Let $\{x_k\}_k$ be a Markov chain controlled by $\theta$, i.e., there exists a transition probability kernel $P_\theta(\cdot,\cdot)$ such that
\begin{align*}
    \mathbb{P}(x_{k+1}\in B|\theta_0,x_0,\cdots,\theta_k,x_k)=P_{\theta_k}(x_k,B),
\end{align*}
almost surely for any Borel-measurable set $B\subseteq\mathbb{R}^d$ for $k\geq 0$. For any $\theta\in\Theta$, $P_\theta$ is irreducible and aperiodic. Additionally, there exists a function $V:\mathbb{R}^d\to [1,\infty)$ and a constant $\alpha_0\geq 8$ such that for any compact set $\Theta'\subset\Theta$, we have the following. For a function $g$, let $P_\theta g(x) \coloneqq \int P_\theta (x,y) g(y)\, \mathrm{d}y $.
\begin{enumerate}[label=(\alph*)]
    \item There exist a set $C\subset \mathbb{R}^d$, an integer $l$, constants $0<\lambda<1$, $b$, $\kappa$, $\delta>0$, and a probability measure $\nu$ such that,
    \begin{align*}
        \sup_{\theta\in\Theta'}P_\theta^lV^{\alpha_0}(x)&\leq \lambda V^{\alpha_0}(x)+bI(x\in C)\quad \forall x\in\mathbb{R}^d,\\
        \sup_{\theta\in\Theta'}P_\theta V^{\alpha_0}(x)&\leq\kappa V^{\alpha_0}(x)\quad \forall x\in\mathbb{R}^d,\\
        \inf_{\theta\in\Theta'}P_\theta^l(x,A)&\geq \delta\nu(A) \quad \forall x\in C, \forall A\in \mathcal{B}_{\mathbb{R}^d}.
    \end{align*}
    where $\mathcal{B}_{\mathbb{R}^d}$ is the Borel $\sigma$-algebra over $\mathbb{R}^d$.
    \item There exists a constant $c>0$, such that, for all $x\in\mathbb{R}^d$, and $\theta,\theta'\in\Theta'$
    \begin{align}\label{eq:noisylipgrad}
        \sup_{\theta\in\Theta'}\norm{\nabla F(\theta,x)}_V&\leq c,\\
        \norm{\nabla F(\theta,x)-\nabla F(\theta',x)}_V&\leq c\norm{\theta-\theta'}_2.
    \end{align}
    \item There exists a constant $c>0$, such that, for all $(\theta,\theta')\in\Theta'\times\Theta'$,
    \begin{align*}
        \norm{P_\theta g-P_{\theta'} g}_V&\leq c\norm{g}_V\norm{\theta-\theta'}_2\quad \forall g\in L_V\\
         \norm{P_\theta g-P_{\theta'} g}_{V^{\alpha_0}}&\leq c\norm{g}_{V^{\alpha_0}}\norm{\theta-\theta'}_2\quad \forall g\in L_{V^{\alpha_0}}. 
    \end{align*}
\end{enumerate}
\end{assumption}
\begin{remark}[On Assumption~\ref{as:liangasA3}]
    Condition (a) of Assumption~\ref{as:liangasA3} is the so-called drift condition and $V$ is the drift function widely used in the Markov chain literature \citep{meyn2012markov}. The drift condition implies that for each fixed $\theta\in\Theta$, the data sequence is $V^{\alpha_0}$-uniformly ergodic, i.e., for a fixed $\theta$, there exists constants $0<\rho_\theta<1$ and $C_\theta>0$ such that for any positive integer $k$, and function $g\in L_V$, we have, 
    \begin{align*}
        \norm{P_\theta^kg-\pi_\theta g}_{V^{\alpha_0}}\leq C_\theta\,\rho_\theta^k\,\norm{g}_{V^{\alpha_0}}.
    \end{align*}
    When $\nabla F(\theta,x)$ is bounded, one can choose $V(x)=1$. Condition (b) of Assumption~\ref{as:liangasA3} implies that in the compact set $\Theta'$, $\norm{\nabla F(\theta,x)}_2$ is $O(V(x))$. \eqref{eq:noisylipgrad} implies that for $\theta\in\Theta'$, $f(\theta)$ has Lipschitz continuous gradient. Condition (c) controls the change of transition kernel by imposing a Lipszhitz property on $P_\theta$ wr.t. $\theta$. 
    
    The main implication of Assumption~\ref{as:liangasA3} is that it ensures the existence and regularity of a solution $u(\theta,x)$ to Poisson equation of the transition kernel $P_\theta$ given by $u(\theta,x)-P_{\theta}u(\theta,x)=\nabla F(\theta,x)-\nabla f(\theta)$. Existence of $u(\theta,x)$ has been fundamental in the analysis of additive functionals of Markov chain (see \cite{andrieu2005stability,meyn2012markov,douc2018markov} for details).  has been verified for numerous applications \citep{liang2010trajectory,karimi2019non,wu2020finite,li2022state}.
    Assumption~\ref{as:liangasA3} on state-dependent Markovian data holds for numerous applications, e.g., strategic classification with adaptive best response \citep{li2022state}; policy-gradient \citep{karimi2019non}, and actor-critic algorithm in reinforcement learning \citep{wu2020finite}; MLE with missing data \citep{liang2010trajectory}. In the example presented in Section~\ref{sec:motivation}, when the agents have a quadratic cost function, Assumption~\ref{as:liangasA3} holds as shown by \citet[Appendix D]{li2022state}.
\end{remark}
We need the following assumption on the stepsizes in Algorithm~\ref{alg:tsgd}, similar to \cite{liang2010trajectory}.
\begin{assumption}\label{as:liangasA4}
The sequence $\{\eta_k\}_k$ and $\{d_k\}_k$ are decreasing, positive and satisfy
\begin{align*}
      \textstyle\sum_{k=1}^\infty\eta_k=\infty,&\quad \frac{\eta_{k+1}-\eta_k}{\eta_k}=o(\eta_{k+1}),\quad d_k=O\left(\eta_k^{(1+\tau)/2}\right)\quad \text{for some}~\tau\in(0,1],\\
    & \textstyle \sum_{k=1}^\infty\frac{\eta_{k}^{(1+\tau)/2}}{\sqrt{k}}<\infty, \quad\text{and}\quad
    \sum_{k=1}^\infty\left(\eta_kd_k+(\eta_k/d_k)^{\alpha_0}\right)<\infty,
\end{align*}
for some constant $\alpha_0\geq 8$ introduced in Assumption~\ref{as:liangasA3}.
\end{assumption}
\begin{remark}\label{rm:step} 
    The stepsizes are chosen such that it is not too small such that $\nabla f(\theta)$ controls the dynamics and it is not too large such that the dynamics are confined to a compact set. The reinitializations serve as a drastic drift towards $\theta^*$ such that if $\theta_k$ changes too rapidly then the algorithm can restart with smaller stepsize which essentially reduce  the effective variance in the gradient noise term in~\eqref{eq:gradnoise} \citep{andrieu2005stability}. One can choose, $\eta_k=Ck^{-a}$, and $d_k=Ck^{-b}$ where $a>1/2$, $b< 3/8$.
\end{remark}
We have the following bound on the estimation error of $\Sigma$ in terms of the $\norm{\cdot}_2$ norm. We provide the proof of the following theorem in the supplementary material. 
\begin{theorem}[State-dependent Markovian Data]\label{th:mainthm}
Let Assumption~\ref{as:strongcon}-\ref{as:liangasA4} be true. Then, we have\footnote{We use the notation $a\lesssim b$ to denote $a\leq Cb$ where $C>0$ is a constant.}
\begin{align*}
    \mathbb{E}[\|\hat{\Sigma}_n-\Sigma\|_2]
    \lesssim  n^{-\frac{a}{4}}\sqrt{d\log n}
    +d\,n^{1/\beta-(1-a)}
    +\sqrt{d}\,n^{-\frac{1}{2\beta}}+ \sqrt{d}\,n^{(1/\beta-(1-a))/2},
\end{align*}
by setting $a_m=O(\floor{m^\beta})$, where $m=1,2,\cdots,M$, and $a_M\leq n<a_{M+1}$.
\end{theorem}
\begin{remark}
When $d\leq n^\frac{1-a}{2}$, choosing $\beta=2/(1-a)$, we obtain the convergence rate of 
$$O\big(n^{-\frac{a}{4}}\sqrt{d\log n}+\sqrt{d}\,n^{-\frac{1-a}{4}}\big).$$ 
Choosing $a=1/2+4\iota$ for arbitrarily small $\iota>0$ we obtain the rate $$O\big(\sqrt{d}\,\max(n^{-\frac{1}{8}+\iota},n^{-\frac{1}{8}-\iota}(\log n)^{1/4})\big).$$
Note that, ignoring the logarithmic factor, this rate matches the best known rate in the $\iid$ case \citep{zhu2021online}. 
\end{remark}
We now consider the case when the data samples are generated from a Markov chain where the transition kernel does not depend on $\theta_k$. In this case, we assume the following simplified version of Assumption~\ref{as:liangasA3}.
\begin{assumption}\label{as:liangasA3si}
Let conditions (a) and (b) of Assumption~\ref{as:liangasA3} be true with $P_\theta$ replaced by $P$, i.e., the transition kernel $P$ does not depend on the state $\theta$. 
\end{assumption}
We provide the proof of the following theorem in the supplementary material. 
\begin{theorem}[State-independent Markovian Data]\label{th:simain}
Let Assumption~\ref{as:strongcon}, \ref{as:liangasA2}, \ref{as:liangasA4}, and \ref{as:liangasA3si} be true. Then, we have
\begin{align*}
    \mathbb{E}[\|\hat{\Sigma}_n-\Sigma\|_2]
    \lesssim &\, d\,n^{1/\beta-(1-a)}
    +\sqrt{d}\,n^{-\frac{1}{2\beta}}+ \sqrt{d}\,n^{(1/\beta-(1-a))/2},
\end{align*}
by setting $a_m=O(\floor{m^\beta})$, where $m=1,2,\cdots,M$, and $a_M\leq n<a_{M+1}$.
\end{theorem}
\begin{remark}
In this case, the term $n^{-\frac{a}{4}}\sqrt{d\log n}$ does not appear in the rate because of the exponential mixing of the data sequence. In fact, the rate matches with the $\iid$ setting in this case. Intuitively, this is due to the polynomial batch-sizes which allow the data-sequence to mix fast enough to be similar to the $\iid$ setting. 
\end{remark}
\input{experiments}
\input{conclusion}

%% file: experiments.tex
\section{Experiments}
In this section, we illustrate the performance of the covariance estimator in linear regression, and logistic regression problems. We focus on constructing online confidence interval for an one-dimensional projection $\theta_v=v^\top\theta^*$ of $\theta^*$, where $v=\boldsymbol{1}$, and $\boldsymbol{1}$ is a $d$ dimensional vector with all entries equal to $1$. At time point $k$, using our estimate $\hat{\Sigma}_k$ of $\Sigma$ we construct the 95\% confidence interval for $\theta_v$ as follows:
\begin{align*}
    \bigg[\bar{\theta}_k^\top \boldsymbol{1}-z_{0.975}\sqrt{(\boldsymbol{1}^\top\hat{\Sigma}_k\boldsymbol{1})/k},~\bar{\theta}_k^\top \boldsymbol{1}+z_{0.975}\sqrt{{(\boldsymbol{1}^\top\hat{\Sigma}_k\boldsymbol{1})}/{k}}\bigg],
\end{align*}
where $z_{0.975}$ is the 97.5 percentile of the standard normal distribution. To evaluate $\expec{\norm{\hat{\Sigma}_n-\Sigma}_2}{}$, and the coverage probability, we take the average of 200 repetitions. For all the experiments we choose $\eta_k=2k^{-a}$, $d_k=10k^{-0.3}$, $a=0.5005$, $\beta=2/(1-a)$, and $a_m=2m^\beta$. We set $\mathcal{K}_0$ to be the set $\{\theta:\norm{\theta}_2\leq 10\}$. Since $\Sigma$, and $\theta^*$ are not known beforehand, we run the algorithm 500 times for each experiment to estimate these quantities. In our experiments, we compare our method to the multiplier-bootstrap-based method from \cite{fang2018online,ramprasad2022online}. We emphasize here that to the best of our knowledge, there is no analysis of such bootstrap methods under state-dependent Markovian sampling. For easier reference, we will call our batch-means method \texttt{BM}, and the bootstrap-based method from \cite{ramprasad2022online} as \texttt{Boot}. 

As these two methods construct confidence intervals with different width and coverage probability, to make a meaningful comparison, we use the Mean Interval Score (MIS) \citep{gneiting2007strictly,askanazi2018comparison,wu2021quantifying}. Intuitively, lower values of MIS favors narrower intervals with higher coverage. Let $[u,l]$ be the $(1-\alpha_1)$-confidence interval for some statistic $Z$, where $u$ and $l$ denote that upper and lower confidence bounds respectively. The population MIS for $[u,l]$ is given by 
\begin{align*}
    \text{MIS}(u,l;\alpha_1)\coloneqq(u-l)+\frac{2}{\alpha_1}\left(\expec{Z-u|Z>u}{}+\expec{l-Z|Z<l}{}\right).
\end{align*}
The sample MIS is given by 
\begin{align*}
    \text{MIS}(u,l;\alpha_1,n)\coloneqq\frac{1}{n}\sum_{i=1}^n\left(u-l+\frac{2}{\alpha_1}(z_i-u)\mathbbm{1}(z_i>u)+\frac{2}{\alpha_1}(l-z_i)\mathbbm{1}(z_i<l)\right). 
\end{align*}
\subsection{State-dependent: Linear Regression}
\textbf{Synthetic Data.} We first consider linear regression with state-dependent Markovian data. Motivated by \cite{li2022state}, we generate the data stream $x_k\coloneqq(u_k,y_k)$ as,
\begin{align}\label{eq:dgplinregsynth}
    u_{k}=(1-\varrho)u_{k-1}+\varrho \upsilon_k+\varrho\varepsilon\theta_{k-1}\tilde{\upsilon}_k \quad y_k=u_k^\top \theta_r+\upsilon_k',
\end{align}
where $0<\varrho<1$, $\varepsilon\in\mathbb{R}$, $\theta_r\in\mathbb{R}_d$, and $\upsilon_k$, $\tilde{\upsilon}_k$ and $\upsilon_k'$ are $\iid$ sequences from $ N(0,\sigma^2 I_d)$, $ N(0,\sigma^2)$, and $ N(0,\sigma^2 )$ respectively. We emphasize here that $\theta^*$ is the equilibrium point of the dynamics and does not necessarily be equal to $\theta_r$. Note that Assumptions~\ref{as:strongcon}-\ref{as:liangasA2} are true for the squared loss in linear regression. Data sequence $\{x_k\}_k$, generated according to \eqref{eq:dgplinregsynth}, satisfies Assumption~\ref{as:liangasA4} for a fixed $\theta$ with $V(x)=\norm{x}_2^2+1$. For this experiment, we choose $\varrho=0.5$, $\varepsilon=0.5$, and $\sigma=1$. Total number of iterations are set to $N=50000$. 
\begin{figure}
    \centering
    \includegraphics[width=0.42\textwidth]{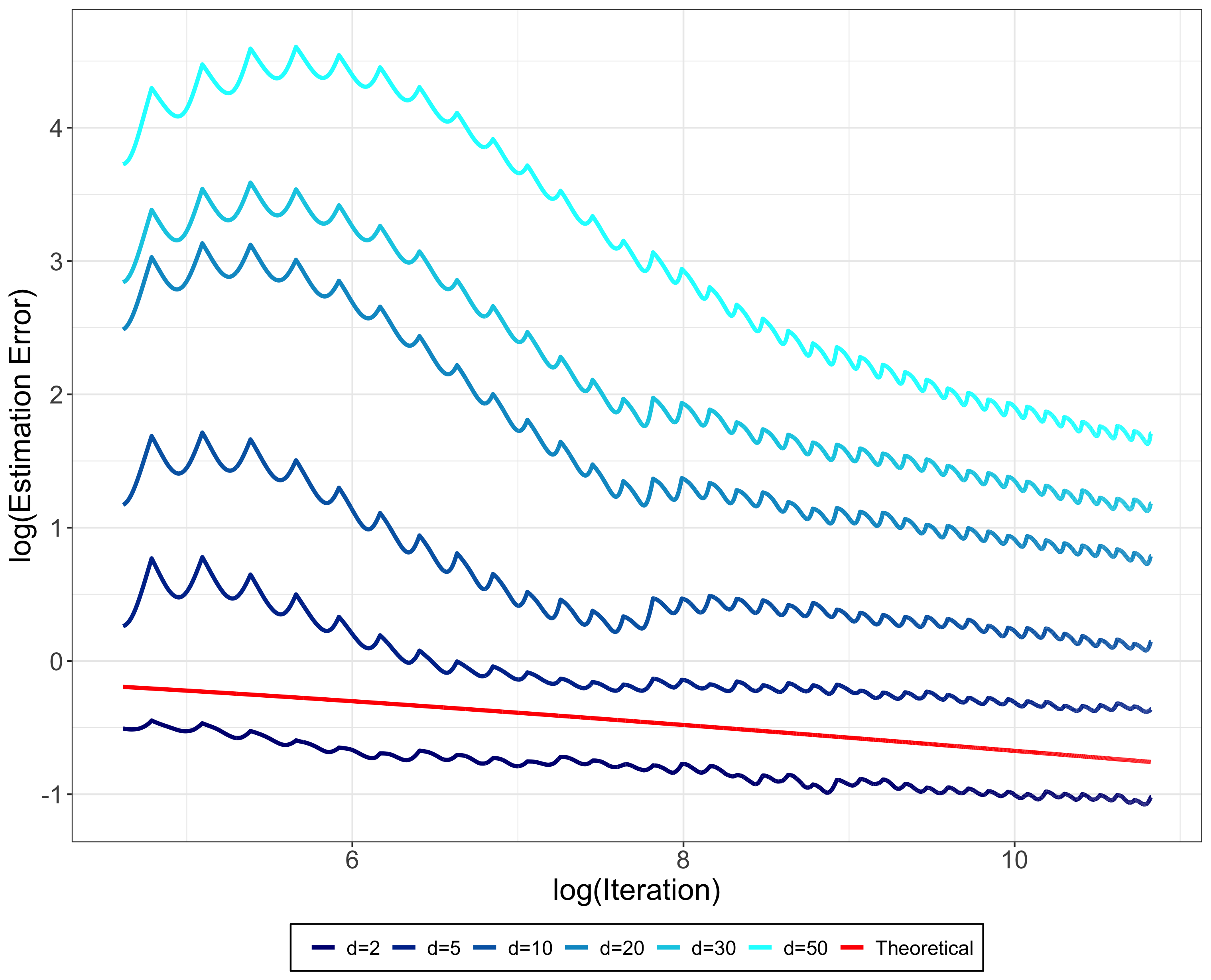}
    \includegraphics[width=0.42\textwidth]{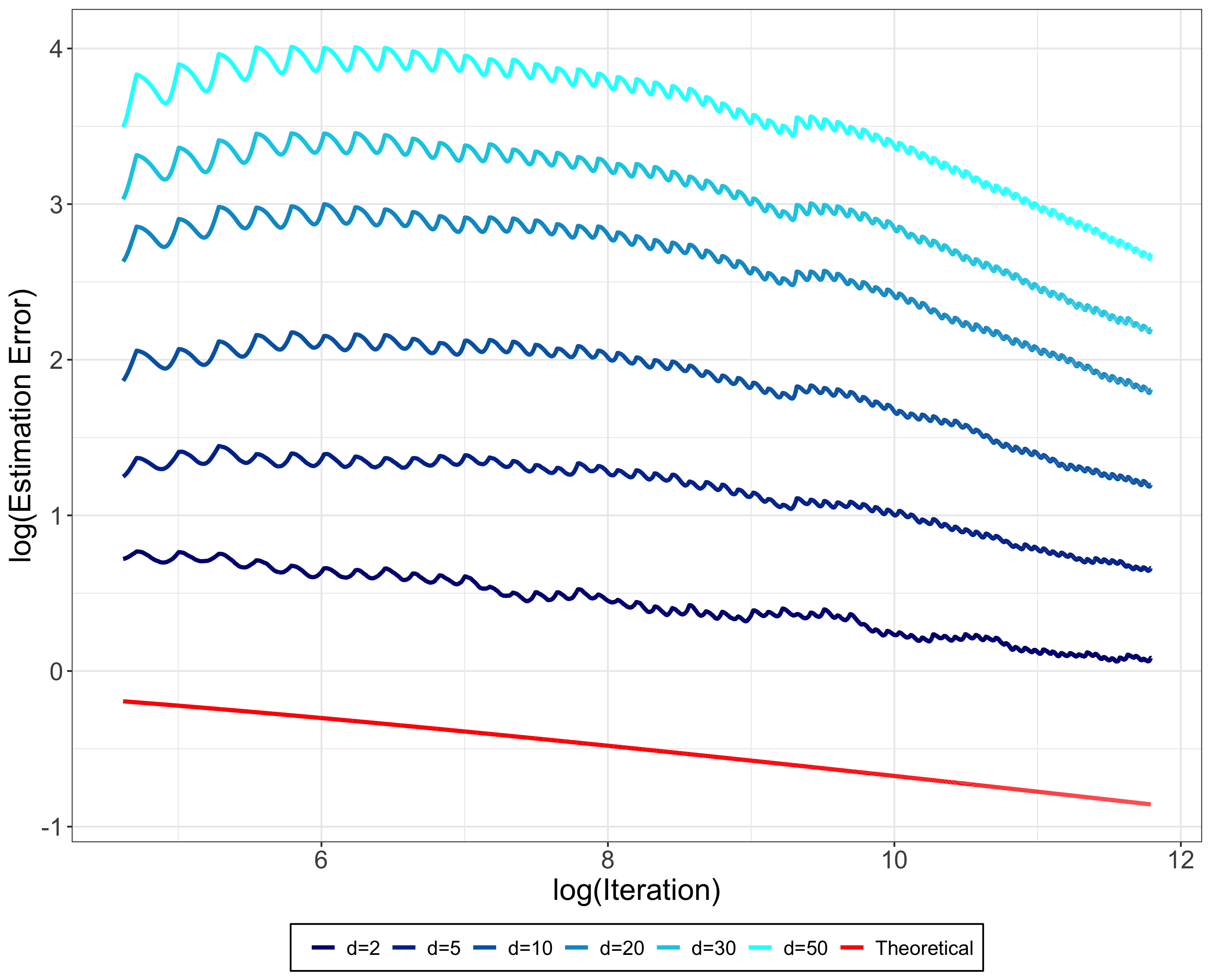}
    \\ 
    \includegraphics[width=0.42\textwidth]{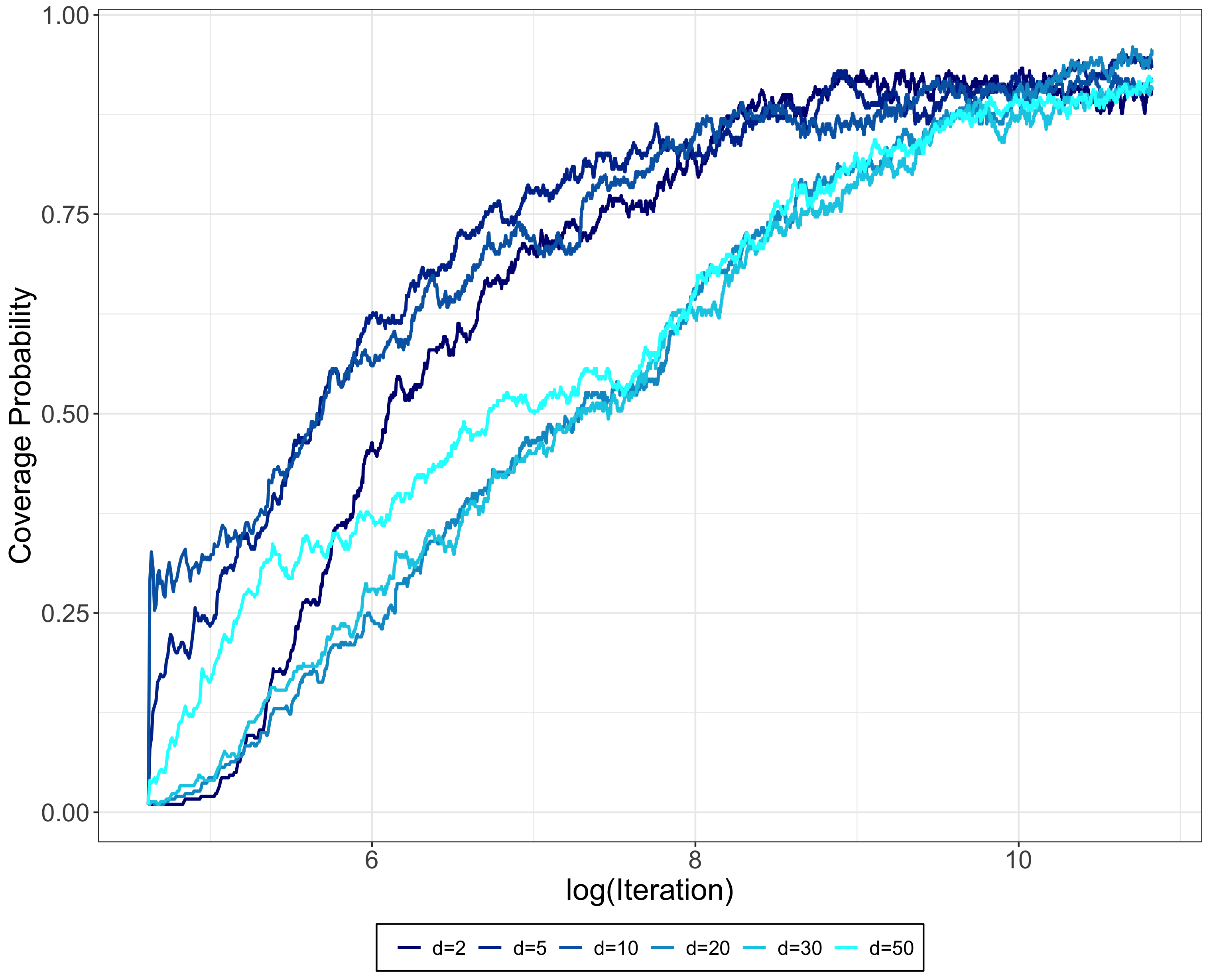}
    \includegraphics[width=0.42\textwidth]{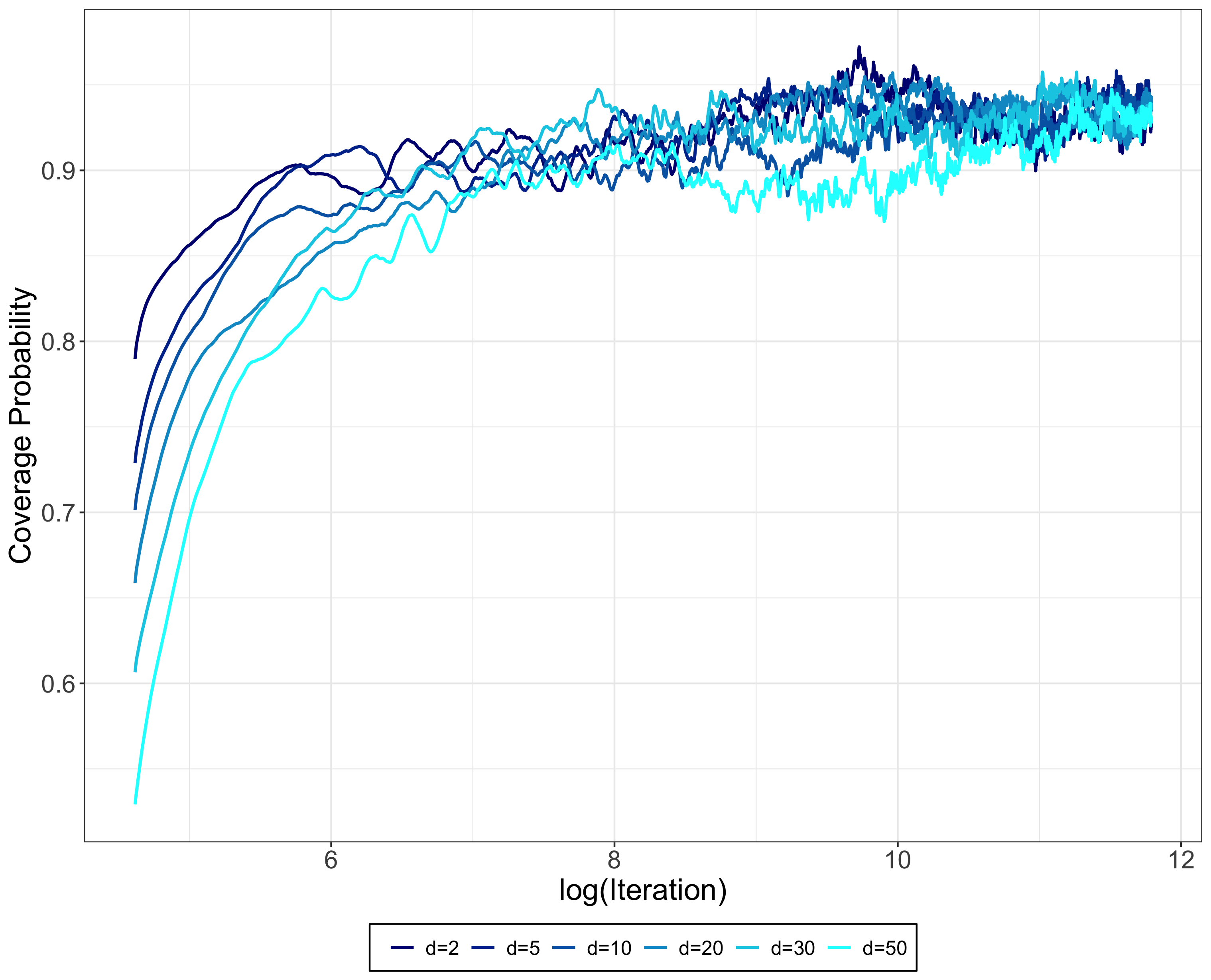}
    \\
    \includegraphics[width=0.42\textwidth]{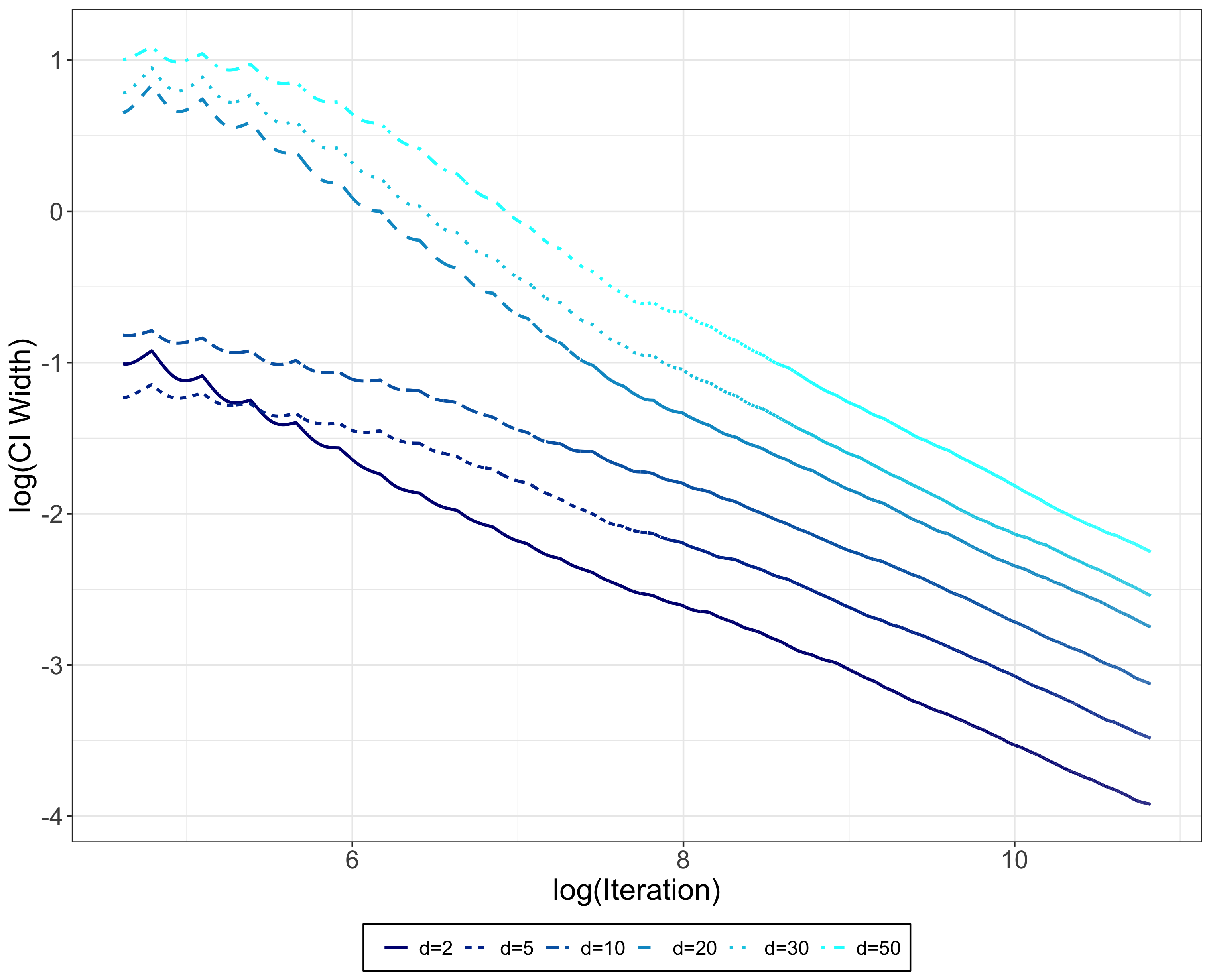}
    \includegraphics[width=0.42\textwidth]{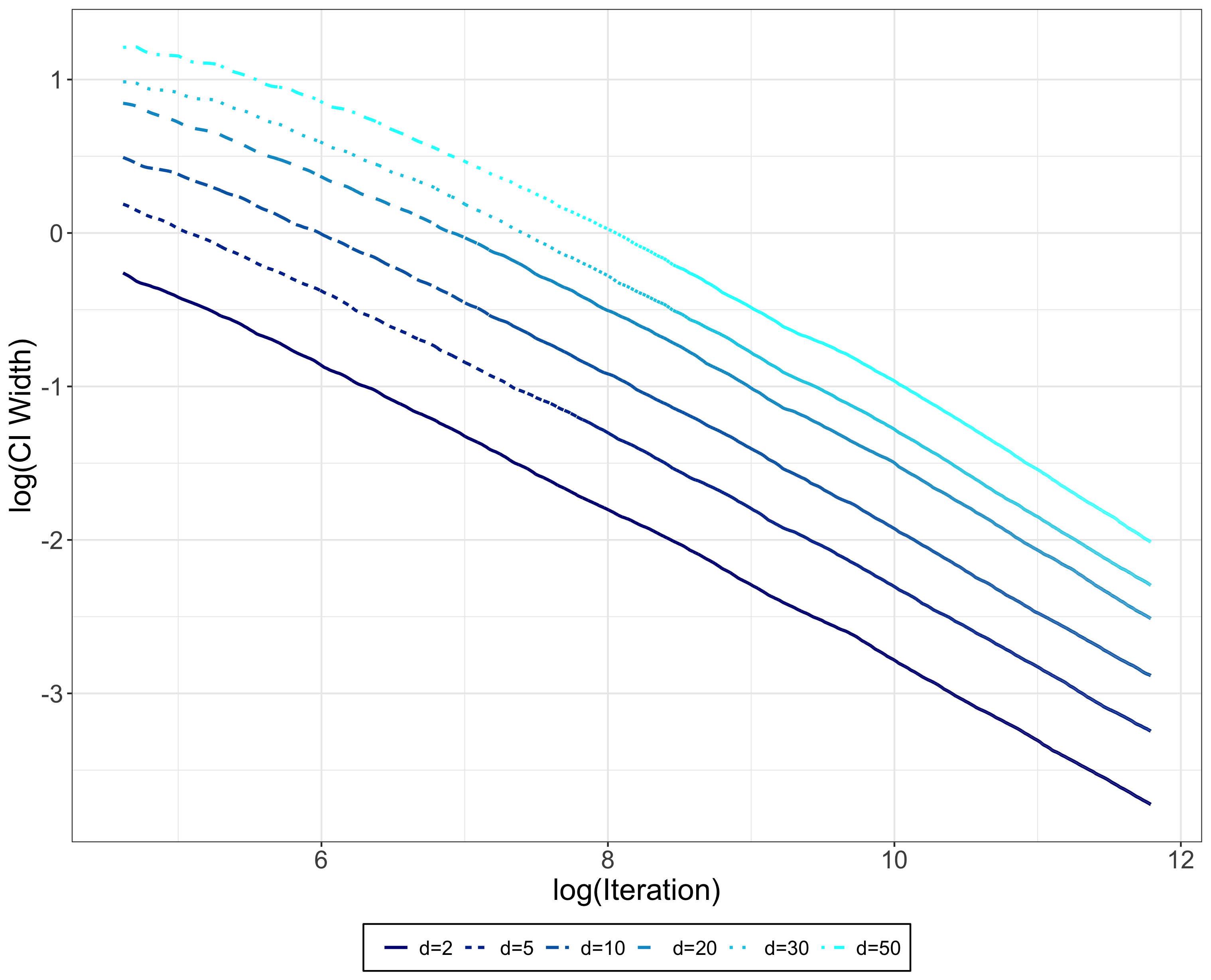}
    \caption{\it Linear regression (left) and logistic regression (right) with synthetic data: In each column, from top to bottom rows are the plots of $\log(\text{Estimation Error})$, coverage probability, and $\log(\text{width of Confidence Interval (CI)})$ respectively. The blue lines correspond to dimensions $d=2,5,10,20,30,50$. The red line in the plots of $\log(\text{Estimation Error})$ corresponds to the theoretical rate obtained in Theorem~\ref{pf:mainthm}. The results show that the slope of $\log(\text{Estimation Error})$ matches the theoretical rate.}
    \label{fig:linreg_synth}
\end{figure}
In the left column of Figure~\ref{fig:linreg_synth}, we show the estimation error $\expec{\norm{\hat{\Sigma}_k-\Sigma}_2}{}$, the coverage probability of 95\% confidence interval for $\theta_v$, and the width of the 95\% confidence interval over iterations. The decay of the estimation error is indeed almost linear in logarithmic scale and the slope matches with the theoretical rate which is approximately $-1/8$. The logarithm of estimation error at the last iteration for $d=5$ and $d=2$ differ by $0.57$ which approximately agrees with the dimension dependency of $\sqrt{d}$.

In the upper row of Figure~\ref{fig:linreg_synth_boot}, we compare \texttt{BM} (blue) with \texttt{Boot} (green). As there is no suggestion in \cite{ramprasad2022online} regarding how to choose the number of bootstrap replicates required to construct a confidence interval, we set the number of replicates to $50, 100, 200$ for $d=2,30,50$ respectively. It can be seen that \texttt{Boot} achieves a higher coverage probability at the expense of wider confidence intervals. From the plot of $\log (\text{MIS})$, one can see that \texttt{BM} is comparable with \texttt{Boot} in lower dimension. But for higher dimension $d=50$, \texttt{BM} performs better than \texttt{Boot} as the large width of the interval constructed by \texttt{Boot} negates the advantage of higher coverage probability.

\textbf{Effect of State-dependence.} We also test the performance of our estimator under various degree of dependence on the state. To do so,  we set $\varepsilon$ in \eqref{eq:dgplinregsynth} to $0.1$, $0.5$, and $0.9$. In linear regression (Fig.~\ref{fig:state_dep_bm_linreg_nuvar}), the increasing state-dependence seems to deteriorate the covariance estimation error but have no effect on the coverage probability.
\vspace{-0.07in}
\subsection{State-dependent: Logistic Regression}
\textbf{Synthetic Data.} Let $\{u_k\}_k$ be generated according to the data-generating mechanism as introduced in \eqref{eq:dgplinregsynth}, and $y_k=\text{Bernoulli}(1/(1+\exp(-u_k^\top\theta_r)))$. Since logistic loss is a strictly convex loss, we add a small regularizer $0.005\norm{\theta}_2^2$ to make it strongly convex. The results, are shown in the right column of Figure~\ref{fig:linreg_synth}. Similar to the linear regression example above, the results agree with our theoretical findings. Although, the empirical convergence  is much slower ($N=125000$) here compared to linear regression. This could be attributed to the poorer condition number of the loss function.
\begin{sidewaysfigure}
    \centering
    \captionsetup{width=.9\linewidth}
    \includegraphics[width=0.22\textwidth,height=2in]{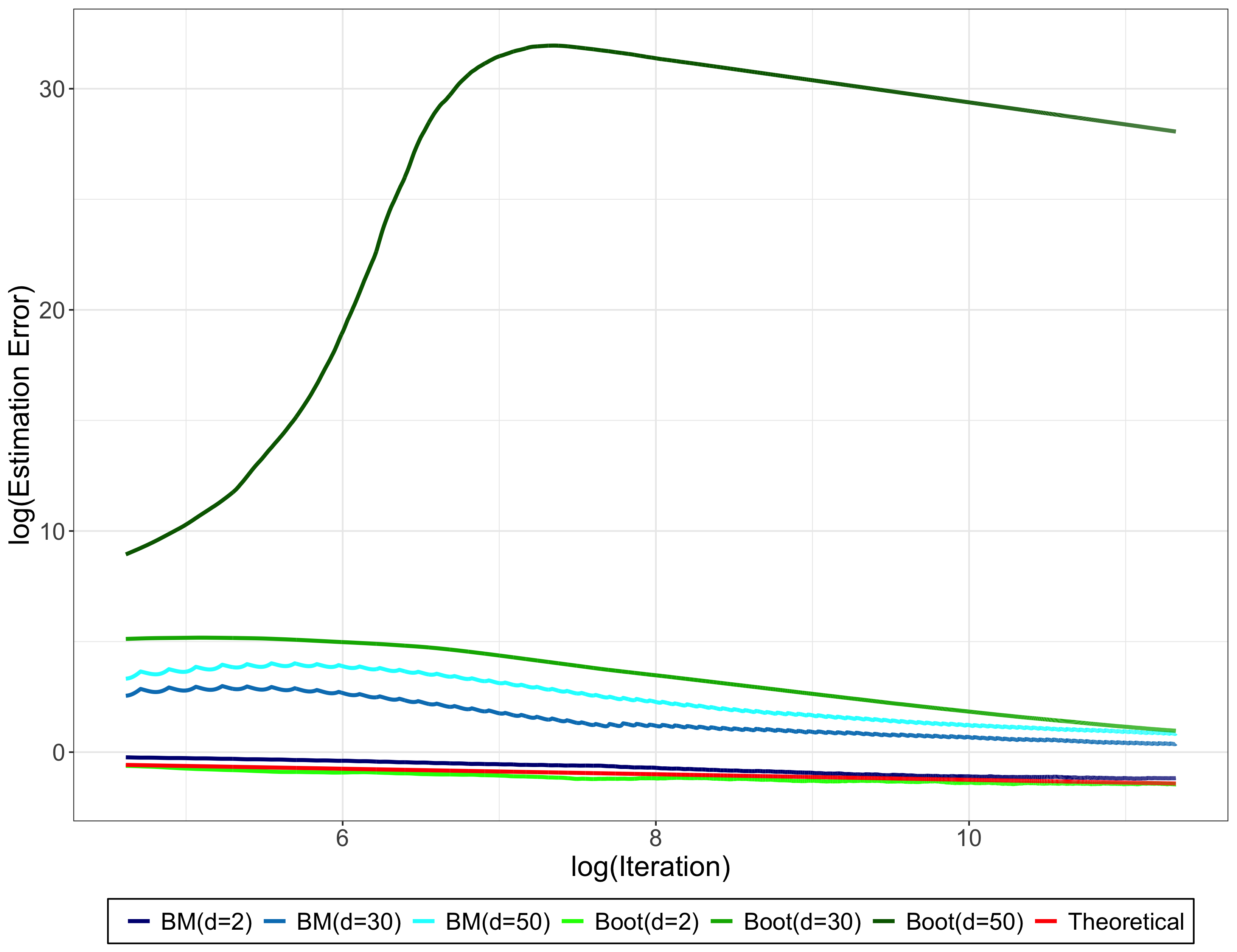}
    \includegraphics[width=0.22\textwidth,height=2in]{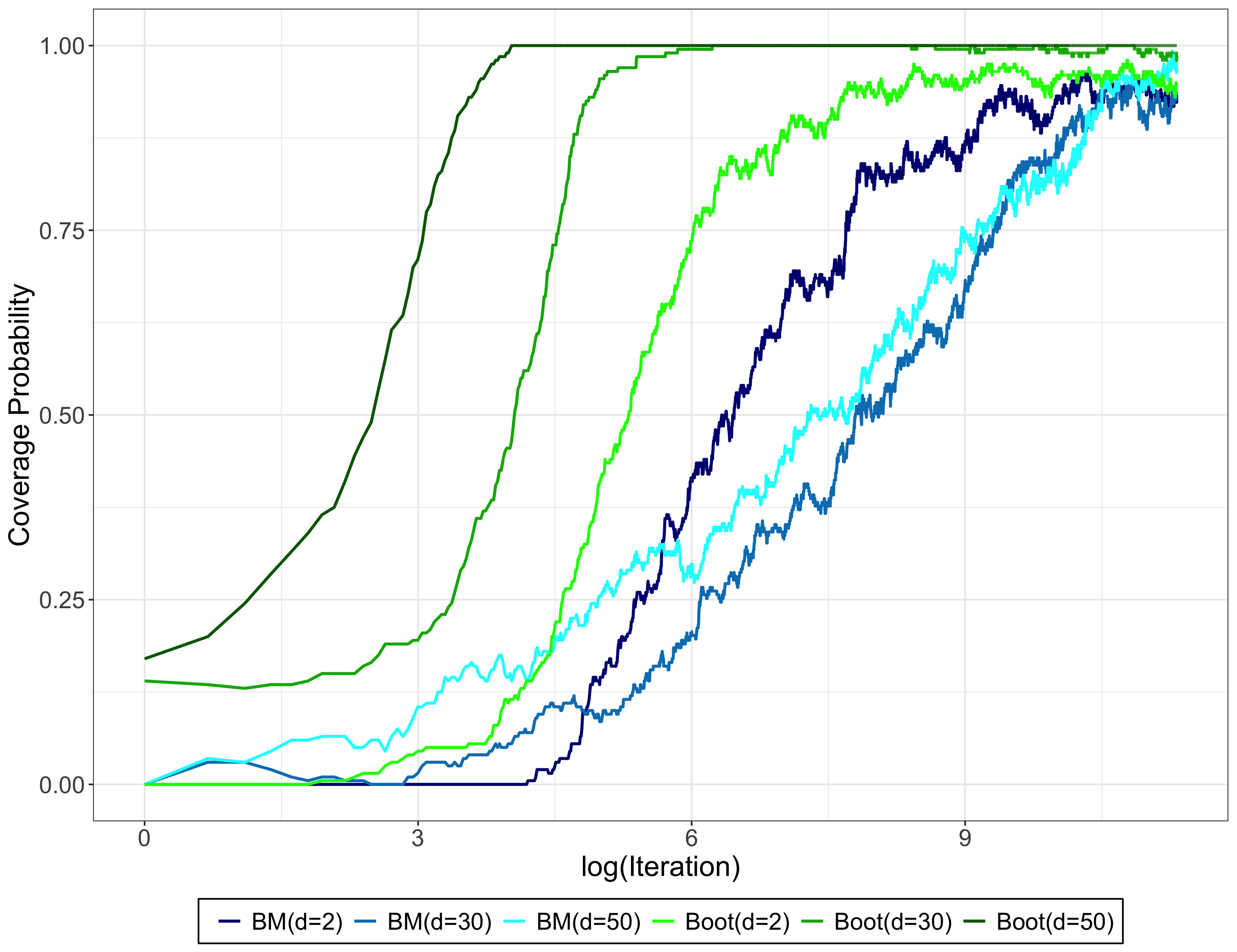}
    \includegraphics[width=0.22\textwidth,height=2in]{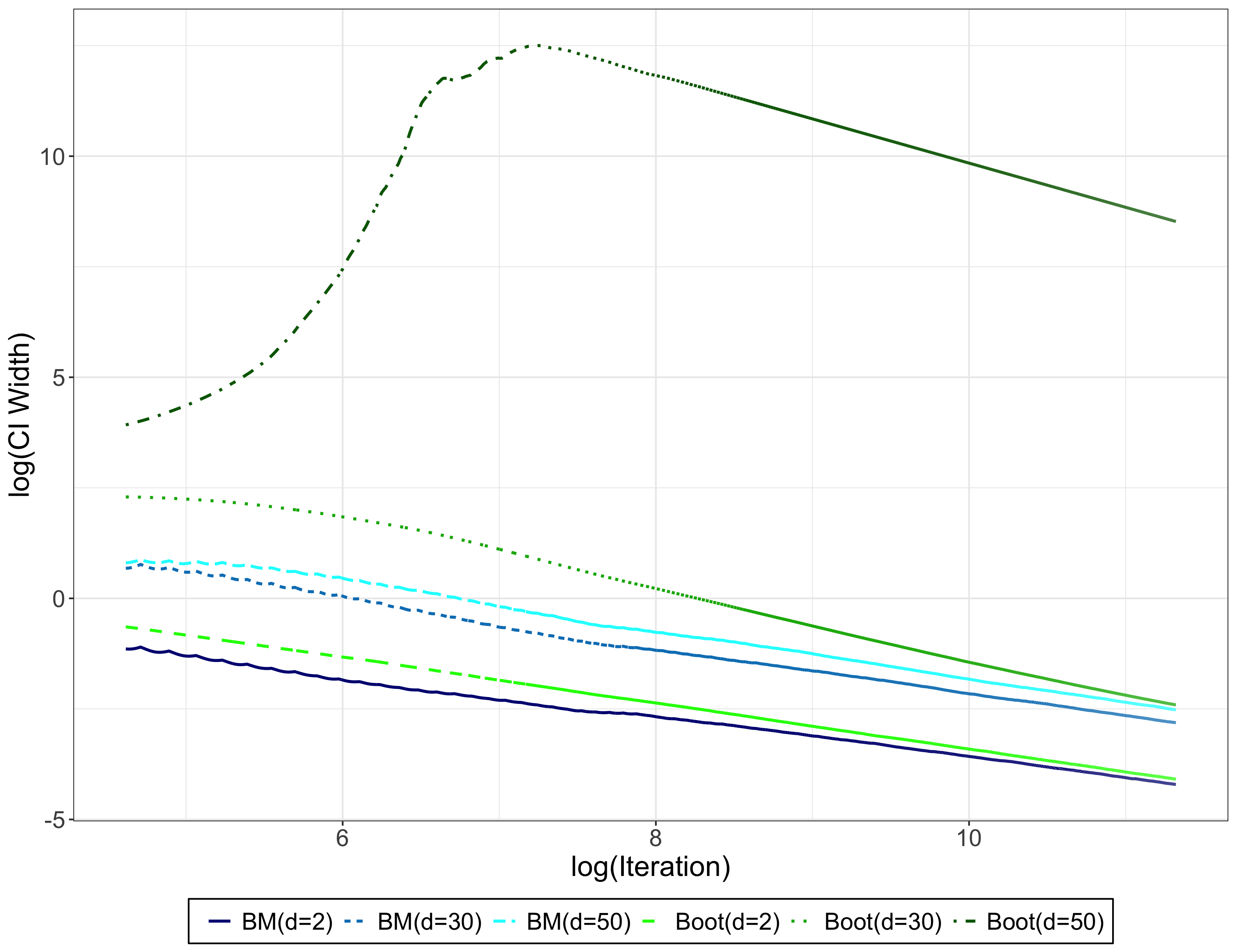}
    \includegraphics[width=0.22\textwidth,height=2in]{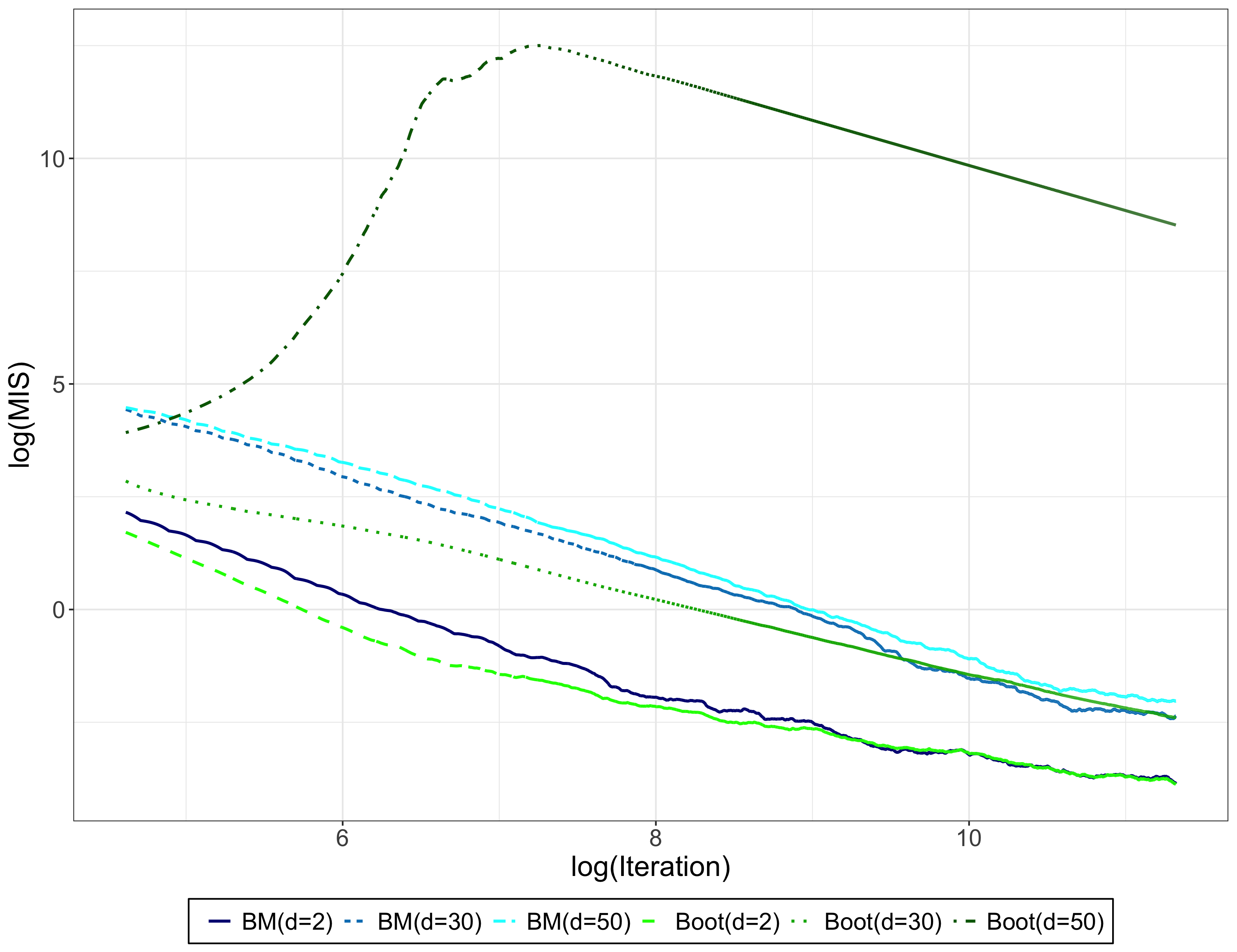}
    \\
    \includegraphics[width=0.22\textwidth,height=2in]{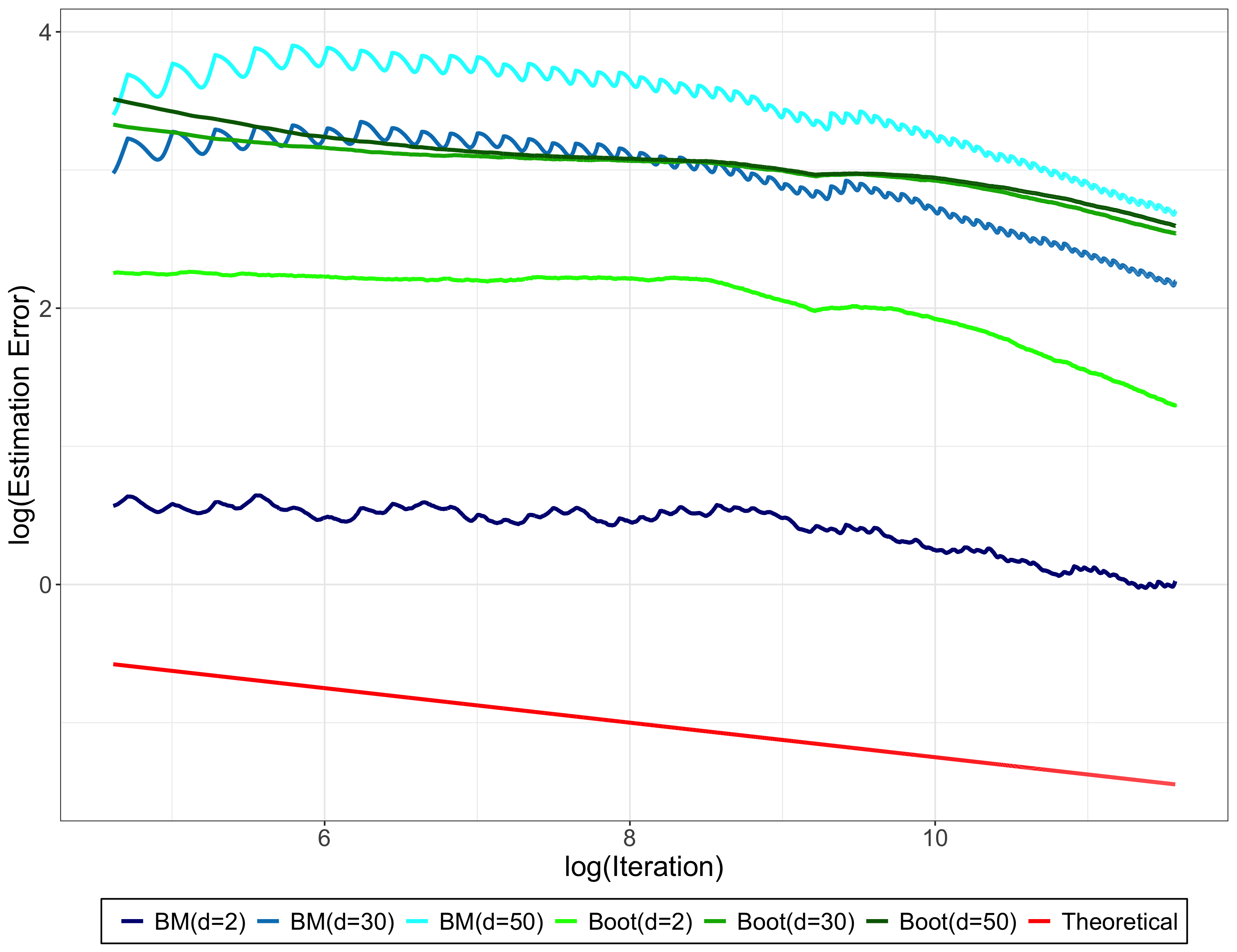}
    \includegraphics[width=0.22\textwidth,height=2in]{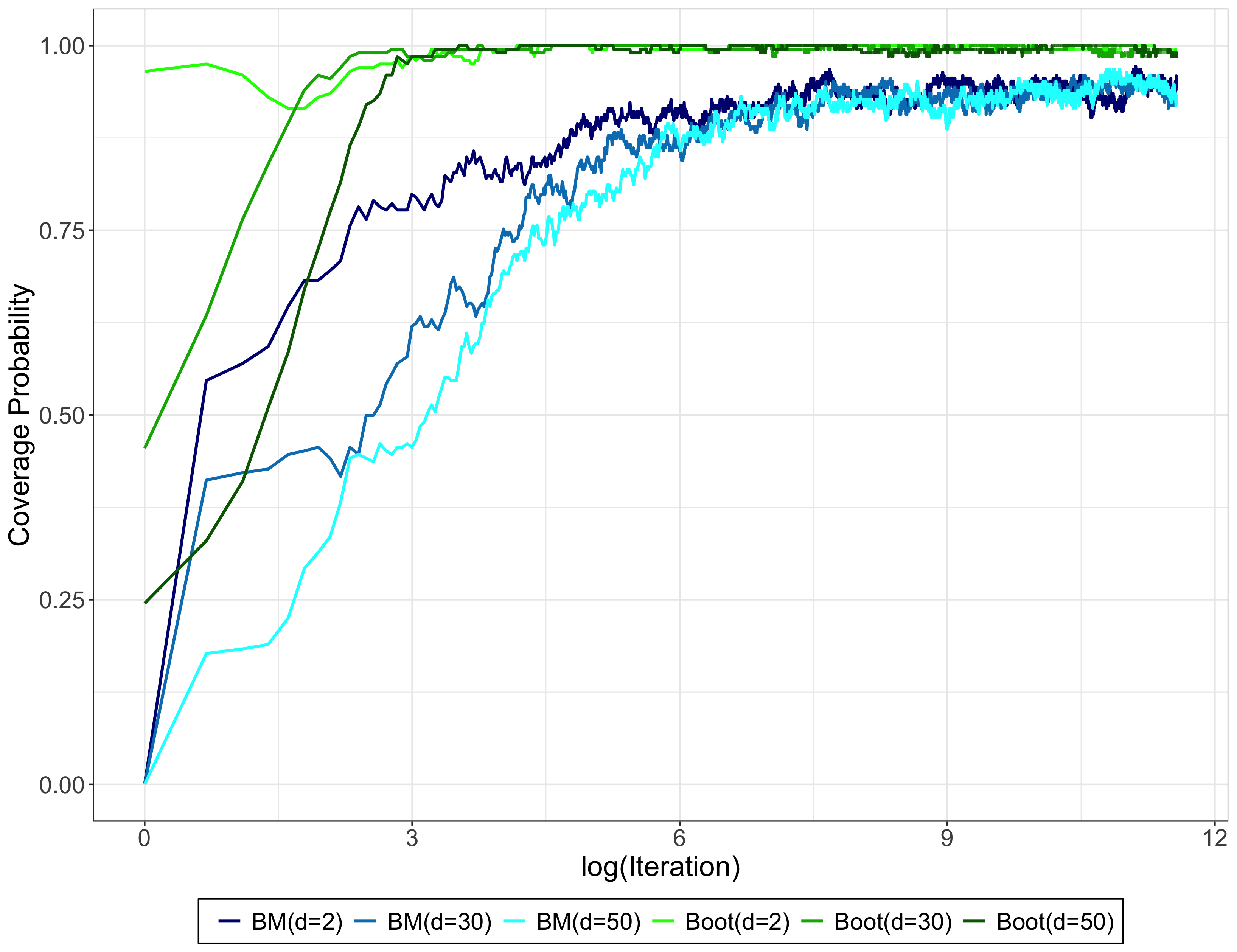}
    \includegraphics[width=0.22\textwidth,height=2in]{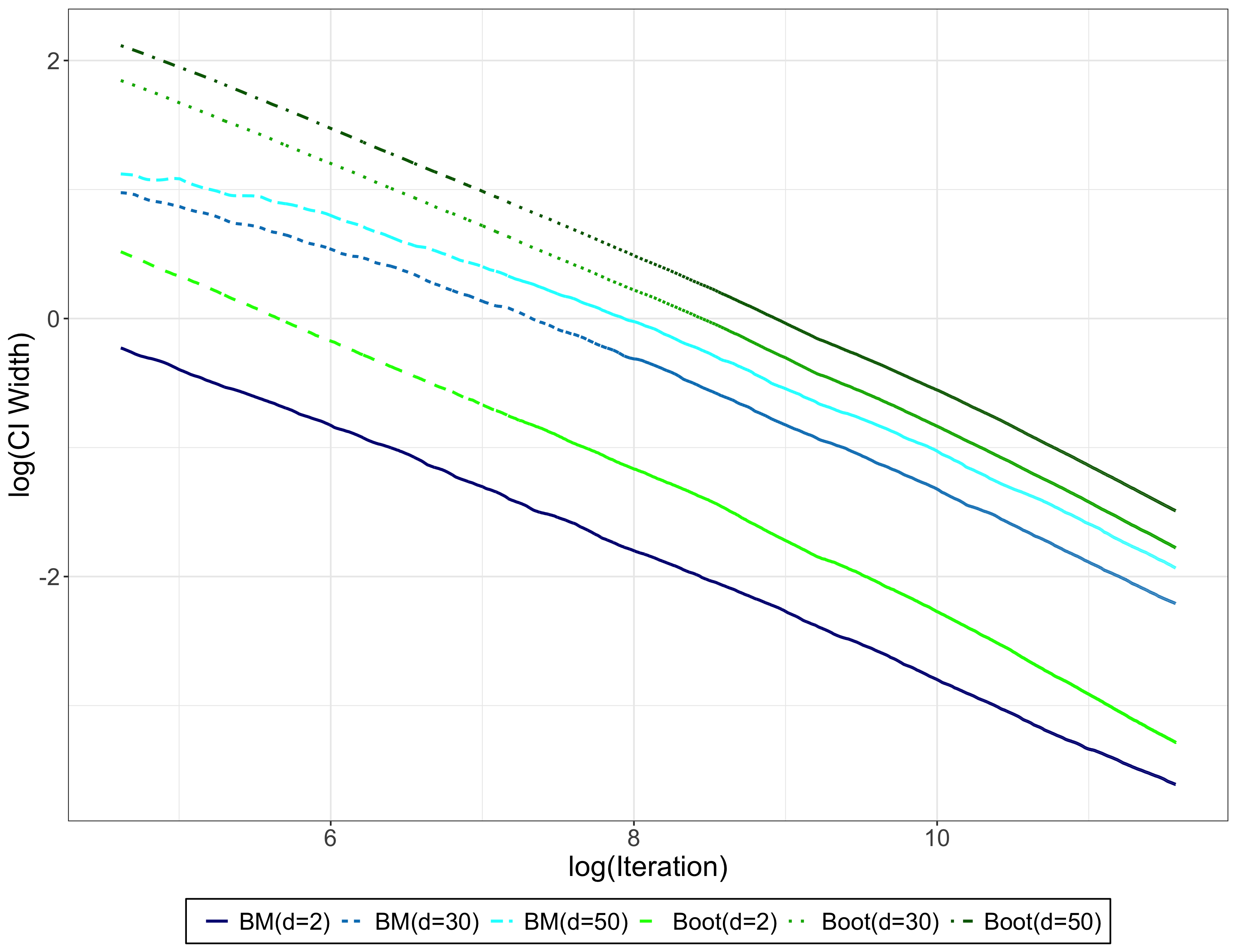}
    \includegraphics[width=0.22\textwidth,height=2in]{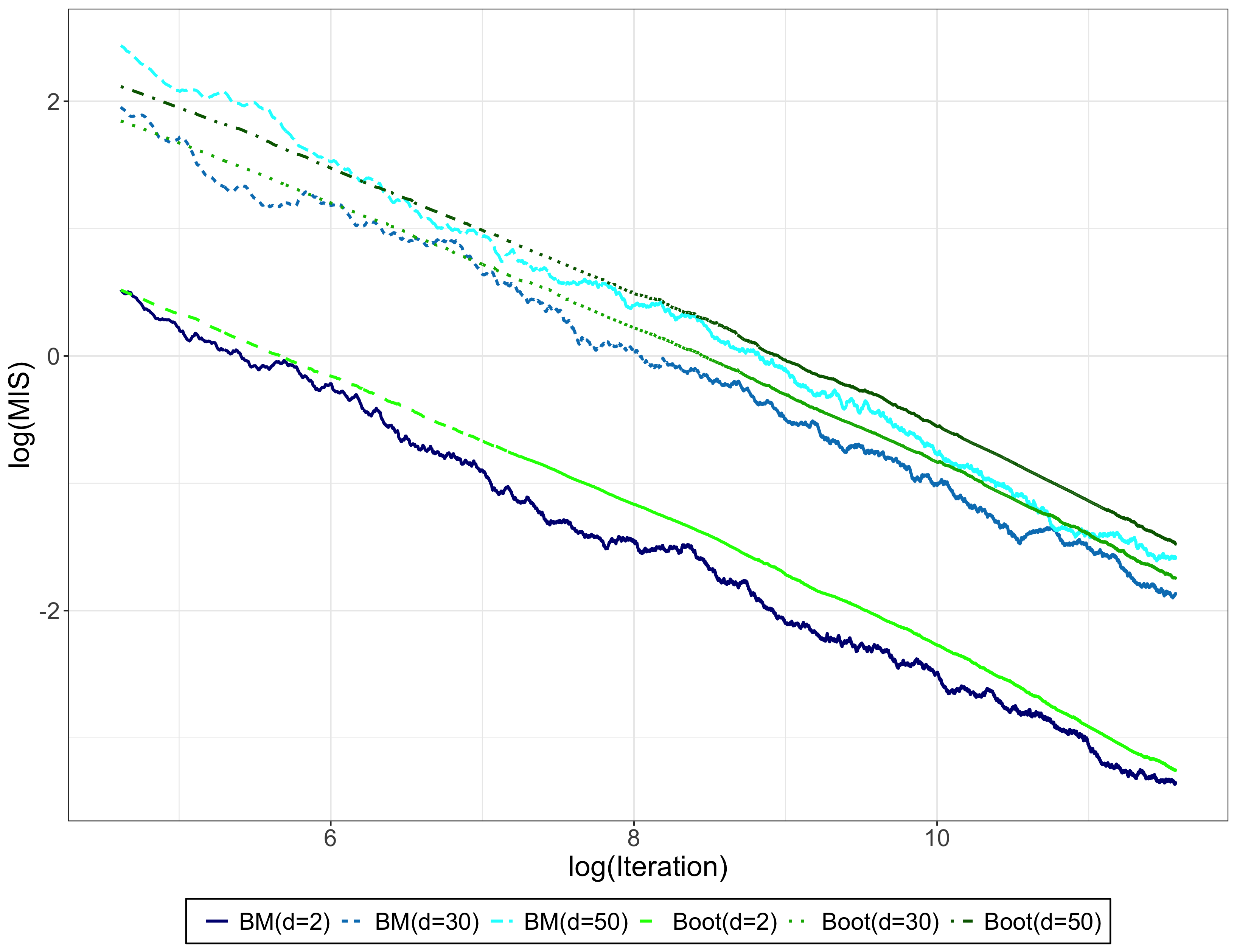}
    \caption{\it Linear regression (upper) and logistic regression (lower) with synthetic state-dependent Markovian data: From left to right columns are the plots of $\log(\text{Estimation Error})$, coverage probability, $\log(\text{width of Confidence Interval (CI)})$, and $\log(\text{MIS})$ respectively. The blue (green) lines correspond to \texttt{BM} (\texttt{Boot}) for dimensions $d=2,30,50$. The red line in the plots of $\log(\text{Estimation Error})$ shows the theoretical rate obtained in Theorem~\ref{pf:mainthm}. \texttt{Boot} achieves higher coverage probability by constructing wider confidence interval. The plot of $\log(\text{MIS})$ shows that, $\hat{\Sigma}_n$ performs better than \texttt{Boot}, especially in high dimensional linear regression. 
    }
    \label{fig:linreg_synth_boot}
\end{sidewaysfigure}
In the lower row of Figure~\ref{fig:linreg_synth_boot} we compare \texttt{BM} with \texttt{Boot}. Similar to linear regression, here too \texttt{Boot} achieves higher coverage probability by constructing wider confidence intervals, and the performance of \texttt{Boot} degrades in higher dimension. Although, the effect of dimension is not as pronounced as linear regression in this case. 
Regarding the effect of state-dependence, in logistic regression (Fig.~\ref{fig:state_dep_bm_logreg_nuvar}), the performance of the estimator $\hat{\Sigma}_n$ is less sensitive to the state-dependence.
\subsection{Data example: Strategic Classification}
In this section we illustrate our algorithm on the strategic classification problem as discussed in Section~\ref{sec:motivation} with the \texttt{GiveMeSomeCredit}\footnote{Available at \url{https://www.kaggle.com/c/GiveMeSomeCredit/data}} dataset. The main task is a credit score classification problem where the bank (learner) has to decide whether a loan should be granted to a client (agent).  Given the knowledge of the classifier the clients can distort some of their personal traits  in order to get approved for a loan. Here we use a linear classifier, given by $
	h(x;\theta)=\theta^\top x,
$
where $\theta,x\in \mathbb{R}^d$. We consider logistic loss with a regularizer. We consider a quadratic cost given by $c(x,x')=\norm{x_S-x_S'}_2^2/(2\lambda)$ where $\lambda$ is the sensitivity of the underlying distribution on $\theta$. We assume that the agents iteratively learn $x_S'$ similar to \cite{li2022state}. Furthermore, following \cite{li2022state}, we also assume that the agents use Gradient Ascent (GA) to learn the best response. 
\begin{figure}[h]
    \centering
    \includegraphics[width=0.9\textwidth]{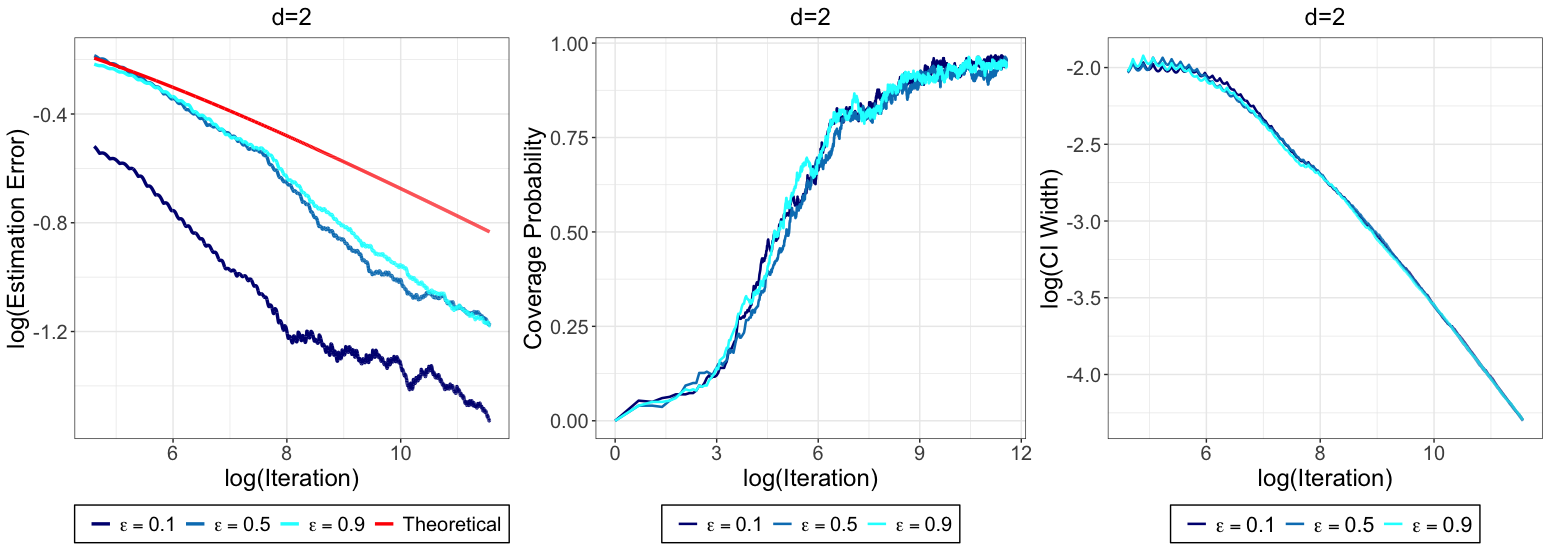} 
    \includegraphics[width=0.9\textwidth]{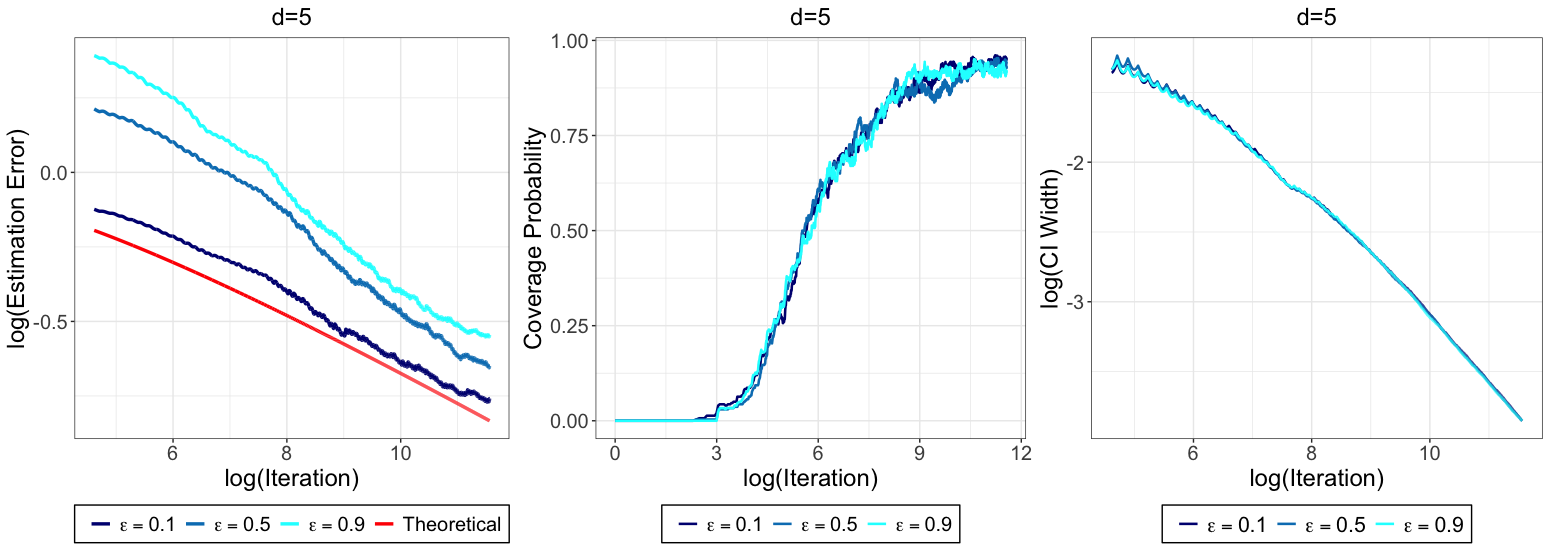}
    \caption{\it Linear Regression with synthetic State-dependent Markovian Dataset. From left to right the columns show the evolution of $\log(\text{Estimation Error})$, coverage probability, and $\log(\text{width of Confidence Interval (CI)})$ respectively over iterations. The top and bottom rows correspond to $d=2$, and $d=5$ respectively. The red line in the plot of $\log(\text{Estimation Error})$ corresponds to the theoretical rate obtained in Theorem~\ref{th:mainthm}.  }
    \label{fig:state_dep_bm_linreg_nuvar}
\end{figure}
After removing outliers, we select a subset of randomly chosen $9000$ samples (agents). Each agent has $10$ features. Note that since Algorithm~\ref{alg:tsgd} computes the gradient on one sample at every iterate, the computation time is independent of the total number of agents. We assume that the agents can modify Revolving Utilization, the Number of Open Credit Lines, and the Number of Real Estate Loans or Lines. We compare \texttt{BM} and \texttt{Boot} under varying degrees of state-dependence by setting the number of agents who can modify the feature at every time instant $n_1=50, 500, 1000, 2000$. For \texttt{Boot}, we set the number of bootstrap replicates to $60$. Similar to \cite{li2022state}, we set $\alpha=0.5\lambda$, and $\lambda=0.01$. From the leftmost plot in the panel of Figure~\ref{fig:logreg_bank} one can see that, over iterations, the slope of the observed estimation error $\expec{\norm{\hat{\Sigma}_n-\Sigma}_2}{}$ matches the theoretical rate. We observe that both \texttt{BM} and \texttt{Boot} behave similarly in terms of coverage probability, CI width, and MIS which is consistent with our observations in the experiments with low-dimensional synthetic data. In terms of the convergence rate of the estimation error of the covariance, \texttt{BM} seems to be better and less sensitive to state-dependence compared to \texttt{Boot}.
\begin{figure}[h]
    \centering
    \includegraphics[width=0.9\textwidth]{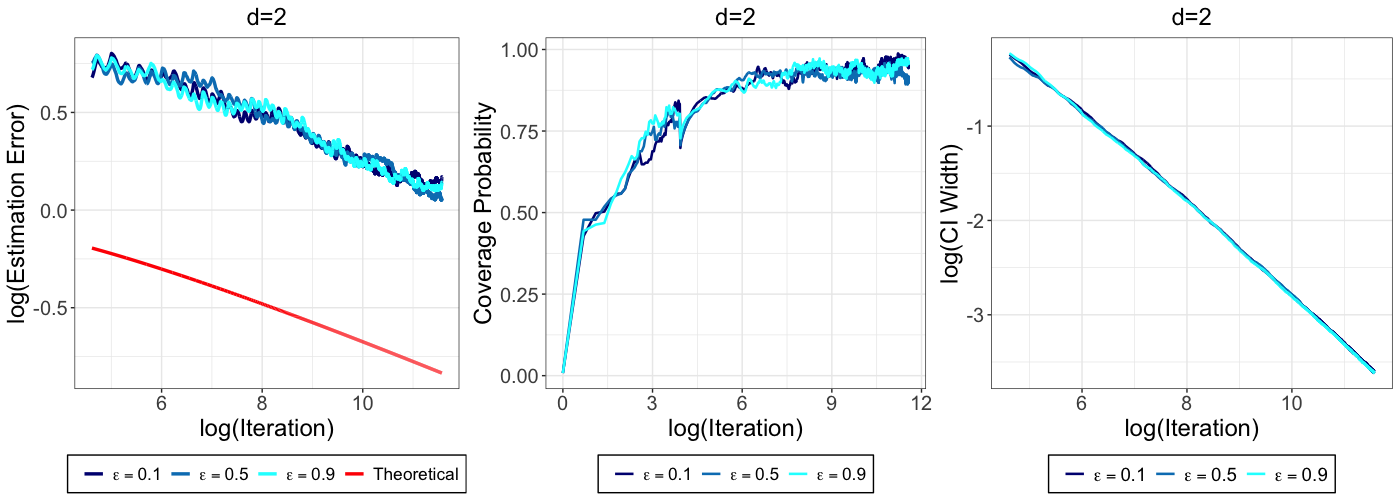} 
    \includegraphics[width=0.9\textwidth]{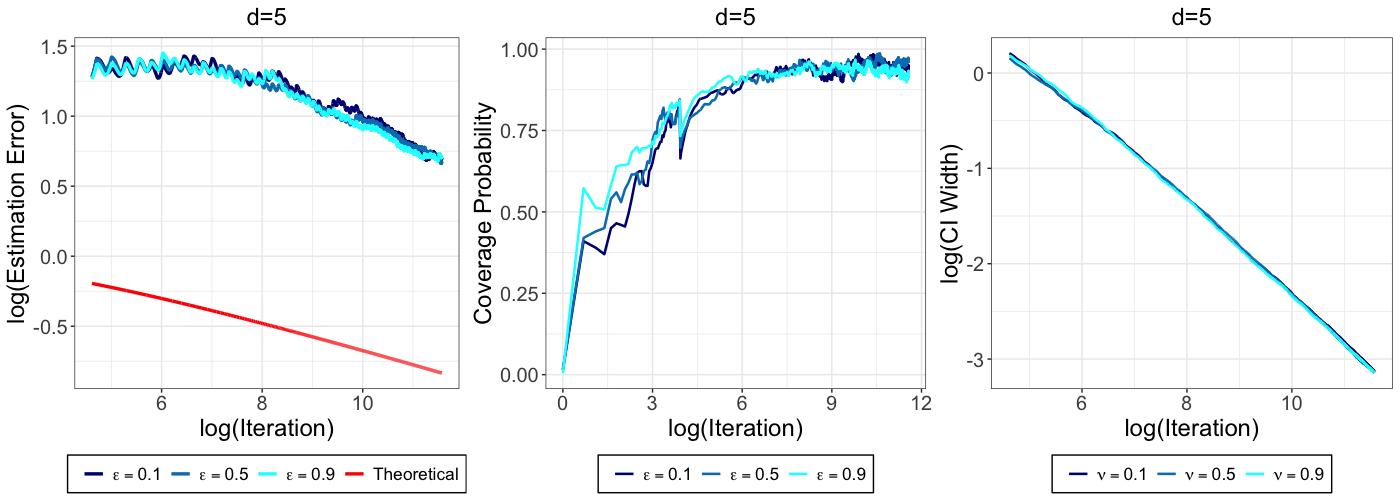}
    \caption{\it Logistic Regression with synthetic State-dependent Markovian Dataset. From left to right the columns show the evolution of $\log(\text{Estimation Error})$, coverage probability, and $\log(\text{width of Confidence Interval (CI)})$ respectively over iterations. The top and bottom rows correspond to $d=2$, and $d=5$ respectively. The red line in the plot of $\log(\text{Estimation Error})$ corresponds to the theoretical rate obtained in Theorem~\ref{th:mainthm}.  }
    \label{fig:state_dep_bm_logreg_nuvar}
\end{figure}

%% file: conclusion.tex
\section{Conclusion and Future Work}
\begin{figure}[t]
    \centering    
    \includegraphics[width=0.36\textwidth]{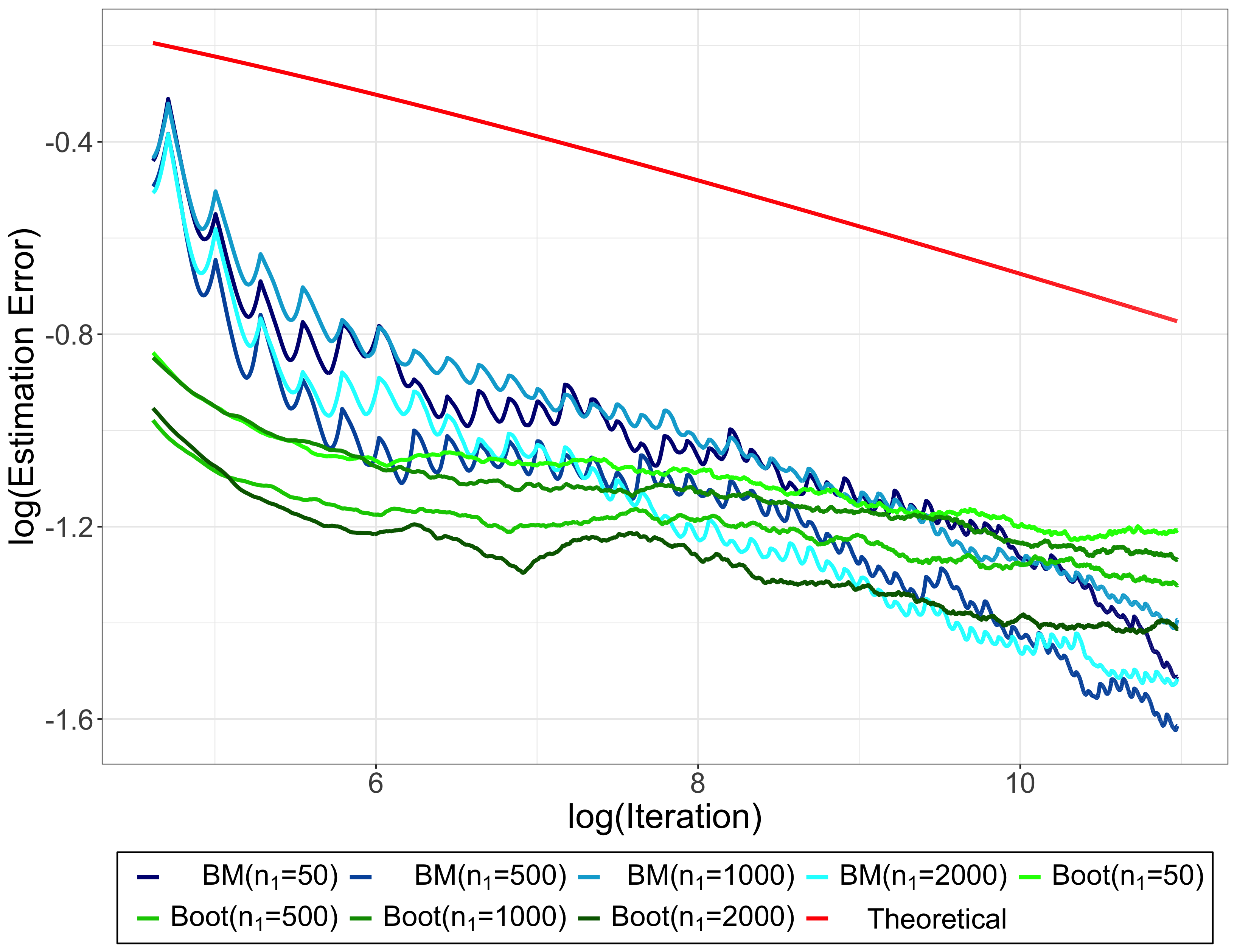} 
    \includegraphics[width=0.36\textwidth]{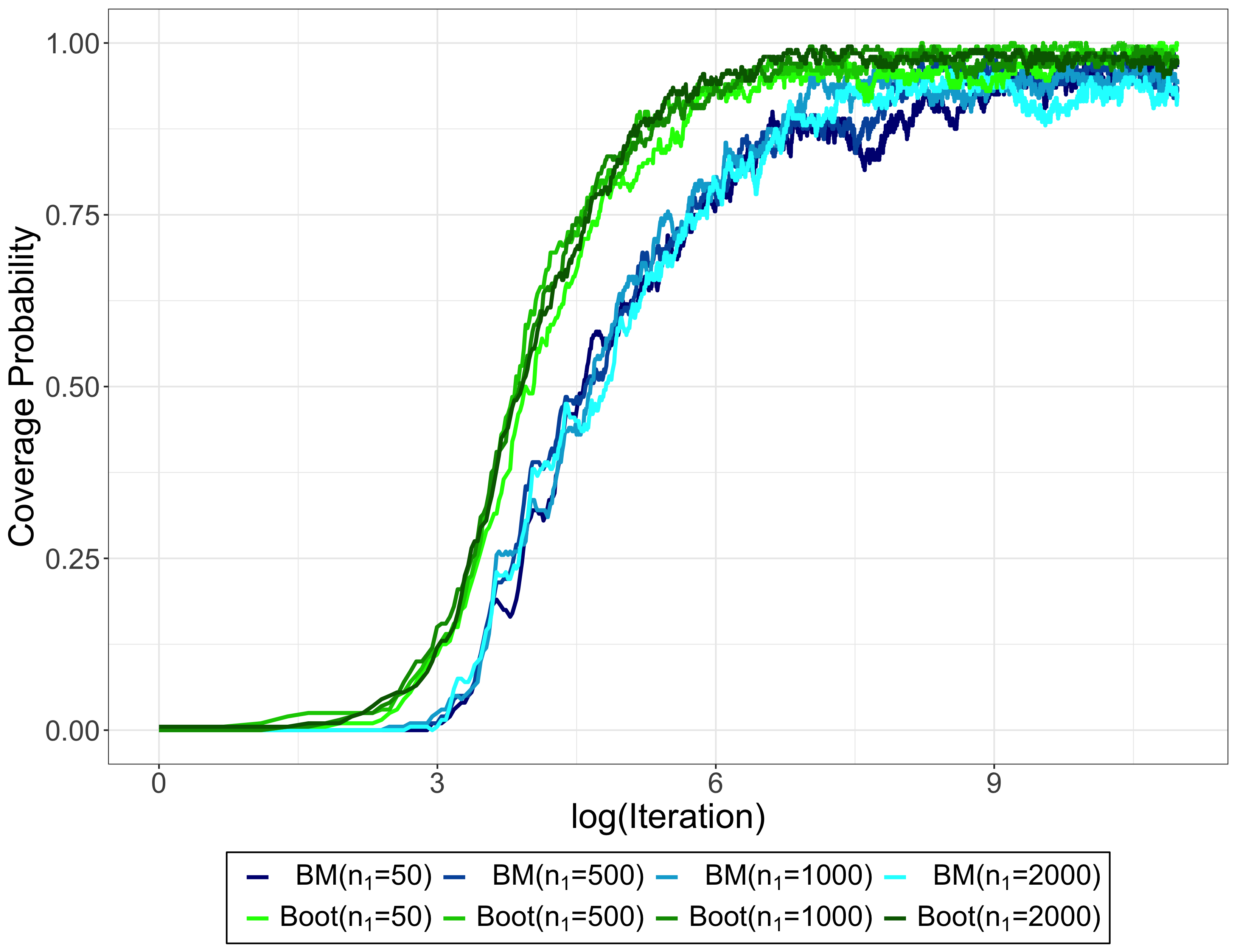}\\
    \includegraphics[width=0.36\textwidth]{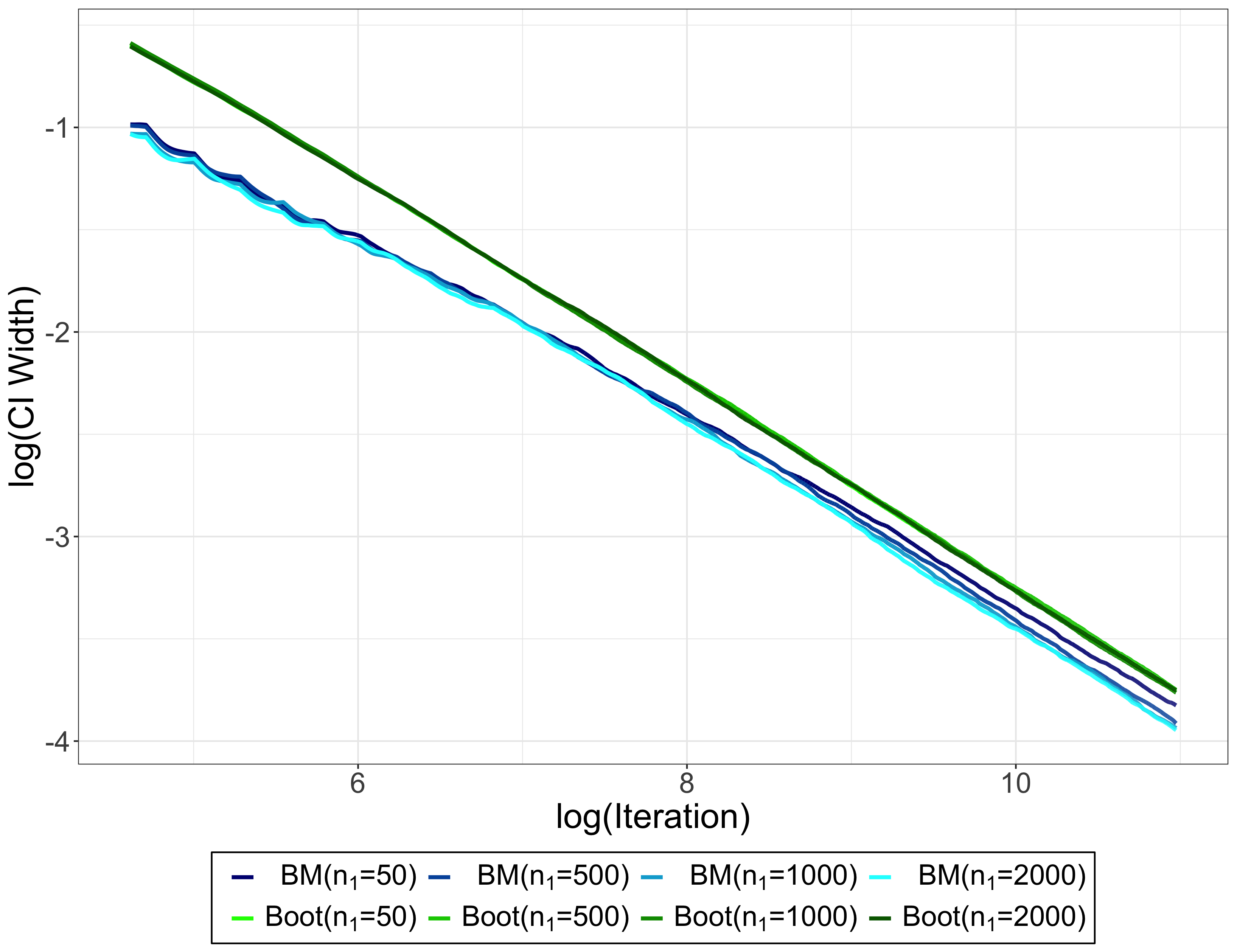}
    \includegraphics[width=0.36\textwidth]{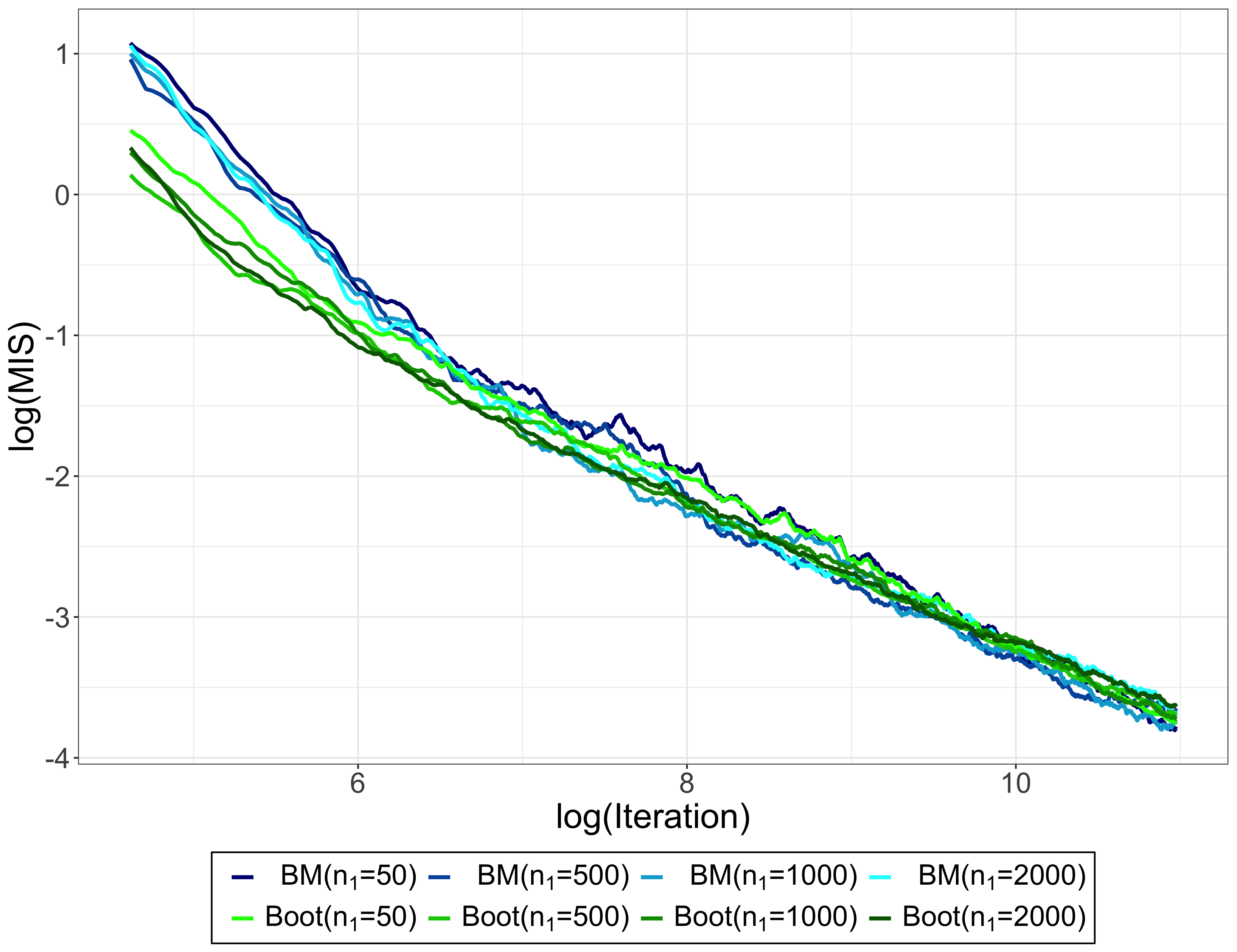}
    \caption{\it Logistic Regression with \texttt{GiveMeSomeCredit} Dataset. The plots show the evolution of $\log(\text{Estimation Error})$, coverage probability, $\log(\text{width of Confidence Interval (CI)})$, and $\log(\text{MIS})$. The red line in the plot of $\log(\text{Estimation Error})$ corresponds to the theoretical rate in Theorem~\ref{th:mainthm}. The slope of $\log(\text{Estimation Error})$ matches the theoretical rate after approximately $500$ iterations. \texttt{BM} and \texttt{Boot} have comparable performances in terms of coverage probability, CI width and MIS. In terms of convergence rate of the estimation error, \texttt{BM} seems to be better and less sensitive to state-dependence compared to \texttt{Boot}.}
    \label{fig:logreg_bank}
\end{figure}
In this work, we study the online batch-means covariance estimator $\hat{\Sigma}_n$ for SGD under Markovian sampling. We show that the convergence rate for $\hat{\Sigma}_n$ is $O\big(\sqrt{d}n^{-1/8}(\log n)^{1/4}\big)$ for state-dependent Markov data, and $O\big(\sqrt{d}n^{-1/8}\big)$ under state-independent Markovian sampling. Ignoring the logarithmic factor in the state-dependent case, this rate matches the best-known rate of convergence in the $\iid$ case \citep{zhu2021online,chen2020statistical}. We experimentally illustrate our results on synthetic and real datasets.  

There are several avenues for future work. For some applications, a convergence rate of $\tilde{O}(n^{-1/8})$, as is the case for $\hat{\Sigma}_n$, may be too slow. Designing an estimator with a faster convergence rate, focusing on specific functionals, is an active area for research. The dimension dependency of $\hat{\Sigma}_n$ is of the order $O(\sqrt{d})$ which may not be suitable for high-dimensional problems. Designing a fully-online estimator with polylogarithmic dependence on $d$ for functionals of SGD under Markovian sampling is open.

%% file: appendix.tex
\clearpage
\section{Auxiliary results from literature}
We need the following two lemma from \cite{andrieu2005stability}.
\begin{lemma}[Lemma A.1,\cite{liang2010trajectory}] \label{lm:liangmainlemma}
Let Assumption~\ref{as:liangasA3} be true. Then the following hold:
\begin{enumerate}
    \item $\nabla F(\theta,x)$ is measurable for all $\theta\in\Theta$, and $\int_\mathcal{X}\|\nabla F(\theta,x)\|_2\pi_\theta(x)\mathrm{d}x<\infty$.
    \item For any $\theta\in\Theta$, the Poisson equation $u(\theta,x)-P_\theta u(\theta,x)=\nabla F(\theta,x)-\nabla f(\theta)$ has a solution $u(\theta,x)$, where $P_\theta u(\theta,x)=\int_{\mathcal{X}}u(\theta,x')P_{\theta} (x,x')\mathrm{d}x'$. There exist a function $V:\mathcal{X}\to[1,\infty)$ such that $\{x\in\mathcal{X},V(x)<\infty\}\neq \varnothing$, and a constant $\nu\in (0,1]$  such that for any compact subset $\mathcal {K}\subset \Theta$, the following holds:
    \begin{enumerate}
        \item $\sup_{\theta\in\mathcal{K}}\|\nabla F(\theta,x)\|_V<\infty$,
        \item $\sup_{\theta\in\mathcal{K}}\left(\|u(\theta,x)\|_V+\|P_{\theta} u(\theta,x)\|_V\right)<\infty$,
        \item $\sup_{(\theta_1,\theta_2)\in\mathcal{K}\times \mathcal{K}}\|\theta_1-\theta_2\|_2\left(\|u(\theta_1,x)-u(\theta_2,x)\|_V+\|P_{\theta_1} u(\theta_1,x)-P_{\theta_2} u(\theta_2,x)\|_V\right)<\infty$. 
    \end{enumerate}
\end{enumerate}
\end{lemma}
\begin{lemma}[Lemma A.2,\cite{liang2010trajectory}]\label{lm:lialemmaa2}
Let Assumptions~\ref{as:strongcon}, \ref{as:liangasA3}, and \ref{as:liangasA4} be true. Let $\mathcal{X}_0\subset\mathcal{X}$ be such that $\sup_{x\in\mathcal{X}_0}V(x)<\infty$ and that $\mathcal{K}_0\subset M_{C_0}$, where $M_{C_0}\coloneqq \{\theta\in\Theta :f(\theta)\leq C_0\}$, and $C_0>0$ is a constant. Then $\sup_k\expec{V^{\alpha_0}(x_k)I(k\geq k_{\sigma_s})}{}<\infty$, where $\alpha_0\geq 8$ is defined in Assumption~\ref{as:liangasA3}.
\end{lemma}

We also introduce the following two results from the stochastic approximation literature used in the proofs of Lemma~\ref{lm:part1}, and Lemma~\ref{lm:part2}. Let $Q$ be a positive-definite matrix. We start with the notation,
\begin{align}
   W_i^j&\coloneqq\textstyle\prod_{k=i+1}^j(\Id-\eta_kQ) \quad \text{for}~~ j>i \quad\text{with}\quad W_i^i\coloneqq\Id,\label{eq:wijdef}\\[2.5pt]
    &S_i^j\coloneqq\textstyle\sum_{k=i+1}^jW_i^k \quad \text{for}~~ j>i \numberthis\label{eq:sijdef}\quad\text{with}\quad S_i^i\coloneqq0.
\end{align}
\begin{lemma}[\cite{polyak1992acceleration,zhu2021online}]\label{lm:Wijbound}
Let $W_i^j$ be defined as in \eqref{eq:wijdef} with $Q\succ 0$, and $\eta_k=\eta k^{-a}$ for $1/2<a<1$. Then, 
\begin{align}
    \norm{W_i^j}_2\leq\exp\big(-\eta\gamma\textstyle\sum_{k=i+1}^jk^{-a}\big)\leq \textstyle \exp\big(-\frac{\gamma\eta}{1-a}\left(j^{1-a}-(i+1)^{1-a}\right)\big), \label{eq:Wijbound}
\end{align}
where $\gamma=\min\left(\lambda_{min}(Q),1/(2\eta)\right)$.
\end{lemma}
\begin{lemma}[\cite{polyak1992acceleration,zhu2021online}]\label{lm:Snormbound}
Let $S_i^j$ be defined as in \eqref{eq:sijdef}, and $\eta_k=\eta k^{-a}$ for $1/2<a<1$. Then, we have $\norm{S_i^j}_2\lesssim (i+1)^a$.
\end{lemma}
\input{proofofthm21}
\section{Proof of Lemma~\ref{lm:expecconviter}}
\begin{proof}[Proof of Lemma~\ref{lm:expecconviter}]\label{pf:theta2224bound}
First, let us consider the following process: $\theta'_k=\theta_k-\tilde{\zeta}_{k+1}$. Then, combining \eqref{eq:decomp}, and \eqref{eq:sgd}, we get,
\begin{align*}
    \theta'_{k+1}=&\theta_k-\eta_{k+1}\left(\nabla f(\theta_k)+\xi_{k+1}(\theta_k,x_{k+1})\right)-\tilde{\zeta}_{k+2}\\[2.5pt]
    =& \theta_k-\eta_{k+1}\left(\nabla f(\theta_k)+e_{k+1}+\nu_{k+1}+\zeta_{k+1}\right)-\tilde{\zeta}_{k+2}\\[2.5pt]
    =& \theta_k-\eta_{k+1}\left(\nabla f(\theta_k)+e_{k+1}+\nu_{k+1}\right)-\tilde{\zeta}_{k+1}+\tilde{\zeta}_{k+2}-\tilde{\zeta}_{k+2}\\[2.5pt]
    =&\theta'_k-\eta_{k+1}\left(\nabla f(\theta_k)+e_{k+1}+\nu_{k+1}\right). 
\end{align*}
Note that for $p\geq 2$, we have
\begin{align}\label{eq:tildethetathetarelation}
    \expec{\|\theta_k-\theta^*\|_2^{p}}{}\leq 2^{p/2-1}\expec{\|\theta'_k-\theta^*\|_2^p}{}+2^{p/2-1}\expec{\|\tzt_{k+1}\|_2^p}{}\lesssim \expec{\|\theta'_k-\theta^*\|_2^p}{}+\eta_{k+1}^p.
\end{align}
Then, it is sufficient to establish bounds on $\expec{\|\theta'_k-\theta^*\|_2^2}{}$ and $\expec{\|\theta'_k-\theta^*\|_2^4}{}$, and then the proof follows from \eqref{eq:tildethetathetarelation}. 

To proceed, note that we have 
\begin{align}
    \|\theta'_{k+1}-\theta^*\|_2^2=&\|\theta'_{k}-\theta^*\|_2^2-2\eta_{k+1}\nabla f(\theta_k)^\top(\theta'_k-\theta^*)-2\eta_{k+1}\left(\nu_{k+1}+e_{k+1}\right)^\top(\theta'_k-\theta^*)\nonumber\\[2.5pt]
    &+\eta_{k+1}^2\norm{\nabla f(\theta_k)+\nu_{k+1}+e_{k+1}}_2^2\nonumber\\[2.5pt]
    \expec{\|\theta'_{k+1}-\theta^*\|_2^2|\cF_k}{}\lesssim &\|\theta'_{k}-\theta^*\|_2^2-2\eta_{k+1}\nabla f(\theta_k)^\top(\theta'_k-\theta^*)-2\eta_{k+1}\expec{\nu_{k+1}|\cF_k}{}^\top(\theta'_k-\theta^*)\nonumber\\[2.5pt]
    &+3\eta_{k+1}^2\left(\norm{\nabla f(\theta_k)}_2^2+C\right).\label{eq:tthetanormboundintermed}
\end{align}
Using Assumption~\ref{as:strongcon}, we have,
\begin{align}\label{eq:gradtthetainprod}
    \nabla f(\theta_k)^\top(\theta'_k-\theta^*)\geq \frac{\mu}{2}\|\theta_k-\theta^*\|_2^2+\nabla f(\theta_k)^\top\tzt_k\geq \frac{\mu}{4}\|\theta'_k-\theta^*\|_2^2-\frac{\mu}{2}\|\tzt_k\|_2^2+\nabla f(\theta_k)^\top\tzt_k. 
\end{align}
Then, from \eqref{eq:tthetanormboundintermed}, using \eqref{eq:gradtthetainprod}, and \eqref{eq:noisylipgrad} in Assumption~\ref{as:liangasA3}, we get
\begin{align*}
  &\expec{\|\theta'_{k+1}-\theta^*\|_2^2|\cF_k}{}\\[2.5pt]
  \lesssim &   \|\theta'_{k}-\theta^*\|_2^2-2\eta_{k+1}\left(\frac{\mu}{4}\|\theta'_k-\theta^*\|_2^2-\frac{\mu}{2}\|\tzt_k\|_2^2+\nabla f(\theta_k)^\top\tzt_k\right)-2\eta_{k+1}\expec{\nu_{k+1}|\cF_k}{}^\top(\theta'_k-\theta^*)\\[2.5pt]
  &+3\eta_{k+1}^2\left(\norm{\nabla f(\theta_k)}_2^2+C\right)\\[2.5pt]
  \lesssim & \|\theta'_{k}-\theta^*\|_2^2-\frac{\mu \eta_{k+1}}{2}\|\theta'_k-\theta^*\|_2^2+\mu \eta_{k+1}\|\tzt_k\|_2^2+\eta_{k+1}\left( \frac{\mu}{8}\|\theta_k-\theta^*\|_2^2+\frac{8L_G^2}{\mu}\|\tzt_k\|_2^2\right)\\[2.5pt]
  &+\eta_{k+1}\left(\frac{8}{\mu}\|\expec{\nu_{k+1}|\cF_k}{}\|_2^2+\frac{\mu}{8}\|\theta'_k-\theta^*\|_2^2\right)
  +3\eta_{k+1}^2\left(L_G^2\norm{\theta_k-\theta^*}_2^2+C\right)\\[2.5pt]
   \lesssim & \left(1-\frac{\mu \eta_{k+1}}{4}\right)\|\theta'_{k}-\theta^*\|_2^2+\mu \eta_{k+1}\|\tzt_k\|_2^2+\frac{8L_G^2\eta_{k+1}}{\mu}\|\tzt_k\|_2^2+\frac{8\eta_{k+1}}{\mu}\|\expec{\nu_{k+1}|\cF_k}{}\|_2^2\\[2.5pt]
  &+3\eta_{k+1}^2\left(L_G^2\norm{\theta'_k-\theta^*}_2^2+L_G^2\norm{\tzt_{k+1}}_2^2+C\right).
\end{align*}
Choosing $\eta_{k+1}\leq \frac{\mu}{24L_G^2}$, using \eqref{eq:tildezetabound}, \eqref{eq:nunormbound}, and taking expectation on both sides we get,
\begin{align}\label{eq:thetak22recur}
    \expec{\|\theta'_{k+1}-\theta^*\|_2^2}{}
    \lesssim & \left(1-\frac{\mu \eta_{k+1}}{8}\right)\expec{\|\theta'_{k}-\theta^*\|_2^2}{}+\mu \eta_{k+1}^3+\frac{8L_G^2\eta_{k+1}^3}{\mu}\\
    &\qquad+\frac{8\eta_{k+1}^3}{\mu}
  +3\eta_{k+1}^2\left(L_G^2\eta_{k+1}^2+C\right)\\
  \lesssim & \left(1-\frac{\mu \eta_{k+1}}{8}\right)\expec{\|\theta'_{k}-\theta^*\|_2^2}{}+\eta_{k+1}^2.
\end{align}
Similarly, we have
\begin{align*}
    \|\theta'_{k+1}-\theta^*\|_2^4& \leq \|\theta'_{k}-\theta^*\|_2^4-4\eta_{k+1}\nabla f(\theta_k)^\top (\theta'_k-\theta^*)\|\theta'_k-\theta^*\|_2^2\\[2.5pt]
    &~~~-4\eta_{k+1}(e_{k+1}+\nu_{k+1})^\top(\theta'_k-\theta^*)\|\theta'_k-\theta^*\|_2^2+\mu\eta_{k+1}\|\theta'_{k}-\theta^*\|_2^4\\[2.5pt]
    &~~~+\eta_{k+1}^3\|\theta'_k-\theta^*\|_2^4+\eta_{k+1}^3(9/\mu+3+\eta_{k+1})\|\nabla f(\theta_k)+e_{k+1}+\nu_{k+1}\|_2^4.
\end{align*}
Taking expectation with respect to $\cF_k$ on both sides we get, 
\begin{align*}
    &\expec{\|\theta'_{k+1}-\theta^*\|_2^4|\cF_k}{}\\[2.5pt]
    \lesssim &~\|\theta'_{k}-\theta^*\|_2^4-4\eta_{k+1}\left(\frac{3\mu}{8}\|\theta'_k-\theta^*\|_2^2-\frac{3\mu}{2}\|\tzt_{k+1}\|_2^2+\nabla f(\theta_k)^\top\tzt_{k+1}\right)\|\theta'_k-\theta^*\|_2^2\\[2.5pt]
    &~-4\eta_{k+1}\expec{\nu_{k+1}|\cF_k}{}^\top(\theta'_k-\theta^*)\|\theta'_k-\theta^*\|_2^2\\[2.5pt]&~+\mu\eta_{k+1}\|\theta'_{k}-\theta^*\|_2^4
    +\eta_{k+1}^3\|\theta'_k-\theta^*\|_2^4+\eta_{k+1}^3(9/\mu+3+\eta_{k+1})\left(\|\theta_k-\theta^*\|_2^4+C\right)\\[2.5pt]
    \lesssim &~\|\theta'_{k}-\theta^*\|_2^4-\frac{\eta_{k+1}\mu}{4}\|\theta'_k-\theta^*\|_2^4-4\eta_{k+1}\left(-\frac{3\mu}{2}\|\tzt_{k+1}\|_2^2+\nabla f(\theta_k)^\top\tzt_{k+1}\right)\|\theta'_k-\theta^*\|_2^2\\[2.5pt]
    &~+4\eta_{k+1}^2\|\theta'_k-\theta^*\|_2^3+\eta_{k+1}^3\left(\|\theta_k-\theta^*\|_2^4+C\right).
\end{align*}
Now, using Young's inequality and choosing $\eta_{k+1}\leq \min((\mu/48)^{3/2},\mu/16)$ we have,
\begin{align}
    &4\eta_{k+1}^2\|\theta'_k-\theta^*\|_2^3\leq \eta_{k+1}^3+3\eta_{k+1}^\frac{5}{3}\|\theta'_k-\theta^*\|_2^4\leq \eta_{k+1}^3+\frac{\eta_{k+1}\mu}{16}\|\theta'_k-\theta^*\|_2^4,\label{eq:eta2ttheta3}\\[2.5pt]
    &6\eta_{k+1}\|\tzt_{k+1}\|_2^2\|\theta'_k-\theta^*\|_2^2\leq 9\|\tzt_{k+1}\|_2^4+\eta_{k+1}^2\|\theta'_k-\theta^*\|_2^4\leq  9\|\tzt_{k+1}\|_2^4+\frac{\eta_{k+1}\mu}{16}\|\theta'_k-\theta^*\|_2^4,\label{eq:etazeta2tthteta2}\\[2.5pt]
    &4\eta_{k+1}\nabla f(\theta_k)^\top\tzt_{k+1}\|\theta'_k-\theta^*\|_2^2\lesssim \frac{\eta_{k+1}\mu}{16}\|\theta'_{k}-\theta^*\|_2^4+\eta_{k+1}\|\tzt_{k+1}\|_2^4. \label{eq:etanablazetatthteta2}
\end{align}
Using  \eqref{eq:eta2ttheta3}, \eqref{eq:etazeta2tthteta2}, \eqref{eq:etanablazetatthteta2}, and \eqref{eq:tildezetabound}, we get,
\begin{align}
    &\expec{\|\theta'_{k+1}-\theta^*\|_2^4|\cF_{k+1}}{}\lesssim \left(1-\frac{\mu\eta_{k+1}}{16}\right)\|\theta'_{k}-\theta^*\|_2^4+\|\tzt_{k+1}\|_2^4+\eta_{k+1}^3\nonumber\\[2.5pt]
    &\expec{\|\theta'_{k+1}-\theta^*\|_2^4}{}\lesssim \left(1-\frac{\mu\eta_{k+1}}{16}\right)\expec{\|\theta'_{k}-\theta^*\|_2^4}{}+\eta_{k+1}^3.\label{eq:thetak24recur}
\end{align}
Then, from \eqref{eq:thetak22recur}, and \eqref{eq:thetak24recur} we have,
\begin{align*}
\expec{\|\theta'_{k+1}-\theta^*\|_2^2}{}
    =O(\eta_{k+1}),\qquad\text{and}\qquad\expec{\|\theta'_{k+1}-\theta^*\|_2^4}{}= O(\eta_{k+1}^2).
\end{align*}
Using \eqref{eq:tildethetathetarelation}, we have,
\begin{align*}
\expec{\|\theta_{k+1}-\theta^*\|_2^2}{}=O(\eta_{k+1}),\qquad\text{and}\qquad\expec{\|\theta_{k+1}-\theta^*\|_2^4}{}= O(\eta_{k+1}^2).
\end{align*}
Now by Jensen's inequality we have,
\begin{align*}
    \expec{\|\theta_{k+1}-\theta^*\|_2}{}\leq \sqrt{\expec{\|\theta_{k+1}-\theta^*\|_2^2}{}}
    =O(\sqrt{\eta_{k+1}}).
\end{align*}
\end{proof}

\section{Proofs for Lemma used in Theorem~\ref{th:mainthm}}\label{sec:proofsthm21}
\begin{lemma}\label{lm:thetabarthetabarbound}
Let the conditions of Lemma~\ref{lm:poisregular} be true. Then, 
\begin{align}
    \expec{\norm{(\textstyle\sum_{i=1}^nl_i)^{-1}\sum_{i=1}^nl_i^2\bar{\vt}_n\bar{\vt}_n^\top}_2}{}\lesssim M^{-1}.\label{eq:thetabarthetabarbound}
\end{align}
\end{lemma}
\begin{proof}[Proof of Lemma~\ref{lm:thetabarthetabarbound}]
First note that we have $\|\bar{\vt}_n\bar{\vt}_n^\top\|_2\leq \tr(\bar{\vt}_n\bar{\vt}_n^\top)$ since $\bar{\vt}_n\bar{\vt}_n^\top\succcurlyeq 0$. Through recursion from \eqref{eq:linit}, we also have that  $
    \vt_n=W_0^n\vt_0+\textstyle\sum_{k=1}^nW_k^n\eta_k\xi_{k}(\theta_{k-1},x_{k}).$
Using the decomposition \eqref{eq:decomp}, we have
\begin{align}\label{eq:A1A2A3def}
\begin{aligned}
     \textstyle    \sum_{k=1}^n\vt_k=&W_0^n\vt_0+\textstyle\sum_{k=1}^nW_k^n\eta_k\left(e_k+\nu_k+\zeta_k\right)\\[2.5pt]
     =&\underbrace{S_0^n\vt_0+\textstyle\sum_{k=1}^n\left(\Id+S^n_k\right)\eta_ke_k}_{\textsf{A}_1}+\underbrace{\textstyle\sum_{k=1}^n\left(\Id+S^n_k\right)\eta_k\nu_k}_{\textsf{A}_2}+\underbrace{\textstyle\sum_{k=1}^n\left(\Id+S^n_k\right)\eta_k\zeta_k}_{\textsf{A}_3}
\end{aligned}
\end{align}
Then, using \cs inequality, we have
\begin{align}\label{eq:thetabarthetabardecomp}
\begin{aligned}
    \expec{\norm{\bar{\vt}_n\bar{\vt}_n^\top}_2}{}=&n^{-2}\tr\big(\expec{\big((A_1+A_2+A_3)(A_1+A_2+A_3)^\top\big)}{}\big)\\[2.5pt]
    \lesssim & n^{-2}\tr\big(\expec{A_1A_1^\top}{}\big)+n^{-2}\tr\big(\expec{A_2A_2^\top}{}\big)+n^{-2}\tr\big(\expec{A_3A_3^\top}{}\big).
\end{aligned}
\end{align}
Since $\{e_k\}_k$ is a martingale-difference sequence, using \cs inequality, and Lemma~\ref{lm:Snormbound} we get, 
\begin{align}
    n^{-2}\tr(\expec{A_1A_1^\top}{})\lesssim& n^{-2}\tr\expec{S_0^n\vt_0(S_0^n\vt_0)^\top}{}+n^{-2}\tr\expec{\textstyle\sum_{k=1}^n\left(\Id+S^n_k\right)\eta_ke_k\left(\sum_{k=1}^n\left(\Id+S^n_k\right)\eta_ke_k\right)^\top}{}\nonumber\\[2.5pt]
    \lesssim & n^{-2}+n^{-2}\tr\expec{\textstyle\sum_{k=1}^n\eta_k^2\left(\Id+S^n_k\right)e_k\left(\left(\Id+S^n_k\right)e_k\right)^\top}{}\nonumber\\[2.5pt]
    \lesssim & n^{-2}+n^{-2}\textstyle\sum_{k=1}^n\eta_k^2\expec{\norm{\left(\Id+S^n_k\right)e_k}_2^2}{}\nonumber\\[2.5pt]
    \lesssim & n^{-2}+n^{-2}\textstyle\sum_{k=1}^n\expec{\norm{e_k}_2^2}{}\nonumber\\[2.5pt]
    =&O(n^{-1}).\label{eq:a1a1bound}
\end{align}
Next, using Lemma~\ref{lm:Snormbound} and \eqref{eq:nunormbound} we hence get
\begin{align*}
    \tr(\expec{A_2A_2^\top}{})=\expec{\norm{A_2}_2^2}{}\leq \expec{(\textstyle\sum_{k=1}^n\norm{(\Id+S^n_k)}_2\eta_k\norm{\nu_k}_2)^2}{}\lesssim& n\textstyle\sum_{k=1}^n\expec{\norm{\nu_k}_2^2}{}\lesssim n\sum_{k=1}^n\eta_k^2.
\end{align*}
By our choice of $\eta_k$ and noting that $a>1/2$, we have $ n\sum_{k=1}^n\eta k^{-2a}\lesssim n$. Hence,
\begin{align}\label{eq:a2a2bound}
    n^{-2}\tr(\expec{A_2A_2^\top}{})=O(n^{-1}).
\end{align}
Next, using \eqref{eq:tildezetabound}, \eqref{eq:zetatelescopicsum}, and Lemma~\ref{lm:Snormbound} we get,
\begin{align*}
    \tr(\expec{A_3A_3^\top}{})
    =&\expec{\norm{\textstyle\sum_{k=1}^n\left(\Id+S^n_k\right)\eta_k\zeta_k}_2^2}{}\\[2.5pt]
    =&\expec{\norm{\textstyle\sum_{k=1}^n\left(\Id+S^n_k\right)\left(\tilde{\zeta}_k-\tilde{\zeta}_{k+1}\right)}_2^2}{}\\[2.5pt]
    =&\expec{\norm{-\tilde{\zeta}_{n+1}+S_1^n\tilde{\zeta}_1+\textstyle\sum_{k=1}^{n-1}\left(S_{k+1}^n-S_k^n\right)\tilde{\zeta}_{k+1}}_2^2}{}\\[2.5pt]
    =&\expec{\norm{-\tilde{\zeta}_{n+1}+S_1^n\tilde{\zeta}_1+\textstyle\sum_{k=1}^{n-1}\left(\eta_{k+1}Q(\Id+S_{k}^n)-\Id\right)\tilde{\zeta}_{k+1}}_2^2}{}\\[2.5pt]
    \lesssim & 1+\expec{\big(\textstyle\sum_{k=1}^{n-1}\norm{\left(\eta_{k+1}Q(\Id+S_{k}^n)-\Id\right)\tilde{\zeta}_{k+1}}_2\big)^2}{}\\[2.5pt]
    \lesssim & 1+\expec{(\textstyle\sum_{k=1}^{n-1}(\eta_{k+1}k^a+1)\norm{\tilde{\zeta}_{k+1}}_2)^2}{}\\[2.5pt]
    \lesssim & 1+n\textstyle\sum_{k=1}^{n-1}\eta_{k+1}^2{}\\[2.5pt]
    \lesssim & n.
\end{align*}
Hence, we have that  
\begin{align}
    n^{-2}\tr(\expec{A_3A_3^\top}{})=O(n^{-1}).\label{eq:a3a3bound}
\end{align}
Combining \eqref{eq:thetabarthetabardecomp}, \eqref{eq:a1a1bound},\eqref{eq:a2a2bound}, and \eqref{eq:a3a3bound} proves Lemma~\ref{lm:thetabarthetabarbound}.
\end{proof}
\begin{lemma}\label{lm:thetathetasigmadiffbound}
Let Assumptions~\ref{as:liangasA2}-\ref{as:liangasA4} be true. Then,
\begin{align}\label{eq:thetathetasigmadiffbound}
\begin{aligned}
    &\expec{\norm{(\textstyle\sum_{i=1}^nl_i)^{-1}\textstyle\sum_{i=1}^n(\sum_{k=t_i}^i\theta_k)(\sum_{k=t_i}^i\theta_k)^
    \top-\Sigma}_2}{}\\[2.5pt] 
    \lesssim &((d/M)^{1/4}+C)M^{-a\beta/2}+ \sqrt{d}M^{(1-\beta(1-a))/2}+M^{-\frac{a\beta}{4}}\sqrt{d\log M}
    +dM^{1-\beta(1-a)}.
\end{aligned}
\end{align}
\end{lemma}
\begin{proof}[Proof of Lemma~\ref{lm:thetathetasigmadiffbound}]
The proof proceeds by bounding the terms in the decomposition introduced in \eqref{eq:lincoverrordecompmain}.\\

\noindent \textbf{Bound on I:} Since, $\Sigma=Q^{-1}SQ^{-1}$, from \eqref{eq:SnSnerrorbound}, we have,
\begin{align}\label{eq:boundonI}
\begin{aligned}
    &\expec{\norm{(\textstyle\sum_{i=1}^n l_i)^{-1}\sum_{i=1}^n(Q^{-1}(\sum_{k=t_i}^i\xi_k(\theta_{k-1},x_k))(\sum_{k=t_i}^i\xi_k(\theta_{k-1},x_k))^
    \top Q^{-1}-\Sigma)}_2}{}\\[2.5pt]
    \lesssim &((d/M)^{1/4}+C)M^{-a\beta/2}+M^{-\frac{a\beta}{4}}\sqrt{d\log M}+\sqrt{d}M^{-1/2}.
\end{aligned}
\end{align}
\textbf{Bound on II:} Using Young's inequality we have,
\begin{align}\label{eq:IVVVIdef}
\expec{\norm{\phi_i\phi_i^\top}_2}{}=&\expec{\norm{\phi_i}_2^2}{}\nonumber\\[2.5pt]
   =& \expec{\norm{S_{t_i-1}^i\vt_{t_i-1}+\textstyle\sum_{k=t_i}^i\left(\eta_kS_k^i+\eta_k\Id-Q^{-1}\right)\xik}_2^2}{}\nonumber\\[2.5pt]
   =&\expec{\norm{S_{t_i-1}^i\vt_{t_i-1}+\textstyle\sum_{k=t_i}^i\left(\eta_kS_k^i+\eta_k\Id-Q^{-1}\right)(e_k+\nu_k+\zeta_k)}_2^2}{}\nonumber\\[2.5pt]
   \lesssim & \underbrace{\expec{\norm{S_{t_i-1}^i\vt_{t_i-1}+\textstyle\sum_{k=t_i}^i\left(\eta_kS_k^i+\eta_k\Id-Q^{-1}\right)e_k}_2^2}{}}_{\mathsf{IV}}
   +\underbrace{\expec{\norm{\textstyle\sum_{k=t_i}^i\left(\eta_kS_k^i+\eta_k\Id-Q^{-1}\right)\nu_k}_2^2}{}}_{\mathsf{V}}\nonumber\\[2.5pt]
   &+\underbrace{\expec{\norm{\textstyle\sum_{k=t_i}^i\left(\eta_kS_k^i+\eta_k\Id-Q^{-1}\right)\zeta_k}_2^2}{}}_{\mathsf{VI}}
\end{align}
To bound term \textsf{II}, we first bound terms \textsf{IV}, \textsf{V} and \textsf{VI} below.\\

\noindent \textit{Bound on \textsf{IV}:} Using Young's inequality we have,
\begin{align*}
    &\expec{\norm{S_{t_i-1}^i\theta_{t_i-1}+\textstyle\sum_{k=t_i}^i\left(\eta_kS_k^i+\eta_k\Id-Q^{-1}\right)e_k}_2^2}{}\\[2.5pt]
    \lesssim & \expec{\norm{S_{t_i-1}^i\theta_{t_i-1}}_2^2}{}+\expec{\norm{\textstyle\sum_{k=t_i}^i\left(\eta_kS_k^i+\eta_k\Id-Q^{-1}\right)e_k}_2^2}{}\\[2.5pt]
    \leq & \norm{S_{t_i-1}^i}_2^2\expec{\norm{\theta_{t_i-1}}_2^2}{}+\expec{\textstyle\sum_{k=t_i}^i\norm{\left(\eta_kS_k^i+\eta_k\Id-Q^{-1}\right)e_k}_2^2}{}\\[2.5pt]
    \lesssim & \norm{S_{t_i-1}^i}_2^2\expec{\norm{\theta_{t_i-1}}_2^2}{}+\expec{\textstyle\sum_{k=t_i}^i\left(\norm{\left(\eta_kS_k^i-Q^{-1}\right)}_2^2+\norm{\eta_k\Id}_2^2\right)\norm{e_k}_2^2}{}.
\end{align*}
Combining Lemma~\ref{lm:Snormbound}, and Lemma~\ref{lm:expecconviter} we get,
\begin{align}\label{eq:ivbound1}
    \norm{S_{t_i-1}^i}_2^2\expec{\norm{\vt_{t_i-1}}_2^2}{}\lesssim (t_i-1)^a.  
\end{align}
From Lemma D.2(3) of \cite{chen2020statistical}, and Lemma~\ref{lm:poisregular}, we have,
\begin{align}\label{eq:ivbound2}
    \expec{\textstyle\sum_{k=t_i}^i\norm{\left(\eta_kS_k^i-Q^{-1}\right)}_2^2\norm{e_k}_2^2}{}\lesssim i^a+l_it_i^{2a-2}. 
\end{align}
We also have,
\begin{align}\label{eq:ivbound3}
    \expec{\textstyle\sum_{k=t_i}^i\norm{\eta_k\Id}_2^2\norm{e_k}_2^2}{}\lesssim 
& \textstyle\sum_{k=t_i}^i k^{-2a}=\int_{t_i}^ik^{-2a}\mathrm{d}k=\frac{1}{2a-1}(t_i^{1-2a}-(t_i+l_i)^{1-2a})\nonumber\\[2.5pt]
\leq & \frac{1}{2a-1}\big(t_i^{1-2a}-t_i^{1-2a}\big(1-(2a-1){l_i}/{t_i}\big)\big) 
\leq l_it_i^{-2a}. 
\end{align}
For $t_i=a_m$, combining \eqref{eq:ivbound1}, \eqref{eq:ivbound2}, and \eqref{eq:ivbound3} we get,
\begin{align}\label{eq:boundoniv}
    \expec{\norm{S_{t_i-1}^i\vt_{t_i-1}+\textstyle\sum_{k=t_i}^i\left(\eta_kS_k^i+\eta_k\Id-Q^{-1}\right)e_k}_2^2}{}
     \lesssim i^a+l_ia_m^{2a-2}. \\
\end{align}
\noindent \textit{Bound on \textsf{V}:} Using Lemma~\ref{lm:poisregular}, and using $\textstyle\sum_{k=t_i}^i k^{-2a}\leq l_it_i^{-2a}$, we have that 
\begin{align}\label{eq:boundonv}
\begin{aligned}
    \expec{\norm{\textstyle\sum_{k=t_i}^i\left(\eta_kS_k^i+\eta_k\Id-Q^{-1}\right)\nu_k}_2^2}{}
    \lesssim &~l_i\textstyle\sum_{k=t_i}^i\expec{\norm{\left(\eta_kS_k^i+\eta_k\Id-Q^{-1}\right)}_2^2\norm{\nu_k}_2^2}{}\\[2.5pt]
    \lesssim &~l_i\textstyle\sum_{k=t_i}^i\expec{\left(\eta_k^2k^{2a}+\eta_k^2+1\right)\norm{\nu_k}_2^2}{}\\[2.5pt]
    \lesssim &~l_i\textstyle\sum_{k=t_i}^i\eta_k^2
    \lesssim  l_i^2t_i^{-2a}.
\end{aligned}
\end{align}

\noindent \textit{Bound on \textsf{VI}:} Using Lemma~\ref{lm:poisregular},
\begin{align}\label{eq:boundonvi}
    &\expec{\norm{(\textstyle\sum_{k=t_i}^i\left(\eta_kS_k^i+\eta_k\Id-Q^{-1}\right)\zeta_k)}_2^2}{}\nonumber\\[2.5pt]
    \lesssim & \expec{\norm{\textstyle\sum_{k=t_i}^i\eta_kS_k^i\zeta_k}_2^2}{}+\expec{\norm{\textstyle\sum_{k=t_i}^i\eta_k\zeta_k}_2^2}{}+\expec{\norm{\textstyle\sum_{k=t_i}^i\zeta_k}_2^2}{}\nonumber\\[2.5pt]
    = & \expec{\norm{\textstyle\sum_{k=t_i}^iS_k^i(\tzk-\tilde{\zeta}_{k+1})}_2^2}{}+\expec{\norm{\textstyle\sum_{k=t_i}^i(\tzk-\tilde{\zeta}_{k+1})}_2^2}{}+\expec{\norm{\textstyle\sum_{k=t_i}^i(\tzk-\tilde{\zeta}_{k+1})/\eta_k}_2^2}{}\nonumber\nonumber\\[2.5pt]
     = & \expec{\norm{S_{t_i}^i\tilde{\zeta}_{t_i}+\textstyle\sum_{k={t_i+1}}^{i}(S_{k}^i-S_{k-1}^i)\tzk}_2^2}{}+\expec{\norm{\tilde{\zeta}_{t_i}-\tilde{\zeta}_{i+1}}_2^2}{}\nonumber\\[2.5pt]
     &\qquad+\expec{\norm{({\tilde{\zeta}_{t_i}}/{\eta_{t_i}})-({\tilde{\zeta}_{i+1}}/{\eta_{i})}+\textstyle\sum_{k=t_i+1}^i\left({1}/{\eta_{k}}-{1}/{\eta_{k-1}}\right)\tzk}_2^2}{}\nonumber\\[2.5pt]
      \lesssim & C+\expec{\norm{\textstyle\sum_{k={t_i+1}}^{i}\eta_{k+1}Q(\Id+S_{k}^i)\tzk}_2^2}{}+\expec{\norm{\textstyle\sum_{k=t_i+1}^ik^{a-1}\tzk}_2^2}{}\nonumber\\[2.5pt]
      \leq & C+l_i\textstyle\sum_{k={t_i+1}}^{i}\expec{\norm{\eta_{k+1}Q(\Id+S_{k}^i)\tzk}_2^2}{}+l_i\textstyle\sum_{k=t_i+1}^i\expec{\norm{k^{a-1}\tzk}_2^2}{}\nonumber\\[2.5pt]
      \leq & C+l_i\textstyle\sum_{k={t_i+1}}^{i}k^{-2a}+l_i\textstyle\sum_{k=t_i+1}^ik^{-2}\nonumber\\[2.5pt]
      \lesssim & C+l_i^2t_i^{-2a}+l_i^2t_i^{-2}.
\end{align}
Combining \eqref{eq:boundoniv}, \eqref{eq:boundonv}, and \eqref{eq:boundonvi}, for $t_i=a_m$, we have,
\begin{align}
    \expec{\norm{\phi_i\phi_i^\top}_2}{}\leq \left(C+l_i^2a_m^{-2a}+l_i^2a_m^{-2}+i^a+l_ia_m^{2a-2}\right) \label{eq:expecphiphibound}\\
    \nonumber
\end{align}

Now we establish the bound on term II. Using \eqref{eq:expecphiphibound}, we have 
\begin{align*}
    \expec{\norm{\slinv\sum_{i=1}^n\phi_i\phi_i^\top}_2}{}\leq \slinv\sum_{i=1}^n\expec{\norm{\phi_i\phi_i^\top}_2}{}
\end{align*}
Note that we also have 
\begin{align*}
    &n \slinv\leq M^{1-\beta}, \qquad \slinv\sum_{i=1}^n(i^a+l_ia_M^{2a-2})\lesssim M^{1-\beta(1-a)},
\end{align*}
and
\begin{align*}
    \slinv\sum_{i=1}^nl_i^{2}a_m^{-2a}&=\big(\textstyle\sum_{m=1}^Mn_m^2\big)^{-1}\sum_{m=1}^M\sum_{i=a_m}^{a_{m+1}-1}l_i^{2}a_m^{-2a}\\[2.5pt]
    &=\big(\textstyle\sum_{m=1}^Mn_m^2\big)^{-1}\sum_{m=1}^Mn_m^{3}a_m^{-2a}\\[2.5pt]
    &\lesssim M^{-1-\beta(2a-1)},
\end{align*}
and
\begin{align*}
    \slinv\sum_{i=1}^nl_i^{2}a_m^{-2}&=\big(\textstyle\sum_{m=1}^Mn_m^2\big)^{-1}\textstyle\sum_{m=1}^M\sum_{i=a_m}^{a_{m+1}-1}l_i^{2}a_m^{-2}\\[2.5pt]
    &=\big(\textstyle\sum_{m=1}^Mn_m^2\big)^{-1}\sum_{m=1}^Mn_m^{3}a_m^{-2}\\[2.5pt]
    &\lesssim M^{-1-\beta},
\end{align*}
Hence, we have the following bound on term II:
 \begin{align}
    \expec{\norm{\slinv\sum_{i=1}^n\phi_i\phi_i^\top}_2}{}\lesssim M^{1-\beta(1-a)}.\label{eq:boundonII}
\end{align}

\noindent \textbf{Bound on III:} Using \cs inequality we get,
\begin{align}
    \expec{\norm{\slinv\sum_{i=1}^n\phi_i\psi_i^\top}_2}{}\leq\sqrt{\frac{\sum_{i=1}^n\expec{\norm{\phi_i\phi_i^\top}_2}{}}{\sum_{i=1}^nl_i}}\,\sqrt{\frac{\sum_{i=1}^n\expec{\norm{\psi_i\psi_i^\top}_2}{}}{\sum_{i=1}^nl_i}}. \label{eq:phipsics}
\end{align}
Since we have the bound on $\expec{\norm{\phi_i\phi_i^\top}_2}{}$ in \eqref{eq:expecphiphibound}, we only need to establish an bound on $\expec{\norm{\psi_i\psi_i^\top}_2}{}$. Towards that, we have 
\begin{align*}
    \expec{\norm{\psi_i\psi_i^\top}_2}{}=&\expec{\norm{(\textstyle\sum_{k=t_i}^iQ^{-1}\xi_k(\theta_{k-1},x_k))(\sum_{k=t_i}^iQ^{-1}\xi_k(\theta_{k-1},x_k))^\top}_2}{}\\[2.5pt]
    \leq & \expec{\norm{\textstyle\sum_{k=t_i}^iQ^{-1}\xi_k(\theta_{k-1},x_k)}_2^2}{}\\[2.5pt]
    \lesssim & \expec{\norm{\textstyle\sum_{k=t_i}^iQ^{-1}e_k}_2^2}{}+\expec{\norm{\textstyle\sum_{k=t_i}^iQ^{-1}\nu_k}_2^2}{}+\expec{\norm{\textstyle\sum_{k=t_i}^iQ^{-1}\zeta_k}_2^2}{}
\end{align*}
Since $e_k$ is a martingale-difference sequence, for any $i>j$ we have, 
\begin{align*}
    \expec{e_i^\top e_j}{}=\expec{\expec{e_i^\top e_j|\cF_{i-1}}{}}{}=\expec{\expec{e_i |\cF_{i-1}}{}^\top e_j}{}=0.
\end{align*}
Then,
\begin{align}
    \expec{\norm{\textstyle\sum_{k=t_i}^iQ^{-1}e_k}_2^2}{}\leq \textstyle\sum_{k=t_i}^i\expec{\norm{e_k}_2^2}{}\lesssim l_i. \label{eq:psipsicomp1}
\end{align}
From \eqref{eq:boundonv}, we have,
\begin{align}
    \expec{\norm{\textstyle\sum_{k=t_i}^iQ^{-1}\nu_k}_2^2}{}\lesssim l_i^2t_i^{-2a}.\label{eq:psipsicomp2}
\end{align}
For $t_i=a_m$, from \eqref{eq:boundonvi} we have,
\begin{align}
    \expec{\norm{\textstyle\sum_{k=t_i}^iQ^{-1}\zeta_k}_2^2}{}\lesssim l_i^2t_i^{-2}.\label{eq:psipsicomp3}
\end{align}
Combining \eqref{eq:psipsicomp1}, \eqref{eq:psipsicomp2}, and \eqref{eq:psipsicomp3}, we get that
\begin{align}
    \frac{\sum_{i=1}^n\expec{\norm{\psi_i\psi_i^\top}_2}{}}{\sum_{i=1}^nl_i}=O(1). \label{eq:psipsilinv}
\end{align}
From \eqref{eq:phipsics}, and \eqref{eq:psipsilinv}, we have,
\begin{align}
    \expec{\norm{\slinv\sum_{i=1}^n\phi_i\psi_i^\top}_2}{}=O\left(M^{(1-\beta(1-a))/2}\right).\label{eq:boundonIII}
\end{align}
Combining \eqref{eq:boundonI}, \eqref{eq:boundonII}, and \eqref{eq:boundonIII}, we get,
\begin{align*}
    &\expec{\norm{(\textstyle\sum_{i=1}^nl_i)^{-1}\textstyle\sum_{i=1}^n(\textstyle\sum_{k=t_i}^i\theta_k)(\textstyle\sum_{k=t_i}^i\theta_k)^
    \top-\Sigma}_2}{}\\[2.5pt]
    \lesssim &\big((d/M)^{1/4}+C\big)M^{-a\beta/2}+M^{-\frac{a\beta}{4}}\sqrt{d\log M}
    +dM^{1-\beta(1-a)}
    + \sqrt{d}M^{(1-\beta(1-a))/2}.
\end{align*}
\end{proof}
\begin{lemma}\label{lm:thetakthetabarbound}
Let Assumptions~\ref{as:liangasA2}-\ref{as:liangasA4} be true. Then,
\begin{align*}
    \expec{\norm{(\textstyle\sum_{i=1}^nl_i)^{-1}\sum_{i=1}^n(\sum_{k=t_i}^i\vt_k)l_i\bar{\vt}_n^\top}_2}{}\lesssim M^{-a\beta/2}. 
\end{align*}
\end{lemma}
\begin{proof}[Proof of Lemma~\ref{lm:thetakthetabarbound}]\label{pf:thetakthetabarbound}
Applying Cauchy-Schwarz inequality, 
\begin{align}
    \expec{\norm{(\textstyle\sum_{i=1}^nl_i)^{-1}\sum_{i=1}^n(\sum_{k=t_i}^i\vt_k)l_i\bar{\vt}_n^\top}_2}{}
    \leq \textstyle\sqrt{\frac{\expec{\norm{\sum_{i=1}^n(\sum_{k=t_i}^i\vt_k)(\sum_{k=t_i}^i\vt_k)^\top}_2}{}}{\sum_{i=1}^nl_i}}\, \sqrt{\frac{\expec{\norm{\sum_{i=1}^nl_i^2\bar{\vt}_n\bar{\vt}_n^\top}_2}{}}{\sum_{i=1}^nl_i}}. \label{eq:thetakthetabarboundcs}
\end{align}
Now plugging $\textstyle\vt_i=W^i_0\vt_0+\sum_{k=1}^iW_k^i\eta_k\xik$, using triangle inequality, and Lemma~\ref{lm:poisregular}, for $t_i=a_m$, we get
\begin{align}\label{eq:thetaithetakintermed}
\begin{aligned}
    &\expec{\norm{\textstyle \sum_{i=1}^n(\textstyle\sum_{k=t_i}^i\vt_k)(\textstyle\sum_{k=t_i}^i\vt_k)^\top}_2}{}\\[2.5pt]
    \leq &\expec{\textstyle\sum_{i=1}^n\norm{(\textstyle\sum_{k=t_i}^i\vt_k)(\textstyle\sum_{k=t_i}^i\vt_k)^\top}_2}{}\\[2.5pt]
    =&\expec{\textstyle\sum_{i=1}^n\norm{(\textstyle\sum_{k=a_m}^i(W^k_0\vt_0+\textstyle\sum_{p=1}^kW_p^k\eta_p\xi_p(\theta_{p-1},x_p)))(\textstyle\sum_{k=a_m}^i(W^k_0\vt_0+\textstyle\sum_{p=1}^kW_p^k\eta_p\xi_p(\theta_{p-1},x_p)))^\top}_2}{}\\[2.5pt]
     =&\expec{\textstyle\sum_{i=1}^n\norm{\textstyle\sum_{k=a_m}^i(W^k_0\vt_0+\textstyle\sum_{p=1}^kW_p^k\eta_p\xi_p(\theta_{p-1},x_p))}_2^2}{}\\[2.5pt]
    =&\textstyle\sum_{i=1}^n\expec{\norm{\textstyle\sum_{k=a_m}^iW^k_0\vt_0+\textstyle\sum_{p=1}^i(\textstyle\sum_{k=\max(p,a_m)}^iW_p^k)\eta_p\xi_p(\theta_{p-1},x_p)}_2^2}{}\\[2.5pt]
    = & \textstyle\sum_{i=1}^n\expec{\norm{\textstyle\sum_{k=a_m}^iW^k_0\vt_0+\textstyle\sum_{p=1}^i(\textstyle\sum_{k=\max(p,a_m)}^iW_p^k)\eta_p(e_p+\nu_p+\zeta_p)}^2_2}{}\\[2.5pt]
    \lesssim & \textstyle\sum_{i=1}^n\expec{\norm{\sum_{k=a_m}^iW^k_0\vt_0}^2_2}{}+\underbrace{\textstyle\sum_{i=1}^n\expec{\norm{\sum_{p=1}^i(\textstyle\sum_{k=\max(p,a_m)}^iW_p^k)\eta_pe_p}^2_2}{}}_{\mathsf{K}_1}\\[2.5pt]
    &+\underbrace{\textstyle\sum_{i=1}^n\expec{\norm{\textstyle\sum_{p=1}^i(\textstyle\sum_{k=\max(p,a_m)}^iW_p^k)\eta_p\nu_p}^2_2}{}}_{\mathsf{K}_2}+\underbrace{\textstyle\sum_{i=1}^n\expec{\norm{\textstyle\sum_{p=1}^i(\textstyle\sum_{k=\max(p,a_m)}^iW_p^k)\eta_p\zeta_p}^2_2}{}}_{\mathsf{K}_3}.
\end{aligned}
\end{align}
Using Lemma~\ref{lm:Wijbound} we get,
\begin{align}\label{eq:w0theta0bound}
\begin{aligned}
\textstyle    \sum_{i=1}^n\expec{\norm{\sum_{k=a_m}^iW^k_0\vt_0}_2^2}{}&\leq \textstyle\sum_{i=1}^n\expec{l_i\sum_{k=a_m}^i\norm{W^k_0}_2^2\norm{\vt_0}_2^2}{}\\[2.5pt]
\textstyle \lesssim \sum_{i=1}^nl_i\sum_{k=a_m}^i&\exp\big(-\textstyle\frac{2\gamma\eta}{1-a}k^{1-a}\big)\lesssim \sum_{i=1}^nl_i.
\end{aligned}
\end{align}
From \cite[Equation (87)]{zhu2021online}, we get,
\begin{align}
 \textstyle   \sum_{i=1}^n\expec{\norm{\sum_{p=1}^i(\sum_{k=\max(p,a_m)}^iW_p^k)\eta_pe_p}^2_2}{}\lesssim \sum_{i=1}^nl_i. \label{eq:wetaepbound}
\end{align}
Now we bound the rest of the terms of \eqref{eq:thetaithetakintermed}. Using Lemma~\ref{lm:poisregular}, we get
\begin{align*}
   &\expec{\norm{ \textstyle\sum_{p=1}^i( \textstyle\sum_{k=\max(p,a_m)}^iW_p^k)\eta_p\zeta_p}^2_2}{}\\[2.5pt]
    =&\expec{\norm{ \textstyle\sum_{p=1}^i( \textstyle\sum_{k=\max(p,a_m)}^iW_p^k)(\tilde{\zeta}_p-\tilde{\zeta}_{p+1})}^2_2}{}\\[2.5pt]
    =&\expec{\norm{( \textstyle\sum_{k=a_m}^iW_1^k)\tilde{\zeta}_1-W_i^i\tilde{\zeta}_{i+1}+ \textstyle\sum_{p=2}^i( \textstyle\sum_{k=\max(p,a_m)}^iW_p^k- \textstyle\sum_{k=\max(p-1,a_m)}^iW_{p-1}^k)\tilde{\zeta}_p}^2_2}{}\\[2.5pt]
    =&\expec{\norm{( \textstyle\sum_{k=a_m}^iW_1^k)\tilde{\zeta}_1-\tilde{\zeta}_{i+1}+ \textstyle\sum_{p=a_m+1}^i\tilde{\zeta}_p+ \textstyle\sum_{p=2}^{a_m-1}( \textstyle\sum_{k=a_m}^iW_p^k)\eta_pQ\tilde{\zeta}_p+ \textstyle\sum_{p=a_m}^i( \textstyle\sum_{k=p}^iW_p^k)\eta_pQ\tilde{\zeta}_p}^2_2}{}.
\end{align*}
Using \eqref{eq:tildezetabound}, and \eqref{eq:Wijbound}, we get
\begin{align}
    &\expec{\norm{\textstyle\sum_{p=2}^{a_m-1}\left(\sum_{k=a_m}^iW_p^k\right)\eta_pQ\tilde{\zeta}_p}_2^2}{}\nonumber\\[2.5pt]
    \leq& a_m \textstyle\sum_{p=2}^{a_m-1}\expec{\norm{(\sum_{k=a_m}^iW_p^k)\eta_pQ\tilde{\zeta}_p}_2^2}{}\nonumber\\[2.5pt]
    \lesssim & a_m \textstyle\sum_{p=2}^{a_m-1}\norm{(\textstyle\sum_{k=a_m}^iW_p^k)}_2^2\eta_p^4\nonumber\\[2.5pt]
    \lesssim & a_m\textstyle \sum_{p=2}^{a_m-1}\big(\textstyle \sum_{k=a_m}^i\exp(-\frac{\eta\gamma}{1-a}(k^{1-a}-(p+1)^{1-a}))\big)^2\eta_p^4\nonumber\\[2.5pt]
    \lesssim & a_m \textstyle \sum_{p=2}^{a_m-1}\big(\int_{a_m}^i\exp(-\frac{\eta\gamma}{1-a}k^{1-a}\big)\mathrm{d}k)^2\exp\big(\frac{2\eta\gamma}{1-a}(p+1)^{1-a}\big)\eta_p^4\nonumber\\[2.5pt]
    \lesssim & a_m\textstyle \sum_{p=2}^{a_m-1}\exp\big(-\frac{2\eta\gamma}{1-a}a_m^{1-a}\big)a_m^{2a}\exp\big(\frac{2\eta\gamma}{1-a}(p+1)^{1-a}\big)\eta_p^4\nonumber\\[2.5pt]
    \lesssim & a_m^{1+2a}\textstyle \exp\big(-\frac{2\eta\gamma}{1-a}a_m^{1-a}\big)\int_{2}^{a_m-1}\exp\big(\frac{2\eta\gamma}{1-a}(p+1)^{1-a}\big)p^{-4a}\mathrm{d}p\nonumber\\[2.5pt]
    \lesssim & a_m^{1+2a}\textstyle \exp\big(-\frac{2\eta\gamma}{1-a}a_m^{1-a}\big)\exp\big(\frac{2\eta\gamma}{1-a}a_m^{1-a}\big)a_m^{-3a}\nonumber\\[2.5pt]
    =&a_m^{1-a}. \label{eq:wetazetabound1}
\end{align} 
Similarly, we have
\begin{align}\label{eq:wetazetabound2}
\begin{aligned}
\expec{\norm{\textstyle\sum_{p=a_m}^i(\sum_{k=p}^iW_p^k)\eta_pQ\tilde{\zeta}_p}^2_2}{}    \lesssim &~l_i \textstyle\sum_{p=a_m}^{i}\exp\big(-\frac{2\eta\gamma}{1-a}p^{1-a}\big)p^{2a}\exp\left(\frac{2\eta\gamma}{1-a}(p+1)^{1-a}\right)p^{-4a}\\[2.5pt]
    \lesssim &~l_i \textstyle\sum_{p=a_m}^{i}p^{-2a}
    \lesssim  l_i^{2}a_m^{-2a},
\end{aligned}
\end{align}
and
\begin{align}
    \expec{\norm{(\textstyle\sum_{k=a_m}^iW_1^k)\tilde{\zeta}_1-\tilde{\zeta}_{i+1}+\sum_{p=a_m+1}^i\tilde{\zeta}_p}^2_2}{}\lesssim &~\expec{\norm{(\textstyle\sum_{k=a_m}^iW_1^k)\tilde{\zeta}_1-\tilde{\zeta}_{i+1}}^2_2}{}+\expec{\norm{\sum_{p=a_m+1}^i\tilde{\zeta}_p}^2_2}{}\nonumber\\[2.5pt]
    \lesssim &~\textstyle \exp\big(-\frac{2\eta\gamma}{1-a}a_m^{1-a}\big)a_m^{2a}+1+l_i^{2}a_m^{-2a}, \label{eq:wremainderbound}
\end{align}
and
\begin{align}
    \expec{\norm{\textstyle\sum_{p=1}^i(\sum_{k=\max(p,a_m)}^iW_p^k)\eta_p\nu_p}^2_2}{} \lesssim a_m^{1-a}+l_i^{2-2a}. \label{eq:wetanubound1}
\end{align}
Combining \eqref{eq:w0theta0bound}, \eqref{eq:wetaepbound}, \eqref{eq:wetazetabound1}, \eqref{eq:wetazetabound2}, \eqref{eq:wremainderbound}, and \eqref{eq:wetanubound1}, we get, 
\begin{align}\label{eq:thetaithetakbound}
\begin{aligned}
    \slinv\expec{\norm{\sum_{i=1}^n(\textstyle\sum_{k=t_i}^i\vt_k)(\sum_{k=t_i}^i\vt_k)^\top}_2}{}\lesssim & \slinv\sum_{m=1}^M\sum_{i=a_m}^{a_{m+1}-1}(a_m^{1-a}+l_i)\\[2.5pt]
    \lesssim & C+\slinv\sum_{m=1}^Mn_ma_m^{1-a}\\[2.5pt]
    \lesssim & M^{1-a\beta}.
\end{aligned}
\end{align}
Combining \eqref{eq:thetabarthetabarbound}, \eqref{eq:thetakthetabarboundcs}, and \eqref{eq:thetaithetakbound}, we get, 
\begin{align*}
    \expec{\norm{(\textstyle\sum_{i=1}^nl_i)^{-1}\sum_{i=1}^n(\sum_{k=t_i}^i\vt_k)l_i\bar{\vt}_n^\top}_2}{}\lesssim M^{-a\beta/2}.
\end{align*}
\end{proof}

\begin{proof}[Proof of Lemma~\ref{lm:esntildesndiffprelim}]\label{pf:esntildesndiffprelim}
Define $b_k=e_k-\te_k$. Note that $b_k$ is a martingale difference sequence as well. First we will establish a bound on $\expec{\|b_k\|_2^2}{}$ which will be crucial to establish a bound on $\expec{\norm{\tS_n-S_n}_2}{}$. Note that we have,
\begin{align*}
\expec{\|b_k\|_2^2}{}=&\expec{\|e_k-\te_k\|_2^2}{}\\[2.5pt]
    =&\expec{\|u(\theta_{k-1},x_{k})-u(\theta^*,x_k)-(P_{\theta_{k-1}} u(\theta_{k-1},x_{k-1})-P_{\theta^*} u(\theta^*,x_{k-1}))\|_2^2}{}\\[2.5pt]
    \lesssim & \expec{\|u(\theta_{k-1},x_{k})-u(\theta^*,x_k)\|_2^2}{}+\expec{\|(P_{\theta_{k-1}} u(\theta_{k-1},x_{k-1})-P_{\theta^*} u(\theta^*,x_{k-1}))\|_2^2}{}\\[2.5pt]
    \leq& \expec{\left(V(x_k)^2+V(x_{k-1})^2\right)\|\theta_{k-1}-\theta^*\|_2^{2}}{}. 
\end{align*}
By Lemma~\ref{lm:expecconviter}, we have
\begin{align}
    \expec{\|b_k\|_2^2}{}\lesssim \expec{\|\theta_{k-1}-\theta^*\|_2^{2}}{}\leq \eta_k\leq k^{-a}\label{eq:bkbound}
\end{align}
From (53) and (54) of \cite{zhu2021online} we have,
\begin{align}\label{eq:esntildesndiffprelim}
\begin{aligned}
    \expec{\norm{\tS_n-S_n}_2}{}\lesssim& \big((d/M)^{1/4}+C\big)\sqrt{\expec{\norm{(\textstyle\sum_{i=1}^nl_i)^{-1}\sum_{i=1}^n(\sum_{k=t_i}^ib_k)(\sum_{k=t_i}^ib_k)^\top}_2}{}}\\[2.5pt]
    &+\expec{\norm{(\textstyle\sum_{i=1}^nl_i)^{-1}\sum_{i=1}^n(\sum_{k=t_i}^ib_k)(\sum_{k=t_i}^ib_k)^\top}_2}{}
\end{aligned}    
\end{align}
Since $b_k$ is a martingale difference sequence, using \eqref{eq:bkbound} for $a_m\leq i<a_{m+1}$ we have,
\begin{align*}
    \expec{\norm{(\textstyle \sum_{i=1}^nl_i)^{-1}\sum_{i=1}^n(\sum_{k=t_i}^ib_k)(\sum_{k=t_i}^ib_k)^\top}_2}{}
    & \leq  (\textstyle\sum_{i=1}^nl_i)^{-1}\sum_{i=1}^n\expec{\tr(\sum_{k=t_i}^ib_k)(\sum_{k=t_i}^ib_k)^\top}{}\\[2.5pt]
    &=(\textstyle \sum_{i=1}^nl_i)^{-1}\sum_{i=1}^n\expec{\norm{\sum_{k=t_i}^ib_k}_2^2}{}\\[2.5pt]&
    =(\textstyle\sum_{i=1}^nl_i)^{-1}\sum_{i=1}^n\sum_{k=t_i}^i\expec{\norm{b_k}_2^2}{}\\[2.5pt]
    &\lesssim (\textstyle \sum_{i=1}^nl_i)^{-1}\sum_{i=1}^n\sum_{k=t_i}^ik^{-a}\\[2.5pt]
    &\leq (\textstyle \sum_{i=1}^nl_i)^{-1}\sum_{m=1}^M\sum_{i=a_m}^{a_{m+1}-1}l_ia_m^{-a}\\[2.5pt]
    &\leq (\textstyle\sum_{i=1}^nl_i)^{-1}\sum_{m=1}^Mn_m^2a_m^{-a}.
\end{align*}
The last inequality above follows from the fact that $\sum_{i=a_m}^{a_{m+1}-1}l_i \asymp n_m^2$. Choosing $a_m=\floor{Cm^\beta}$ where $\beta>1/(1-a)$, we have, $n_m\asymp m^{\beta-1}$, and,
\begin{align}
    (\textstyle \sum_{i=1}^nl_i)^{-1}\sum_{m=1}^Mn_m^2a_m^{-a}\lesssim M^{-a\beta}.\label{eq:expecbkbkbound}
\end{align}
Combining \eqref{eq:expecbkbkbound} with \eqref{eq:esntildesndiffprelim} we get,
\begin{align}
    \expec{\norm{\tS_n-S_n}_2}{}\lesssim& \big((d/M)^{1/4}+C\big)M^{-a\beta/2}.\label{eq:SntideSnerrorbound}
\end{align}
Combining \eqref{eq:SntideSnerrorbound} with \eqref{eq:SntideSnerrorbound}, we get,
\begin{align*}
    \expec{\norm{S_n-S}_2}{}\lesssim& \big((d/M)^{1/4}+C\big)M^{-a\beta/2}+M^{-\frac{a\beta}{4}}\sqrt{d\log M}+\sqrt{d}M^{-1/2}.
\end{align*}
\end{proof}
\section{Proof of Theorem~\ref{th:simain}}
The proof for under state-independent Markovian sampling follows that of the state-dependent case with the following modification. Recall the decomposition in~\eqref{eq:firstdecomp}. Lemma~\ref{lm:part2} remains unchanged in the state-independent Markovian data case, whereas, we have the following analog of Lemma~\ref{lm:part1}. Recall also that the result in Lemma~\ref{lm:part1} was proved by handling the terms from the decomposition in~\eqref{eq:lem21decomp}. Under state-independent Markovian sampling, Lemma~\ref{lm:thetabarthetabarbound}, and Lemma~\ref{lm:thetakthetabarbound} remain unchanged. Lemma~\ref{lm:thetathetasigmadiffboundsi} plays the analog of Lemma~\ref{lm:thetathetasigmadiffbound}.  


\begin{lemma}\label{lm:thetathetasigmadiffboundsi}
Let Assumption~\ref{as:strongcon}, Assumption~\ref{as:liangasA2}, Assumption~\ref{as:liangasA3si} and Assumption~\ref{as:liangasA4} be true. Then,
\begin{align*}
\begin{aligned}
    &\expec{\norm{(\textstyle\sum_{i=1}^nl_i)^{-1}\textstyle\sum_{i=1}^n(\sum_{k=t_i}^i\theta_k)(\sum_{k=t_i}^i\theta_k)^
    \top-\Sigma}_2}{}\\[2.5pt] 
    \lesssim &((d/M)^{1/4}+C)M^{-a\beta/2}+ \sqrt{d}M^{(1-\beta(1-a))/2}
    +dM^{1-\beta(1-a)}.
\end{aligned}
\end{align*}
\end{lemma}
Under state-dependent Markovian sampling, recall that Lemma~\ref{lm:snserrorboundmain} forms the key to prove Lemma~\ref{lm:thetathetasigmadiffbound}. Under state-independent Markovian sampling, we have the following analog of Lemma~\ref{lm:snserrorboundmain}. With this result in hand, the rest of the proof of Lemma~\ref{lm:thetathetasigmadiffboundsi} follows mutatis mutandis that of Lemma~\ref{lm:thetathetasigmadiffbound}. 
\begin{lemma}\label{lm:snserrorboundmainsi}
Let Assumption~\ref{as:strongcon}, Assumption~\ref{as:liangasA2}, Assumption~\ref{as:liangasA3si} and Assumption~\ref{as:liangasA4} be true. Then,
\begin{align*}
    \expec{\norm{S_n-S}_2}{}\lesssim& \big((d/M)^{1/4}+C\big)M^{-{a\beta}/{2}}+\sqrt{d}M^{-1/2},
\end{align*}
where
\begin{align*}
    S_n\coloneqq \textstyle \left(\sum_{i=1}^nl_i\right)^{-1}\sum_{i=1}^n(\sum_{k=t_i}^ie_k)(\sum_{k=t_i}^ie_k)^\top,
\end{align*}
$e_k=u(\theta_{k-1},x_{k})-Pu(\theta_{k-1},x_{k-1})$, and $S=\lim_{k\to\infty}\expec{e_ke_k^\top}{}$.
\end{lemma} 
\begin{proof}[Proof of Lemma~\ref{lm:snserrorboundmainsi}]\label{pf:snserrorboundmainsi}
Note that $S$ is defined as the covariance of random variables $\bar{e}_k=u(\theta^*,\hat{x}_k)-P u(\theta^*,\hat{x}_{k-1})$ where $\hat{x}_{k-1}\sim\pi$. To establish the convergence rate of $S_n$ to $S$ we introduce the following intermediate sequence,
\begin{align*}
    \te_k\coloneqq u(\theta^*,x_k)-P u(\theta^*,x_{k-1}),~~\text{where,}~~P u(\theta^*,x_{k-1})=\textstyle\int u(\theta^*,x)P (x_{k-1},x)\mathrm{d}x.
\end{align*}
Note that $\te_k$ is a martingale difference sequence with respect to the filtration $\cF_k$ as
\begin{align*}
    \expec{\te_k|\cF_{k-1}}{}=\expec{u(\theta_{k-1},x_k)|\cF_{k-1}}{}-P u(\theta^*,x_{k-1})=0.
\end{align*}
Similar to the proof of Lemma~\ref{lm:snserrorboundmain}, consider the following 
\begin{align*}
\textstyle    \tS_n=\left(\sum_{i=1}^nl_i\right)^{-1}\sum_{i=1}^n(\sum_{k=t_i}^i\te_k)(\sum_{k=t_i}^i\te_k)^\top.
\end{align*}
By the triangle inequality, we have,
\begin{align*}
    \expec{\|S_n-S\|_2}{}\leq \expec{\norm{S_n-\tS_n}_2}{}+ \expec{\|\tS_n-S\|_2}{}.
\end{align*}
We now bound the terms on the right hand side above displayed equation. \\

\textbf{Bound on $\expec{\norm{S_n-\tS_n}_2}{}$.} Define $b_k=e_k-\te_k$. Note that $b_k$ is a martingale difference sequence as well. First let us establish a bound on $\expec{\|b_k\|_2^2}{}$ which will be crucial to establish a bound on $\expec{\norm{S_n-\tS_n}_2}{}$.
Then,
\begin{align*}
\expec{\|b_k\|_2^2}{} =&\expec{\|e_k-\te_k\|_2^2}{}\\
    =&\expec{\|u(\theta_{k-1},x_{k})-u(\theta^*,x_k)-(P u(\theta_{k-1},x_{k-1})-P u(\theta^*,x_{k-1}))\|_2^2}{}\\
    \lesssim & \expec{\|u(\theta_{k-1},x_{k})-u(\theta^*,x_k)\|_2^2}{}+\expec{\|(P u(\theta_{k-1},x_{k-1})-P u(\theta^*,x_{k-1}))\|_2^2}{}\\
    \leq& \expec{\left(V(x_k)^2+V(x_{k-1})^2\right)\|\theta_{k-1}-\theta^*\|_2^{2}}{}. 
\end{align*}
Now, using Lemma~\ref{lm:expecconviter}, we have
\begin{align*}
    \expec{\|b_k\|_2^2}{}\lesssim \expec{\|\theta_{k-1}-\theta^*\|_2^{2}}{}\leq \eta_k\leq k^{-a}. 
\end{align*}
Similar to \eqref{eq:SntideSnerrorbound}, we have, 
\begin{align}\label{eq:SntideSnerrorboundsindep}
    \expec{\norm{S_n-\tS_n}_2}{}\lesssim \left((d/M)^{1/4}+C\right)M^{-a\beta/2}.
\end{align}
\textbf{Bound on $\expec{\|\tS_n-S\|_2}{}$.} Since $\te_k$ is a martingale difference sequence, we have 
\begin{align}
    \expec{\tS_n}{}=(\textstyle\sum_{i=1}^nl_i)^{-1}\sum_{i=1}^n\expec{(\sum_{k=t_i}^i\te_k)(\sum_{k=t_i}^i\te_k)^\top}{}=(\sum_{i=1}^nl_i)^{-1}\sum_{i=1}^n\sum_{k=t_i}^i\expec{\te_k\te_k^\top}{}.\label{eq:ekekequaltoSsi}
\end{align}
Now let us concentrate on the term $\expec{\te_k\te_k^\top}{}$. Note that this is function of $x_k$, and $x_{k-1}$. For convenience let us write $\te_k\te_k^\top=\mathcal{A}(x_{k-1},x_k)$. 
By Assumption~\ref{as:liangasA3si} the Markov chain $\{x_k\}_k$ is $V$-uniformly ergodic. Let $\Delta$ denote the joint distribution of $(x_{k-1},x_k)$ conditioned on $x_{0}$, and $\Delta_{\theta^*}$ is the joint distribution $(\hat x_{k-1},\hat x_k)$ where $\hat x_{k-1}\sim\pi_{\theta^*}$, and $\hat x_k$ is obtained by applying the transition operator $P$ on $\hat x_{k-1}$. Note that $\expec{\mathcal{A}(\hat x_{k-1},\hat x_k)}{}=S$. One can write $\Delta=PP(x'_{k-1})=P\pi_{\theta^*}+P(P(x'_{k-1})-\pi_{\theta^*})=\Delta_{\theta^*}+P(P(x'_{k-1})-\pi_{\theta^*})$. Then, 
\begin{align*}
    \norm{\Delta-\Delta_{\theta^*}}_{TV}
    \leq & CV(x_{0})\rho^{k},
\end{align*}
for some constant $C>0$, and $0<\rho<1$. 

Then we have, 
\begin{align}
    \expec{\mathcal{A}(x_{k-1},x_k)|x_{0}}{}=S+T_{5,k}, \label{eq:Asrelationsi}
\end{align}
where $\norm{T_{5,k}}_2\lesssim V(x_{0})\rho^{k}$. Taking expectation on both sides of \eqref{eq:Asrelationsi} with respect to $x_{0}$ we have,
\begin{align}
   \expec{\te_k\te_k^\top}{}= \expec{\mathcal{A}(x_{k-1},x_k)}{}=S+\expec{T_{5,k}}{}, \label{eq:Asrelationfinalsi}
\end{align}
where, $\expec{T_{5,k}}{}\lesssim \rho^{k}$. 
Then from \eqref{eq:ekekequaltoSsi}
\begin{align*}
    \expec{\tS_n}{}
    =(\textstyle\sum_{i=1}^nl_i)^{-1}\sum_{i=1}^n\sum_{k=t_i}^i\expec{\te_k\te_k^\top}{}
    = S + (\sum_{i=1}^nl_i)^{-1}\sum_{i=1}^n\sum_{k=t_i}^i\expec{T_{5,k}}{}.
\end{align*}
We also have, 
\begin{align}
    \textstyle(\sum_{i=1}^nl_i)^{-1}\sum_{i=1}^n\sum_{k=t_i}^i\expec{T_{5,k}}{}
    \lesssim & (\textstyle\sum_{i=1}^nl_i)^{-1}\sum_{i=1}^n\sum_{k=t_i}^i\exp\big(-k\log \big(\frac{1}{\rho}\big)\big)\nonumber\\
    \lesssim &\frac{\int_0^\infty\exp(-m^\beta)m^{2\beta}\mathrm{d}m}{\sum_{m=1}^Mn_m^2}\nonumber\\
    \lesssim & M^{-1-2\beta}. \label{eq:sumT5lboundsi}
\end{align}
As $\tS_n-S$ is a symmetric matrix, we have,
\begin{align}
    \expec{\|\tS_n-S\|_2}{}\leq \expec{\sqrt{\tr(\tS_n-S)^2}}{}\leq \sqrt{\tr(\expec{(\tS_n-S)^2}{})}\leq\sqrt{d\norm{\expec{(\tS_n-S)^2}{}}_2}\label{eq:snstr2normrelsi}
\end{align}
Now,
\begin{align*}
 \textstyle   \expec{(\tS_n-S)^2}{}=\expec{\tS_n^2}{}+2(\sum_{i=1}^nl_i)^{-1}\sum_{i=1}^n\sum_{k=t_i}^i\expec{T_{5,k}}{}S-S^2.
\end{align*}
Using, \eqref{eq:sumT5lboundsi}, we get, 
\begin{align*}
    \norm{2(\textstyle\sum_{i=1}^nl_i)^{-1}\sum_{i=1}^n\sum_{k=t_i}^i\expec{T_{5,k}}{}S}_2\lesssim M^{-1-2\beta}. 
\end{align*}
Now we will show that $\expec{\tS_n^2}{}-S^2$ is small. Similar to equation (45) in \cite{zhu2021online}, the following preliminary decomposition takes place for $\expec{\tS_n^2}{}$:
\begin{align*}
    \tS_n^2=(\textstyle\sum_{i=1}^nl_i)^{-2}(R_1+R_2), 
\end{align*}
where,
\begin{align*}
    R_1=\textstyle\sum_{m=1}^{M-1}\sum_{i=a_m}^{a_{m+1}-1}&\bigg[2\sum_{j=a_m}^{i-1}\sum_{a_m\leq p_1\neq p_2\leq j}\left(\te_{p_1}\te_{p_1}^\top \te_{p_1}\te_{p_2}^\top+\te_{p_1}\te_{p_1}^\top \te_{p_2}\te_{p_1}^\top\right) \\&\qquad\qquad \qquad\qquad\qquad+\textstyle \sum_{a_m\leq p_1\neq p_2\leq i}\left(\te_{p_1}\te_{p_1}^\top \te_{p_1}\te_{p_2}^\top+\te_{p_1}\te_{p_1}^\top \te_{p_2}\te_{p_1}^\top\right)\bigg]\\
    \quad+\textstyle\sum_{i=a_M}^{n}&\bigg[2\textstyle\sum_{j=a_M}^{i-1}\sum_{a_M\leq p_1\neq p_2\leq j}\left(\te_{p_1}\te_{p_1}^\top \te_{p_1}\te_{p_2}^\top+\te_{p_1}\te_{p_1}^\top \te_{p_2}\te_{p_1}^\top\right)\\&\qquad\qquad\qquad\qquad\qquad+\textstyle\sum_{a_M\leq p_1\neq p_2\leq i}\left(\te_{p_1}\te_{p_1}^\top \te_{p_1}\te_{p_2}^\top+\te_{p_1}\te_{p_1}^\top \te_{p_2}\te_{p_1}^\top\right)\bigg],
\end{align*}
and,
\begin{align*}
    R_2=\textstyle\sum_{i=1}^n\sum_{j=1}^n\sum_{p=t_i}^i\sum_{q=t_j}^j \te_p\te_p^\top \te_q\te_q^\top.
\end{align*}
Using equation (46) from \cite{zhu2021online}, we have,
\begin{align}
    \norm{\expec{R_1}{}}_2\lesssim M^{-1}.\label{eq:R1boundsi}
\end{align}
Similar to equation (47) in \cite{zhu2021online}, we have,
\begin{align}\label{eq:IIinitialdecompsi}
\begin{aligned}
    &\norm{\big(\sum_{i=1}^nl_i\big)^{-2}\expec{R_2}{}-S^2}_2\\ \lesssim &\big(\sum_{i=1}^{a_{M+1}-1}l_i\big)^{-2}\sum_{m=1}^M\sum_{k=1}^M\sum_{i=a_m}^{a_{m+1}-1}\sum_{j=a_k}^{a_{k+1}-1}\sum_{p=a_m}^i\sum_{q=a_k}^j\norm{\expec{\te_p\te_p^\top \te_q\te_q^\top}{}-S^2}_2.
\end{aligned}
\end{align}
Similar to state-dependent Markovian sampling setting, we decompose \eqref{eq:IIinitialdecompsi} into two cases: \begin{enumerate}
    \item $p$ and $q$ belong to either same block or neighboring blocks, i.e., $\abs{m-k}\leq 1$.
    \item $p$ and $q$ are at least $1$ block apart, i.e., $\abs{m-k}> 1$.
\end{enumerate}
\textbf{Case I. $\abs{m-k}\leq 1$.}
Similar to \eqref{eq:R3bound}, we get, 
\begin{align}
 \textstyle  (\sum_{i=1}^{a_{M+1}-1}l_i)^{-2} R_3\lesssim M^{-1}.\label{eq:R3boundsi}
\end{align}
\textbf{Case II. $\abs{m-k}> 1$.}
Let
\begin{align*}
    R_4\coloneqq \mathop{\sum_{m=1}^M\sum_{k=1}^M}\limits_{\abs{m-k}> 1}\sum_{i=a_m}^{a_{m+1}-1}\sum_{j=a_k}^{a_{k+1}-1}\sum_{p=a_m}^i\sum_{q=a_k}^j \norm{\expec{\te_p\te_p^\top \te_q\te_q^\top}{}-S^2}_2. 
\end{align*}
Let us assume $k\leq m-2$. Then, 
\begin{align*}
    \expec{\te_p\te_p^\top \te_q\te_q^\top|\cF_{a_{m-1}-1}}{}=\expec{\te_p\te_p^\top |\cF_{a_{m-1}-1}}{}\te_q\te_q^\top.
\end{align*}
Similar to \eqref{eq:Asrelationfinalsi}, we also have that
\begin{align*}
    \expec{\te_p\te_p^\top |\cF_{a_{m-1}-1}}{}=S+\expec{T_{5,p}|\cF_{a_{m-1}-1}}{}.
\end{align*}
Hence, we have, $\expec{\te_p\te_p^\top \te_q\te_q^\top}{}=S\expec{\te_q\te_q^\top}{}+\expec{\left(T_{5,p}\right)\te_q\te_q^\top}{}$. Similar to \eqref{eq:Asrelationfinalsi}, we have $\expec{\te_q\te_q^\top}{}=S+\expec{T_{5,q}}{}$. Then, 
\begin{align*}
    \norm{\expec{\te_p\te_p^\top \te_q\te_q^\top}{}-S^2}_2
    \leq&\norm{\expec{T_{5,p}\te_q\te_q^\top}{}}_2+\norm{S\expec{T_{5,q}}{}}_2\\
    \leq & \expec{\norm{T_{5,p}}_2\norm{\te_q\te_q^\top}_2}{}+\norm{S}_2\norm{\expec{T_{5,q}}{}}_2\\
    \lesssim & \sqrt{\expec{\norm{T_{5,p}}_2^2{}}{}\expec{\norm{\te_q\te_q^\top}_2^2}{}}+\exp(-q)\\
    \lesssim & \exp(-p)+\exp(-q).
\end{align*}
Then, using $a_m\asymp m^\beta$, and $n_m\asymp (m+1)^{\beta-1}$, we have, 
\begin{align*}
    R_4\lesssim &\mathop{\sum_{m=1}^M\sum_{k=1}^M}\limits_{\abs{m-k}> 1}\sum_{i=a_m}^{a_{m+1}-1}\sum_{j=a_k}^{a_{k+1}-1}\sum_{p=a_m}^i\sum_{q=a_k}^j \left(e^{-p}+e^{-q}\right)\\
    \lesssim & \mathop{\sum_{m=1}^M\sum_{k=1}^M}\limits_{\abs{m-k}> 1} \bigg(n_k^2\sum_{i=a_m}^{a_{m+1}-1}\sum_{p=a_m}^ie^{-p}+n_m^2\sum_{j=a_k}^{a_{k+1}-1}\sum_{q=a_k}^je^{-q}\bigg)\\
    \lesssim & \sum_{m=1}^M\sum_{k=1}^M\left(n_k^2n_m^2\left(e^{-a_m}+e^{-a_k}\right)\right).
\end{align*}
Using \eqref{eq:li2bound}, and \eqref{eq:sumT5lboundsi} we have,
\begin{align}\label{eq:IVboundeeeesi}
 \textstyle   (\sum_{i=1}^{a_{M+1}-1}l_i)^{-2}R_4\lesssim M^{-1-2\beta}.
\end{align}
So, combining \eqref{eq:snstr2normrelsi}, \eqref{eq:R1boundsi}, \eqref{eq:R3boundsi}, and \eqref{eq:IVboundeeeesi}, we get,
\begin{align}\label{eq:SntideSerrorboundsi}
    \expec{\|\tS_n-S\|_2}{}\lesssim \sqrt{{d}/{M}}.
\end{align}
Combining \eqref{eq:SntideSerrorboundsi} with \eqref{eq:SntideSnerrorboundsindep}, we get,
\begin{align*}
    \expec{\norm{S_n-S}_2}{}\lesssim& \big((d/M)^{1/4}+C\big)M^{-a\beta/2}+\sqrt{d/M}.
\end{align*}
\end{proof}
\input{stateindepexpts}

%% file: proofofthm21.tex
\section{Proof details for Theorem~\ref{th:mainthm}}\label{sec:proofoutline}
We now provide the proof of Theorem~\ref{th:mainthm}, highlighting the main differences from the $\iid$ case.
\begin{proof}[Proof of Theorem~\ref{th:mainthm}]\label{pf:mainthm}
Define the sequence $\{\vt_k\}_k$ that evolves according to a linearized version of the dynamics of $\{\theta_k\}_k$ guided by the same noise sequence $\{\xi_{k}(\theta_{k-1},x_{k})\}_k$ (as defined in~\eqref{eq:gradnoise}), i.e., 
\begin{align}\label{eq:linit}
&\vt_0=\theta_0
   & \vt_{k}=\vt_{k-1}-\eta_{k}(Q\vt_{k-1}+\xi_{k}(\theta_{k-1},x_{k}))=(I-\eta_{k}Q)\vt_{k-1}+\eta_k \xi_{k}(\theta_{k-1},x_{k}),
\end{align}
where $Q$ is a positive definite matrix. We will show that the covariance estimator $\hat{\Sigma}_n$ for the nonlinear case is close to $\tilde{\Sigma}_n$ when $Q=\nabla^2 f(\theta^*)$. Indeed, note that \eqref{eq:linit} corresponds to the case where the gradient $\nabla f(\vt_k)$ is a linear function of $\vt_k$, i.e.,  $\nabla f(\vt_k)=Q\vt_k$. 
Now, recalling the notations above~\eqref{eq:covestimator}, define $\tilde{\Sigma}_n$ as
\begin{align}\label{eq:sigmatilde}
    \tilde{\Sigma}_n=\frac{\sum_{i=1}^n\left(\sum_{k=t_i}^i\vt_k-l_i\bar{\vt}_n\right)\left(\sum_{k=t_i}^i\vt_k-l_i\bar{\vt}_n\right)^
    \top}{\sum_{i=1}^nl_i}.
\end{align}
By triangle inequality, we then have that 
\begin{align}\label{eq:firstdecomp}
        \mathbb{E}[\|\hat{\Sigma}_n-\Sigma\|_2]\leq \underset{\textsf{Lemma}~\ref{lm:part2}}{\colorboxed{black}{\mathbb{E}[\|\hat{\Sigma}_n-\tilde\Sigma\|_2]}} + \underset{\textsf{Lemma}~\ref{lm:part1}}{\colorboxed{black}{    \mathbb{E}[\|\tilde{\Sigma}_n-\Sigma\|_2]}}.
\end{align}
Theorem~\ref{th:mainthm} is then proved by invoking the results in Lemma~\ref{lm:part1} and Lemma~\ref{lm:part2}, which are discussed next.
\end{proof}

In Lemma~\ref{lm:part2}, we show that the covariance estimator $\hat{\Sigma}_n$ is close to $\tilde{\Sigma}_n$ in~\eqref{eq:sigmatilde}, when $Q=\nabla^2 f(\theta^*)$.
\begin{lemma}\label{lm:part2}
Under Assumption~\ref{as:strongcon}-\ref{as:liangasA4}, for $Q=\nabla^2 f(\theta^*)$, we have, $
        \mathbb{E}[\|\hat{\Sigma}_n-\tilde\Sigma\|_2]
\lesssim M^{-\frac12}.$
\end{lemma} 
\begin{proof}[Proof of Lemma~\ref{lm:part2}]

Let $\varphi_k=\theta_k-\vt_k$. Then we have,
\begin{align}\label{eq:varphidef}
    \varphi_{k+1}=(I-\eta_{k+1}\nabla^2 f(\theta^*))\varphi_k-\eta_{k+1}\left(\nabla f(\theta_k)-\nabla^2 f(\theta^*)(\theta_k-\theta^*)\right).
\end{align}
By Assumption~\ref{as:liangasA2} we have, $\norm{\nabla f(\theta_k)-\nabla^2 f(\theta^*)(\theta_k-\theta^*)}_2\leq \norm{\theta_k-\theta^*}_2^2$. Note that no explicit noise term exists in \eqref{eq:varphidef}. One main step of the proof is to show the following bound.
$$\norm{\textstyle\left(\sum_{i=1}^nl_i\right)^{-1}\textstyle\sum_{i=1}^n (\sum_{k=t_i}^i\varphi_k-l_i\bar{\varphi}_n)(\sum_{k=t_i}^i\varphi_k-l_i\bar{\varphi}_n)^
    \top}_2\lesssim M^{-1},$$
where $\bar{\varphi}_n=n^{-1}\sum_{k=1}^n\varphi_k$. To establish the above bound (and also the proof of Lemma~\ref{lm:part1}), one needs to establish the expected convergence of $\norm{\theta_k-\theta^*}_2^p$, $p=1,2,4$ under state-dependent Markovian data; see Lemma~\ref{lm:expecconviter}. Lemma~\ref{lm:expecconviter} is novel and could be of independent interest. Using Lemma~\ref{lm:expecconviter}, the proof of Lemma~\ref{lm:part2} follows using simple but tedious algebraic manipulations similar to \cite{zhu2021online}. 
\end{proof}

\begin{lemma}\label{lm:part1}
Let Assumption~\ref{as:liangasA2}-\ref{as:liangasA4} be true. Then, choosing $a_m=O(\floor{m^\beta})$,
we have
\begin{align*}
    \expec{\norm{\tilde{\Sigma}_n-\Sigma}_2}{}\lesssim & \big((d/M)^{1/4}+C\big)M^{-a\beta/2}+M^{-\frac{a\beta}{4}}\sqrt{d\log M}
    +d\big(M^{1-\beta(1-a)}+M^{(\beta-1)(1-2a)}\big)\\
    &+\sqrt{{d}}{M}^{-1/2}+ \sqrt{d}\big(M^{(1-\beta(1-a))/2}+M^{(\beta-1)(1/2-a)}\big).
\end{align*}
\end{lemma}

\begin{proof}[Proof of Lemma~\ref{lm:part1}]
Using triangle inequality, we have the following decomposition:
\begin{align}\label{eq:lem21decomp}
\begin{aligned}
    \expec{\norm{\tilde{\Sigma}_n-\Sigma}_2}{}\leq ~~~&\underset{\textsf{Lemma~\ref{lm:thetabarthetabarbound}}}{\colorboxed{black}{\expec{\norm{\left(\textstyle\sum_{i=1}^nl_i\right)^{-1}\textstyle\sum_{i=1}^nl_i^2\bar{\vt}_n\bar{\vt}_n^\top}_2}{}}}\\
    +&\underset{\textsf{Lemma~\ref{lm:snserrorboundmain} \& Lemma~\ref{lm:thetathetasigmadiffbound}}}{\colorboxed{black}{\expec{\norm{\left(\textstyle \sum_{i=1}^nl_i\right)^{-1}\textstyle\sum_{i=1}^n\big(\sum_{k=t_i}^i\vt_k\big)\big(\sum_{k=t_i}^i\vt_k\big)^
    \top-\Sigma}_2}{}}}\\
    +&\underset{\textsf{Lemma~\ref{lm:thetakthetabarbound}}}{\colorboxed{black}{2\expec{\norm{\left(\textstyle\sum_{i=1}^nl_i\right)^{-1}\textstyle\sum_{i=1}^n\big(\textstyle\sum_{k=t_i}^i\vt_k\big)l_i\bar{\vt}_n^\top}_2}{}}}.
\end{aligned}
\end{align}
The proof is completed by combining the results in Lemma~\ref{lm:thetabarthetabarbound}, Lemma~\ref{lm:snserrorboundmain}, Lemma~\ref{lm:thetathetasigmadiffbound}, and Lemma~\ref{lm:thetakthetabarbound}.
\end{proof}

Handling the second term in the right hand side of~\eqref{eq:lem21decomp} requires additional effort. We proceed by further decomposing that term as below:
\begin{align}
\begin{aligned}\label{eq:lincoverrordecompmain}
    &\expec{\norm{(\textstyle \sum_{i=1}^nl_i)^{-1}\sum_{i=1}^n(\sum_{k=t_i}^i\vt_k)(\sum_{k=t_i}^i\vt_k)^
    \top-\Sigma}_2}{}\\
    \lesssim & \underbrace{\expec{\norm{(\textstyle\sum_{i=1}^n l_i)^{-1}\sum_{i=1}^n(Q^{-1}(\sum_{k=t_i}^i\xi_k(\theta_{k-1},x_k))(\sum_{k=t_i}^i\xi_k(\theta_{k-1},x_k))^
    \top Q^{-1}-\Sigma)}_2}{}}_{\textsf{I}}\\
    &+\underbrace{\expec{\norm{\slinv\sum_{i=1}^n\phi_i\phi_i^\top}_2}{}}_{\textsf{II}}+ \underbrace{\expec{\norm{\slinv\sum_{i=1}^n\phi_i\psi_i^\top}_2}{}}_{\textsf{III}}, 
\end{aligned}
\end{align}
where $\psi_i=\sum_{k=t_i}^iQ^{-1}\xi_k(\theta_{k-1},x_k)$, and $\phi_i=S_{t_i-1}^i\vt_{t_i-1}+\sum_{k=t_i}^i\left(\eta_kS_k^i+\eta_k\Id-Q^{-1}\right)\xik$. A key step in handling term \textsf{I} is provided in Lemma~\ref{lm:snserrorboundmain}. Terms \textsf{II} and \textsf{III} are directly handled in Lemma~\ref{lm:thetathetasigmadiffbound} in Section~\ref{sec:proofsthm21}.

We now discuss various aspects of the proofs of  Lemma~\ref{lm:part2}, and Lemma~\ref{lm:part1} that require significantly different techniques than the $\iid$ sampling case.\\

\textbf{1. Gradient noise from Markovian sampling.} Compared to $\iid$ data, we encounter additional error terms throughout the proof when the data is being sampled from a Markov chain. As described in Lemma~\ref{lm:poisregular}, the gradient noise sequence $\{\xi_{k}(\theta_{k-1},x_{k})\}_k$ decomposes into three components: (i) a martingale difference sequence $\{e_k\}_k$, (ii) a sequence $\{\nu_k\}_k$ consisting of terms with small norm, and (iii) a sequence $\{\zeta_k\}_k$ which has a telescoping-sum like structure. 
\begin{lemma}[Lemma A.5 in \cite{liang2010trajectory}]\label{lm:poisregular}
Let Assumption~\ref{as:strongcon}, \ref{as:liangasA3}, and \ref{as:liangasA4} be true. Let $\mathcal{X}_0\subset\mathcal{X}$ be such that $\sup_{x\in\mathcal{X}_0}V(x)<\infty$ and that $\mathcal{K}_0\subset M_{C_0}$, where $M_{C_0}$ is as in Lemma~\ref{lm:lialemmaa2}. Then there exists $\mathbb{R}^d$-valued random processes $\{e_k\}_k$, $\{\nu_k\}_k$, and $\{\zeta_k\}_k$ defined on a probability sapce $(\Omega,\cF,\mathbb{P})$ such that the following decomposition holds
    \begin{align}\label{eq:decomp}
        \xi_{k+1}(\theta_k,x_{k+1})=e_k+\nu_k+\zeta_k,
    \end{align}
    where 
\begin{align}
    &e_k=u(\theta_{k-1},x_{k})-P_{\theta_{k-1}}u(\theta_{k-1},x_{k-1}),\nonumber\\
    &\nu_k=P_{\theta_{k}}u(\theta_k,x_{k})-P_{\theta_{k-1}}u(\theta_{k-1},x_k)+\frac{\eta_{k+1}-\eta_k}{\eta_k}P_{\theta_{k}}u(\theta_k,x_k),\nonumber\\
    &\tilde{\zeta}_k=\eta_{k}P_{\theta_{k-1}}u(\theta_{k-1},x_{k-1}) \label{eq:tildezetabound}\\
    &\zeta_k=\frac{1}{\eta_k}\left(\tilde{\zeta}_k-\tilde{\zeta}_{k+1}\right).\label{eq:zetatelescopicsum}
\end{align}
We also have the following observations: 
\begin{enumerate}[label=(\alph*)]
\item The sequence $\{e_k\}_k$ is a martingale difference sequence, and $\frac{1}{\sqrt{n}}\sum_{k=1}^ne_k\to N(0,S)$ in distribution, where $S=\lim_{k\to\infty}\expec{e_ke_k^\top}{}$.
\item The term $\frac{1}{\sqrt{k}}\sum_{i=1}^k\expec{\norm{\nu_i}_2}{}\to 0$, as $k\to\infty$.
\item The term $\expec{\norm{\sum_{i=1}^k\eta_i\zeta_i}_2}{}\to 0$, as $k\to\infty$.
\item From \cite[Equation (31)]{liang2010trajectory} we also have that,
\begin{align}\label{eq:nunormbound}
\expec{\|\nu_k\|_2}{}=O(\eta_{k}) \qquad\text{and}\qquad \expec{\norm{\textstyle \sum_{k=1}^n\eta_k\zeta_k}_2}{}=O(\eta_{n+1}).
\end{align}
\end{enumerate}
\end{lemma}

The terms $\{\nu_k\}_k$, and $\{\zeta_k\}_k$ which are not present in the $\iid$ sampling setting, lead to additional error terms in the analysis of covariance estimator \eqref{eq:lincoverrordecompmain}. Specifically the terms $A_2$ and $A_3$ in \eqref{eq:A1A2A3def}, the terms $V$ and $VI$ in \eqref{eq:IVVVIdef}, and $K_2$ and $K_3$ in \eqref{eq:thetaithetakintermed} are not present in the $\iid$ case. We show, with explicit rates, that these error terms converge to $0$ under Markovian sampling.\\

\textbf{2. Optimization bounds.} We prove the following result on the expected convergence rate of the SGD iterates $\{\theta_k\}_k$ \eqref{eq:sgd}. While such results are previously known for $\iid$ sampling case (see \citet[Lemma 3.2]{chen2020statistical}), our results below under the state-dependent Markovian sampling is novel and is of independent interest. 
\begin{lemma}\label{lm:expecconviter}
Let Assumptions~\ref{as:strongcon}-\ref{as:liangasA3} be true. Then, for the updates generated by Algorithm~\ref{alg:tsgd}, we have,
\begin{align*}
\expec{\|\theta_{k+1}-\theta^*\|_2}{}
    =O(\sqrt{\eta_{k+1}}).\quad
    \expec{\|\theta_{k+1}-\theta^*\|_2^2}{}
    =O(\eta_{k+1}).\quad
    \expec{\|\theta_{k+1}-\theta^*\|_2^4}{}= O(\eta_{k+1}^2).
\end{align*}
\end{lemma} 
We present here the main trick to prove the above lemma in the state-dependent Markovian sampling case, while deferring the detailed proof to Appendix~\ref{pf:theta2224bound}.
\begin{proof}[Outline of the proof ]
Recall that $\theta_k$ is generated by the following update. 
\begin{align*}
    \theta_{k+1}=\theta_{k}-\eta_{k+1}\nabla F(\theta_k,x_{k+1}).
\end{align*}
First, let us consider the following perturbed process generated from $\{\theta_k\}$.
\begin{align}\label{eq:perturbtheta}
    \theta'_k=\theta_k-\tilde{\zeta}_{k+1}. 
\end{align}
Then, using \eqref{eq:decomp}, we get,
\begin{align*}
    \theta'_{k+1}
    =&\theta'_k-\eta_{k+1}\left(\nabla f(\theta_k)+e_{k+1}+\nu_{k+1}\right). 
\end{align*}
Notice that this trick gets rid of the $\zeta_{k+1}$ component in the gradient noise decomposition~\eqref{eq:decomp}. But the true gradient is now evaluated at $\theta_k$ instead of $\theta_k'$. The error $\expec{\norm{\nabla f(\theta_k)-\nabla f(\theta_k')}_2}{}$ is small because, by Lemma~\ref{lm:poisregular}, $\expec{\norm{\theta_k'-\theta_k}_2}{}\leq \eta_{k+1}$, and $f$ has Lipschitz continuous gradient by Assumption~\ref{as:liangasA3}. After this modification, the additional error term $\nu_k$ still remains compared to the $\iid$ case. But by Lemma~\ref{lm:poisregular}, we have $\expec{\norm{\nu_k}_2}{}=O(\eta_k)$. For $p\geq 2$, we also have
\begin{align*}
    \expec{\|\theta_k-\theta^*\|_2^{p}}{}\leq 2^{p/2-1}\expec{\|\theta'_k-\theta^*\|_2^p}{}+2^{p/2-1}\expec{\|\tzt_{k+1}\|_2^p}{}\lesssim \expec{\|\theta'_k-\theta^*\|_2^p}{}+\eta_{k+1}^p.
\end{align*}
Then it suffices to prove the result for the sequence $\{\theta'_k\}_k$.
\end{proof}
\vspace{0.1in}

\textbf{3. Convergence of \textsf{I} in \eqref{eq:lincoverrordecompmain}.} A key step towards bounding term I in \eqref{eq:lincoverrordecompmain} is to show the following result.
\begin{lemma}\label{lm:snserrorboundmain}
Let Assumptions~\ref{as:strongcon}-\ref{as:liangasA4} be true. Then,
\begin{align}\label{eq:SnSnerrorbound}
    \expec{\norm{S_n-S}_2}{}\lesssim& \big((d/M)^{1/4}+C\big)M^{-\frac{a\beta}{2}}+M^{-\frac{a\beta}{4}}\sqrt{d\log M}+\sqrt{d}M^{-1/2},
\end{align}
where
\begin{align*}
    S_n\coloneqq \textstyle \left(\sum_{i=1}^nl_i\right)^{-1}\sum_{i=1}^n(\sum_{k=t_i}^ie_k)(\sum_{k=t_i}^ie_k)^\top,
\end{align*}
and $S=\lim_{k\to\infty}\expec{e_ke_k^\top}{}$ as defined in Lemma~\ref{lm:poisregular}.
\end{lemma} 
\begin{proof}[Proof of Lemma~\ref{lm:snserrorboundmain}]
Note that $S$ is defined as the covariance of random variables $\bar{e}_k=u(\theta^*,\hat{x}_k)-P_{\theta^*} u(\theta^*,\hat{x}_{k-1})$ where $\hat{x}_{k-1}\sim\pi_{\theta^*}$. To establish the convergence rate of $S_n$ to $S$ we introduce two intermediate sequences.
\begin{align}
    \te_k\coloneqq u(\theta^*,x_k)-P_{\theta_{k-1}} u(\theta^*,x_{k-1}),\label{eq:tildekdef}\\
    e'_k\coloneqq u(\theta^*,x_k)-P_{\theta^*} u(\theta^*,x_{k-1}),\label{eq:ekprimedef}
\end{align}
where 
$$P_{\theta_{k-1}} u(\theta^*,x_{k-1})=\textstyle\int u(\theta^*,x)P_{\theta_{k-1}} (x_{k-1},x)\mathrm{d}x,$$ and 
$$P_{\theta^*} u(\theta^*,x_{k-1})=\textstyle\int u(\theta^*,x)P_{\theta^*} (x_{k-1},x)\mathrm{d}x.$$ 
Note that $\te_k$ is a martingale difference sequence with respect to the filtration $\cF_k$ as
\begin{align*}
    \expec{\te_k|\cF_{k-1}}{}=\expec{u(\theta_{k-1},x_k)|\cF_{k-1}}{}-P_{\theta^*} u(\theta^*,x_{k-1})=0.
\end{align*}
Consider the following 
\begin{align*}
    \tS_n=\textstyle \left(\sum_{i=1}^nl_i\right)^{-1}\sum_{i=1}^n\left(\sum_{k=t_i}^i\te_k\right)\left(\sum_{k=t_i}^i\te_k\right)^\top.
\end{align*}
By triangle inequality we have,
\begin{align*}
  \expec{\|S_n-S\|_2}{}\leq  \underset{\textsf{Lemma~\ref{lm:esntildesndiffprelim}}}{\colorboxed{black}{\expec{\norm{S_n-\tS_n}_2}{}}}+\underset{\textsf{Lemma~\ref{lm:SntideSerrorbound}}}{\colorboxed{black}{ \expec{\|\tS_n-S\|_2}{}}}.
\end{align*}
Lemma~\ref{lm:snserrorboundmain} follows by combining the bounds in Lemma~\ref{lm:SntideSerrorbound} and Lemma~\ref{lm:esntildesndiffprelim} introduced next. 
\end{proof}

\begin{lemma}\label{lm:esntildesndiffprelim}
Let Assumptions~\ref{as:liangasA2}-\ref{as:liangasA4} be true. Then, 
\begin{align*}
    \expec{\norm{S_n-\tS_n}_2}{}\lesssim& \left((d/M)^{1/4}+C\right)M^{-a\beta/2}.
\end{align*}
\end{lemma}
\noindent We defer the proof of Lemma~\ref{lm:esntildesndiffprelim} to Appendix~\ref{pf:esntildesndiffprelim} as its proof follows by algebraic manipulations.

\begin{lemma}\label{lm:SntideSerrorbound}
    Let Assumptions~\ref{as:liangasA2}-\ref{as:liangasA4} be true. Then,
    \begin{align*}
    \expec{\|\tS_n-S\|_2}{}\leq M^{-\frac{a\beta}{4}}\sqrt{d\log M}+\sqrt{d}M^{-1/2}+M^{-\frac{a\beta}{2}}\sqrt{d\log M}.
\end{align*}
\end{lemma}
We provide the proof of Lemma~\ref{lm:SntideSerrorbound} here as this is, besides Lemma~\ref{lm:expecconviter}, one of the main pillars of the proof of Theorem~\ref{th:mainthm} and is quite different from the $\iid$ data setting.
\begin{proof}[Proof of Lemma~\ref{lm:SntideSerrorbound}]
Unlike under $\iid$ sampling, $\tS_n$ is not an unbiased estimator of $S$ under Markovian sampling. So first we establish a bound on the bias term.  Since $\te_k$ is a martingale difference sequence, we have 
\begin{align}
    \expec{\tS_n}{}=&\textstyle \left(\sum_{i=1}^nl_i\right)^{-1}\sum_{i=1}^n\expec{(\sum_{k=t_i}^i\te_k)(\sum_{k=t_i}^i\te_k)^\top}{}\nonumber\\
    =&\textstyle \left(\sum_{i=1}^nl_i\right)^{-1}\sum_{i=1}^n\sum_{k=t_i}^i\expec{\te_k\te_k^\top}{}.\label{eq:ekekequaltoS}
\end{align}
Now from \eqref{eq:tildekdef}, and \eqref{eq:ekprimedef}, we have,
\begin{align}\label{eq:t1t2t3decomp}
\begin{aligned}
    \expec{\te_k\te_k^\top}{}=&\expec{\left(e'_k+(P_{\theta^*} -P_{\theta_{k-1}}) u(\theta^*,x_{k-1})\right)\left(e'_k+(P_{\theta^*} -P_{\theta_{k-1}}) u(\theta^*,x_{k-1})\right)^\top}{}\\
    =&\expec{e'_k{e'_k}^\top}{}+\expec{\underbrace{(P_{\theta^*} -P_{\theta_{k-1}}) u(\theta^*,x_{k-1}){e'_k}^\top}_{\mathsf{T_{1,k}}}}{}+\expec{\underbrace{e'_k(P_{\theta^*} -P_{\theta_{k-1}}) u(\theta^*,x_{k-1})^\top}_{\mathsf{T_{2,k}}}}{}\\
    &+\expec{\underbrace{(P_{\theta^*} -P_{\theta_{k-1}}) u(\theta^*,x_{k-1})(P_{\theta^*} -P_{\theta_{k-1}}) u(\theta^*,x_{k-1})^\top}_{\mathsf{T_{3,k}}}}{}.
    \end{aligned}
\end{align}
Now, observe that
\begin{align*}
    \expec{T_{1,k}}{}=&\expec{(P_{\theta^*} -P_{\theta_{k-1}}) u(\theta^*,x_{k-1}){e'_k}^\top}{}\\
    =&\expec{\expec{(P_{\theta^*} -P_{\theta_{k-1}}) u(\theta^*,x_{k-1}){e'_k}^\top|\cF_{k-1}}{}}{}\\
    =&\expec{(P_{\theta^*} -P_{\theta_{k-1}}) u(\theta^*,x_{k-1})(\expec{u(\theta^*,x_k)|\cF_{k-1}}{}-P_{\theta^*} u(\theta^*,x_{k-1}))^\top}{}\\
    =&-\expec{T_{3,k}}{}.
\end{align*}
Similarly, $\expec{T_{2,k}}{}=-\expec{T_{3,k}}{}$. Hence,
\begin{align}
    \expec{\te_k\te_k^\top}{}
    =\expec{e'_k{e'_k}^\top}{}-\expec{T_{3,k}}{}.\label{eq:ektildeekprimerel}
\end{align}
Note that, using $\expec{T_{3,k}}{}\succcurlyeq 0$, Assumption \ref{as:liangasA3}, and \cs inequality we have 
\begin{align}\label{eq:T3kbound}
    \expec{\norm{T_{3,k}}_2}{}= &\expec{\norm{(P_{\theta^*} -P_{\theta_{k-1}}) u(\theta^*,x_{k-1})(P_{\theta^*} -P_{\theta_{k-1}}) u(\theta^*,x_{k-1})^\top}_2}{}\nonumber\\
    =&\expec{\norm{(P_{\theta^*} -P_{\theta_{k-1}}) u(\theta^*,x_{k-1})}_2^2}{}\nonumber\\
    \leq &\expec{\norm{\theta_{k-1}-\theta^*}_2^2 V(x_{k-1})^2}{}\nonumber\\
    \leq & \sqrt{\expec{\norm{\theta_{k-1}-\theta^*}_2^4}{}\expec{V(x_{k-1})^4}{}} \,\,\lesssim  \eta_k. 
\end{align}
Now let us look at $\expec{e'_k{e'_k}^\top}{}$. Note that this is a function of $x_k$, and $x_{k-1}$. For convenience let us write $e'_k{e'_k}^\top=\mathcal{A}(x_{k-1},x_k)$. Here we make two observations. Firstly, the distance $\norm{\theta_{k+1}-\theta_{k}}_2$ between two consecutive iterates is of the order of the step-size $\eta_{k+1}$, and secondly, $\expec{\norm{\theta_k-\theta^*}_2}{}=O(\sqrt{\eta_k})$. Combining these two facts, we deduce that for small enough $b_0$, the expectation of the variables $\mathcal{A}(x_{k-1},x_k)$, and $\mathcal{A}(x_{k-1}',x_k')$ are close, where,
\begin{align*}
    &\mathcal{A}(x_{k-1},x_k)\sim P_{\theta_0}(x_0)P_{\theta_1}(x_1|x_0)\cdots P_{\theta_{k-1}}(x_{k}|x_{k-1})\\
    &\mathcal{A}(x_{k-1}',x_k')\sim P_{\theta_0}(x_0)P_{\theta_1}(x_1|x_0)\cdots P_{\theta_{k-b_0-1}}(x_{k-b_0}|x_{k-b_0-1})P_{\theta^*}(x_{k-b_0+1}'|x_{k-b_0})\cdots P_{\theta^*}(x_{k}'|x_{k-1}').
\end{align*}
 But, by Assumption~\ref{as:liangasA3}, we have that for a fixed $\theta$ the chain mixes exponentially fast to $\pi_\theta$. Now, for a proper choice of $b_0$, $\expec{\mathcal{A}(x_{k-1}',x_k')}{}$ is close to $\expec{\mathcal{A}(\hat{x}_{k-1},\hat{x}_k)}{}$ where $\hat{x}_{k-1}\sim\pi_{\theta^*}$, and $\hat{x}_k$ is obtained from $\hat{x}_{k-1}$ after transition according to the transition operator $P_{\theta^*}$. This implies $\expec{\mathcal{A}(\hat{x}_{k-1},\hat{x}_k)}{}=S$. 

Now let us define another sequence $\{x'_i\}$ as follows: $x'_i=	x_i$  for $ i=1,2,\cdots,k-b_0$, and $x'_{k-b_0+j}$, $j=1,\cdots,b_0$ are obtained after applying transition operator $P_{\theta^*}$ to $x'_{k-b_0}$ $j$ times. Let $\hat e_k\coloneqq u(\theta^*,x'_k)-P_{\theta^*}u(\theta^*,x'_{k-1})$.

For convenience, we introduce the following notation 
$$
P_{\theta^*}^{b_0} P_{\theta_{k-b_0-1}}\cdots P_0 \coloneqq P_{\theta^*}(x_{k-1},x_k)P_{\theta^*}(x_{k-2},x_{k-1})\cdots P_{\theta_{k-b_0-1}}(x_{k-b_0-1},x_{k-b_0})\cdots P_0(x_0).
$$
Then, we have
\begin{align}
    &\expec{\hat e_k\hat e_k^\top}{}=\textstyle \int \mathcal{A}(x_{k-1},x_k)P_{\theta^*}^{b_0} P_{\theta_{k-b_0-1}}\cdots P_0 \mathrm{d}x_k\, \mathrm{d}x_{k-1}\cdots \mathrm{d}x_0\nonumber\\
    &\expec{e_k' {e'_k}^\top}{}=\textstyle \int \mathcal{A}(x_{k-1},x_k)P_{\theta_{k-1}}P_{\theta_{k-2}}\cdots P_0\, \mathrm{d}x_k\,\mathrm{d}x_{k-1}\cdots \mathrm{d}x_0.\label{eq:hatekdef}
\end{align}
Hence, we can write 
\begin{align}\label{eq:ekhatekorimediff}
   \expec{\hat e_k\hat e_k^\top}{}=\expec{e_k' {e'_k}^\top}{}+\sum_{j=1}^{b_0}\textstyle \int \mathcal{A}(x_{k-1},x_k)P_{\theta^*}^{b_0-j}(P_{\theta_{k-(b_0-j)}}-P_{\theta^*})P_{\theta_{k-(b_0-j)-1}}\cdots P_0 \mathrm{d}x_k\mathrm{d}x_{k-1}\cdots \mathrm{d}x_0.
\end{align}
Now note that by Assumption~\ref{as:liangasA3}(c) we have,
\begin{align}\label{eq:totalvariation}
    \norm{P_{\theta^*}^{b_0-j}P_{\theta_{k-(b_0-j)}}P_{\theta_{k-(b_0-j)-1}}\cdots P_0-P_{\theta^*}^{b_0-j+1}P_{\theta_{k-(b_0-j)-1}}\cdots P_0}_{TV}\lesssim \sqrt{\eta_{k-(b_0-j)}}.
\end{align}
Using \eqref{eq:totalvariation}, we get,
\begin{align*}
    \textstyle\int \mathcal{A}(x_{k-1},x_k)P_{\theta^*}^{b_0-j}(P_{\theta_{k-(b_0-j)}}-P_{\theta^*})P_{\theta_{k-(b_0-j)-1}}\cdots P_0 \mathrm{d}x_k\mathrm{d}x_{k-1}\cdots \mathrm{d}x_0\lesssim \sqrt{\eta_{k-(b_0-j)}}.
\end{align*}
From \eqref{eq:ekhatekorimediff}, we get,
\begin{align}\label{eq:ekprimeekhatrel}
    \expec{e_k' {e'_k}^\top}{}=\expec{\hat e_k\hat e_k^\top}{}-T_{4,k}, 
\end{align}
where 
\begin{align}\label{eq:T4normbound}
    \norm{T_{4,k}}_2\lesssim \sum_{j=0}^{b_0-1}\sqrt{\eta_{k-j}}.
\end{align}
By Assumption~\ref{as:liangasA3}, one has that for each $\theta$, the Markov chain $\{x_k\}_k$ is $V$-uniformly ergodic. Let $\Delta$ denote the joint distribution of $(x'_{k-1},x'_k)$ conditioned on $x_{k-b_0}$, and $\Delta_{\theta^*}$ is the joint distribution $(\hat x_{k-1},\hat x_k)$ where $\hat x_{k-1}\sim\pi_{\theta^*}$, and $\hat x_k$ is obtained by applying the transition operator $P_{\theta^*}$ on $\hat x_{k-1}$. Note that $\expec{\mathcal{A}(\hat x_{k-1},\hat x_k)}{}=S$. One can write $\Delta=P_{\theta^*}P(x'_{k-1})=P_{\theta^*}\pi_{\theta^*}+P_{\theta^*}(P(x'_{k-1})-\pi_{\theta^*})=\Delta_{\theta^*}+P_{\theta^*}(P(x'_{k-1})-\pi_{\theta^*})$. Then, 
\begin{align*}
    \norm{\Delta-\Delta_{\theta^*}}_{TV}=& \norm{P_{\theta^*}(P(x'_{k-1})-\pi_{\theta^*})}_{TV}\\
    =&\textstyle \int\left(\int P_{\theta^*}(x'_k|x'_{k-1})\mathrm{d}x_k\right)(P(x'_{k-1})-\pi_{\theta^*(x'_{k-1})})\mathrm{d}x'_{k-1}\\
    \leq & CV(x_{k-b_0})\rho^{b_0},
\end{align*}
for some constant $C>0$, and $0<\rho<1$. Then we have, 
\begin{align}\label{eq:Asrelation}
    \expec{\mathcal{A}(x'_{k-1},x'_k)|x_{k-b_0}}{}=S+T_{5,k}, 
\end{align}
where $\norm{T_{5,k}}_2\lesssim V(x_{k-b_0})\rho^{b_0}$. Taking expectation on both sides of \eqref{eq:Asrelation} with respect to $x_{k-b_0}$ we have,
\begin{align}\label{eq:Asrelationfinal}
    \expec{\mathcal{A}(x'_{k-1},x'_k)}{}=S+\expec{T_{5,k}}{}, 
\end{align}
where, $\expec{\norm{T_{5,k}}_2}{}\lesssim \rho^{b_0}$. Choosing $b_0=\ceil{\log \eta_k/\log \rho}$, and from \eqref{eq:T4normbound}, we get, 
\begin{align}\label{eq:T4T5bound}
    {\norm{T_{4,k}}_2}\lesssim \eta_k\log k, \quad \expec{\norm{T_{5,k}}_2}{}\lesssim \eta_k. 
\end{align}
Combining \eqref{eq:ektildeekprimerel}, \eqref{eq:ekprimeekhatrel}, and \eqref{eq:Asrelationfinal}, we get,
\begin{align}\label{eq:ekekdecomp}
    \expec{\te_k\te_k^\top}{}=S-\expec{T_{3,k}}{}-T_{4,k}+\expec{T_{5,k}}{}.
\end{align}
Then, from \eqref{eq:ekekequaltoS}, we get, $\expec{\tS_n}{}
    =(\textstyle \sum_{i=1}^nl_i)^{-1}\sum_{i=1}^n\sum_{k=t_i}^i\expec{\te_k\te_k^\top}{} =S+T_6,$ where we have $T_6=(\textstyle \sum_{i=1}^nl_i)^{-1}\textstyle\sum_{i=1}^n\sum_{k=t_i}^i(\expec{T_{5,k}}{}-\expec{T_{3,k}}{}-T_{4,k})$. Now, using \eqref{eq:T3kbound}, and \eqref{eq:T4T5bound} we have, 
\begin{align*}
    \norm{T_6}_2\leq & (\textstyle\sum_{i=1}^nl_i)^{-1}\sum_{i=1}^n\sum_{k=t_i}^i\norm{\expec{T_{5,k}}{}-\expec{T_{3,k}}{}-T_{4,k}}_2\\
    \lesssim & (\textstyle\sum_{i=1}^nl_i)^{-1}\sum_{i=1}^n\sum_{k=t_i}^i \eta_k \log k \\
    \lesssim &(\textstyle\sum_{i=1}^nl_i)^{-1}\sum_{i=1}^n\sum_{k=t_i}^ik^{-a}\log k\\
    \leq &\log a_M(\textstyle\sum_{i=1}^nl_i)^{-1}\sum_{m=1}^M\sum_{i=a_m}^{a_{m+1}-1}l_ia_m^{-a}\\
    \leq &\log a_M(\textstyle\sum_{i=1}^nl_i)^{-1}\sum_{m=1}^Mn_m^2a_m^{-a}.
\end{align*}
The last inequality follows from the fact that $    \sum_{i=a_m}^{a_{m+1}-1}l_i \asymp n_m^2$. Choosing $a_m=\floor{Cm^\beta}$ where $\beta>1/(1-a)$, we have, $n_m\asymp m^{\beta-1}$, and,
\begin{align}\label{eq:expecbkbkboundmain}
\textstyle    \log a_M(\sum_{i=1}^nl_i)^{-1}\sum_{m=1}^Mn_m^2a_m^{-a}\lesssim M^{-a\beta}\log M.
\end{align}

As $\tS_n-S$ is a symmetric matrix, we have,
\begin{align}\label{eq:snstr2normrel}
   \expec{\|\tS_n-S\|_2}{}\leq \mathbb{E}\big[\textstyle \sqrt{\tr(\tS_n-S)^2}\big]\leq \sqrt{\tr(\expec{(\tS_n-S)^2}{})}\leq\sqrt{d\norm{\expec{(\tS_n-S)^2}{}}_2}.
\end{align}
Now, note that $
    \expec{(\tS_n-S)^2}{}=\expec{\tS_n^2}{}+2T_6S-S^2$. Using \eqref{eq:expecbkbkboundmain}, we hence have, 
\begin{align}\label{eq:sn2decompmid}
    \norm{2T_6S}_2\lesssim M^{-a\beta}\log M. 
\end{align}
Now we will show that $\expec{\tS_n^2}{}-S^2$ is small. First note that similar to Equation (45) in \cite{zhu2021online}, we have that $    \tS_n^2=\left(\textstyle\sum_{i=1}^nl_i\right)^{-2}(R_1+R_2)$, where
\begin{align*}
    R_1=\textstyle \sum_{m=1}^{M-1}\sum_{i=a_m}^{a_{m+1}-1}\bigg[&2\textstyle\sum_{j=a_m}^{i-1}\sum_{a_m\leq p_1\neq p_2\leq j}\big(\te_{p_1}\te_{p_1}^\top \te_{p_1}\te_{p_2}^\top+\te_{p_1}\te_{p_1}^\top \te_{p_2}\te_{p_1}^\top\big)\\
    &\qquad\qquad+\textstyle\sum_{a_m\leq p_1\neq p_2\leq i}\big(\te_{p_1}\te_{p_1}^\top \te_{p_1}\te_{p_2}^\top+\te_{p_1}\te_{p_1}^\top \te_{p_2}\te_{p_1}^\top\big)\bigg]\\
    +\textstyle\sum_{i=a_M}^{n}\bigg[&2\textstyle\sum_{j=a_M}^{i-1}\sum_{a_M\leq p_1\neq p_2\leq j}\left(\te_{p_1}\te_{p_1}^\top \te_{p_1}\te_{p_2}^\top+\te_{p_1}\te_{p_1}^\top \te_{p_2}\te_{p_1}^\top\right)\\
    &\qquad\qquad+\textstyle\sum_{a_M\leq p_1\neq p_2\leq i}\left(\te_{p_1}\te_{p_1}^\top \te_{p_1}\te_{p_2}^\top+\te_{p_1}\te_{p_1}^\top \te_{p_2}\te_{p_1}^\top\right)\bigg],
\end{align*}
and,$R_2=\textstyle\sum_{i=1}^n\sum_{j=1}^n\sum_{p=t_i}^i\sum_{q=t_j}^j \te_p\te_p^\top \te_q\te_q^\top$. Now, using equation (46) from \cite{zhu2021online}, we have,
\begin{align}\label{eq:R1bound}
    \left(\textstyle \sum_{i=1}^nl_i\right)^{-2}\norm{\expec{R_1}{}}_2\lesssim M^{-1}.
\end{align}
Similar to equation (47) in \cite{zhu2021online}, we have,
\begin{align}\label{eq:IIinitialdecomp}
\begin{aligned}
    &~~\norm{\left(\textstyle \sum_{i=1}^nl_i\right)^{-2}\expec{R_2}{}-S^2}_2\\
    \lesssim&~\left(\textstyle\sum_{i=1}^{a_{M+1}-1}l_i\right)^{-2}\textstyle\sum_{m=1}^M\sum_{k=1}^M\sum_{i=a_m}^{a_{m+1}-1}\sum_{j=a_k}^{a_{k+1}-1}\sum_{p=a_m}^i\sum_{q=a_k}^j\norm{\expec{\te_p\te_p^\top \te_q\te_q^\top}{}-S^2}_2.
\end{aligned}
\end{align}
At this point we need a more careful analysis to bound the right hand side of \eqref{eq:IIinitialdecomp}. Unlike \cite{zhu2021online}, $\expec{\te_p\te_p^\top \te_q\te_q^\top}{}\neq S^2$ when $p$ and $q$ belong to different blocks since $S$ is defined as the asymptotic covariance of a martingale-difference sequence derived from a $\theta_k$-dependent Markov chain whereas $S$ is the covariance of an $\iid$ sequence in \cite{zhu2021online}. So, now we establish a bound on this bias. We now decompose \eqref{eq:IIinitialdecomp} into two terms.  
\begin{enumerate}[noitemsep]
    \item $p$ and $q$ belong to either same block or neighboring blocks, i.e., $\abs{m-k}\leq 1$; see $R_3$ in~\eqref{eq:r3}.
    \item $p$ and $q$ are at least $1$ block apart, i.e., $\abs{m-k}> 1$; see $R_4$ in~\eqref{eq:r4}.
\end{enumerate}
In the first component, the correlation among the terms are high. So we just use a constant bound for these terms. In the second component, using the fact that the blocks are separated enough, we show that the correlation between the blocks are small. Then we show that these terms $\te_p\te_p^\top \te_q\te_q^\top$ can estimate $S^2$ with small bias.\\

\textbf{Case I, $\abs{m-k}\leq 1$:} Combining the terms where $p$ and $q$ belong to either the same block or neighboring blocks we get, 
\begin{align}\label{eq:r3}
    R_3\coloneqq \mathop{\sum_{m=1}^M\sum_{k=1}^M}\limits_{\abs{m-k}\leq 1}\sum_{i=a_m}^{a_{m+1}-1}\sum_{j=a_k}^{a_{k+1}-1}\sum_{p=a_m}^i\sum_{q=a_k}^j \norm{\expec{\te_p\te_p^\top \te_q\te_q^\top}{}-S^2}_2. 
\end{align}
By Lemma~\ref{lm:liangmainlemma} we have, $\norm{\expec{\te_p\te_p^\top \te_q\te_q^\top}{}}_2\leq C$ for some constant $C>0$. When $n_m\asymp m^{\beta-1}$, we have $$n_m^2+n_{m+1}^2+n_{m-1}^2\leq (4^\beta+2)n_m^2.$$ 
Then, using $\sum_{i=a_m}^{a_{m+1}-1}l_i\lesssim n_m^2$, we have 
\begin{align}\label{eq:diffless1numbound}
    R_3\lesssim \mathop{\sum_{m=1}^M\sum_{k=1}^M}\limits_{\abs{m-k}\leq 1}\sum_{i=a_m}^{a_{m+1}-1}\sum_{j=a_k}^{a_{k+1}-1}l_il_j
    \lesssim \mathop{\sum_{m=1}^M\sum_{k=1}^M}\limits_{\abs{m-k}\leq 1}n_m^2n_k^2
    \lesssim \sum_{m=1}^M n_m^4
    \lesssim M^{4\beta-3}.
\end{align}
We also have, 
\begin{align}\label{eq:li2bound}
    \left(\textstyle \sum_{i=1}^{a_{M+1}-1}l_i\right)^{2}=\left(\textstyle\sum_{m=1}^{M}\sum_{i=a_m}^{a_{m+1}-1}l_i\right)^{2}\asymp \left(\textstyle\sum_{m=1}^{M}n_m^2\right)^{2}\asymp M^{4\beta-2}.
\end{align}
Combining \eqref{eq:diffless1numbound}, and \eqref{eq:li2bound}, we get, 
\begin{align}\label{eq:R3bound}
   \left(\textstyle\sum_{i=1}^{a_{M+1}-1}l_i\right)^{-2} R_3\lesssim M^{-1}.
\end{align}

\textbf{Case II, $\abs{m-k}> 1$:}
Combining the terms where $p$ and $q$ are at least $1$ block apart, i.e., $\abs{m-k}> 1$, we get, 
\begin{align}\label{eq:r4}
    R_4\coloneqq \mathop{\sum_{m=1}^M\sum_{k=1}^M}\limits_{\abs{m-k}> 1}\sum_{i=a_m}^{a_{m+1}-1}\sum_{j=a_k}^{a_{k+1}-1}\sum_{p=a_m}^i\sum_{q=a_k}^j \norm{\expec{\te_p\te_p^\top \te_q\te_q^\top}{}-S^2}_2. 
\end{align}
Let us assume $k\leq m-2$. Then, $    \expec{\te_p\te_p^\top \te_q\te_q^\top|\cF_{a_{m-1}-1}}{}=\expec{\te_p\te_p^\top |\cF_{a_{m-1}-1}}{}\te_q\te_q^\top$. Similar to \eqref{eq:t1t2t3decomp}, we get,
\begin{align}\label{eq:eptildeprimecondn}
    \expec{\te_p\te_p^\top |\cF_{a_{m-1}-1}}{}=\expec{e_p'{e_p'}^\top |\cF_{a_{m-1}-1}}{}-\expec{T_{3,p}|\cF_{a_{m-1}-1}}{}.
\end{align}
Now, for consider the sequence $\{\hat{e}_p\}_p$ defined similarly as in \eqref{eq:hatekdef}. We will choose $b_0$ such that $b_0\leq n_{m-1}+1$. Then similar to \eqref{eq:ekprimeekhatrel}, we have, 
\begin{align}\label{eq:ekprimeekhatrelcond}
    \expec{e_p'{e_p'}^\top|\cF_{a_{m-1}-1}}{}=\expec{\hat e_p\hat e_p^\top|\cF_{a_{m-1}-1}}{}-T'_{4,p}, 
\end{align}
where 
\begin{align*}
    T'_{4,p}=\textstyle\sum_{j=1}^{b_0}\int \mathcal{A}(x_{p-1},x_p)P_{\theta^*}^{b_0-j}(P_{\theta_{p-(b_0-j)}}-P_{\theta^*})P_{\theta_{p-(b_0-j)-1}}\cdots P_{a_{m-1}} \mathrm{d}x_p\mathrm{d}x_{p-1}\cdots \mathrm{d}x_{a_{m-1}},
\end{align*}
and $\norm{T'_{4,p}}_2=\sum_{j=1}^{b_0}\sqrt{\eta_{p-j}}$. Similar to \eqref{eq:Asrelationfinal}, we have,
\begin{align}\label{eq:Asrelationfinalcondn}
    \expec{\mathcal{A}(x'_{p-1},x'_p)|\cF_{a_{m-1}-1}}{}=S+\expec{T_{5,p}|\cF_{a_{m-1}-1}}{}, 
\end{align}
where $\norm{T_{5,p}}_2\leq V(x_{p-b_0})\rho^{b_0}$.
Combining \eqref{eq:eptildeprimecondn}, \eqref{eq:ekprimeekhatrelcond}, and \eqref{eq:Asrelationfinalcondn}, we hence have $\expec{\te_p\te_p^\top |\cF_{a_{m-1}-1}}{}=S+\expec{T_{5,p}|\cF_{a_{m-1}-1}}{}-T'_{4,p}-\expec{T_{3,p}|\cF_{a_{m-1}-1}}{}$. Then, we have, 
\begin{align*}
    \expec{\te_p\te_p^\top \te_q\te_q^\top}{}=S\expec{\te_q\te_q^\top}{}+\expec{\left(T_{5,p}-T'_{4,p}-T_{3,p}\right)\te_q\te_q^\top}{}.
\end{align*}
Similar to \eqref{eq:ekekdecomp}, we also have that $\expec{\te_q\te_q^\top}{}=S-\expec{T_{3,q}}{}-T_{4,q}+\expec{T_{5,q}}{}$. Similar to \eqref{eq:T3kbound}, for $\alpha_0\geq 8$, we also have, 
\begin{align*}
    \expec{\norm{T_{3,p}}_2^\frac43}{}\leq \sqrt{\expec{\norm{\theta_{p-1}-\theta^*}_2^4\expec{V(x_{p-1})^8}{}}{}}\lesssim\eta_p.
\end{align*}
Then, we have that
\begin{align*}
    &\norm{\expec{\te_p\te_p^\top \te_q\te_q^\top}{}-S^2}_2\\
    \leq&~\norm{\expec{\left(T_{5,p}-T'_{4,p}-T_{3,p}\right)\te_q\te_q^\top}{}}_2+\norm{S\left(\expec{T_{3,q}}{}+T_{4,q}-\expec{T_{5,q}}{}\right)}_2\\
    \leq &~\expec{\norm{\left(T_{5,p}-T'_{4,p}-T_{3,p}\right)}_2\norm{\te_q\te_q^\top}_2}{}+\norm{S}_2\norm{\left(\expec{T_{3,q}}{}+T_{4,q}-\expec{T_{5,q}}{}\right)}_2\\
    \lesssim &~\sqrt{\expec{\norm{\left(T_{5,p}-T'_{4,p}\right)}_2^2{}}{}\expec{\norm{\te_q\te_q^\top}_2^2}{}}+\expec{\norm{T_{3,p}}_2^\frac43}{}\expec{\norm{\te_q\te_q^\top}_2^4}{}+\textstyle\sum_{j=0}^{b_0-1}\sqrt{\eta_{q-j}}\\
    \lesssim &~\textstyle \sum_{j=0}^{b_0-1}\sqrt{\eta_{p-j}}+\sum_{j=0}^{b_0-1}\sqrt{\eta_{q-j}}\\
    \lesssim &~\sqrt{\eta_p}\log p+\sqrt{\eta_q}\log q.
\end{align*}
Then, using $a_m\asymp m^\beta$, and $n_m\asymp (m+1)^{\beta-1}$, we have, 
\begin{align*}
    R_4\lesssim &\mathop{\sum_{m=1}^M\sum_{k=1}^M}\limits_{\abs{m-k}> 1}\sum_{i=a_m}^{a_{m+1}-1}\sum_{j=a_k}^{a_{k+1}-1}\sum_{p=a_m}^i\sum_{q=a_k}^j \left(\sqrt{\eta_p}\log p+\sqrt{\eta_q}\log q\right)\\
    \lesssim & \mathop{\sum_{m=1}^M\sum_{k=1}^M}\limits_{\abs{m-k}> 1} \bigg(n_k^2\sum_{i=a_m}^{a_{m+1}-1}\sum_{p=a_m}^i\sqrt{\eta_p}\log p+n_m^2\sum_{j=a_k}^{a_{k+1}-1}\sum_{q=a_k}^j\sqrt{\eta_q}\log q\bigg)\\
    \lesssim & \sum_{m=1}^M\sum_{k=1}^M\left(n_k^2n_m^2\left(a_m^{-\frac{a}{2}}\log a_{m+1}+a_k^{-\frac{a}{2}}\log a_{k+1}\right)\right)\\
    \lesssim & M^{2+4(\beta-1)-\frac{a\beta}{2}}\log M.
\end{align*}
Using \eqref{eq:li2bound} we have,
\begin{align}\label{eq:IVboundeeee}
    \left(\textstyle \sum_{i=1}^{a_{M+1}-1}l_i\right)^{-2}R_4\lesssim M^{-\frac{a\beta}{2}}\log M.
\end{align}
So, combining \eqref{eq:snstr2normrel}, \eqref{eq:sn2decompmid}, \eqref{eq:R1bound}, \eqref{eq:R3bound}, and \eqref{eq:IVboundeeee}, we get,
\begin{align*}
    \expec{\|\tS_n-S\|_2}{}\leq M^{-\frac{a\beta}{4}}\sqrt{d\log M}+\sqrt{d}M^{-1/2}+M^{-\frac{a\beta}{2}}\sqrt{d\log M}.
\end{align*}
\end{proof}

%% file: stateindepexpts.tex
\section{Experiments on State Independent Markov Chain}
Though our main goal in this paper is online inference for SGD with state-dependent Markovian data, the estimator $\hat{\Sigma}_n$ can be used under state-independent Markovian data as well as we show in Theorem~\ref{th:simain}. In this section, we compare \texttt{BM} with \texttt{Boot} for completion. We set $\varepsilon=0$ in \eqref{eq:dgplinregsynth} to generate exponentially mixing Markovian data. It is worth mentioning here that the theoretical guarantees of \texttt{Boot} hold in this setting when $\nabla F(\theta,x)$ is linear in $\theta$ which does not hold for important applications like logistic regression. In contrast, Theorem~\ref{th:simain} allows for non-linearity in $\nabla F(\theta,x)$. In Figure~\ref{fig:boot_bm_linlogregsi}, we compare the performances of \texttt{BM} with \texttt{Boot} in linear (upper row) and logistic (lower row) regression. We observe similar behavior as the state-dependent Markovian data setting in this case as well.
\begin{sidewaysfigure}
    \centering
    \includegraphics[width=.245\textwidth,height=2in]{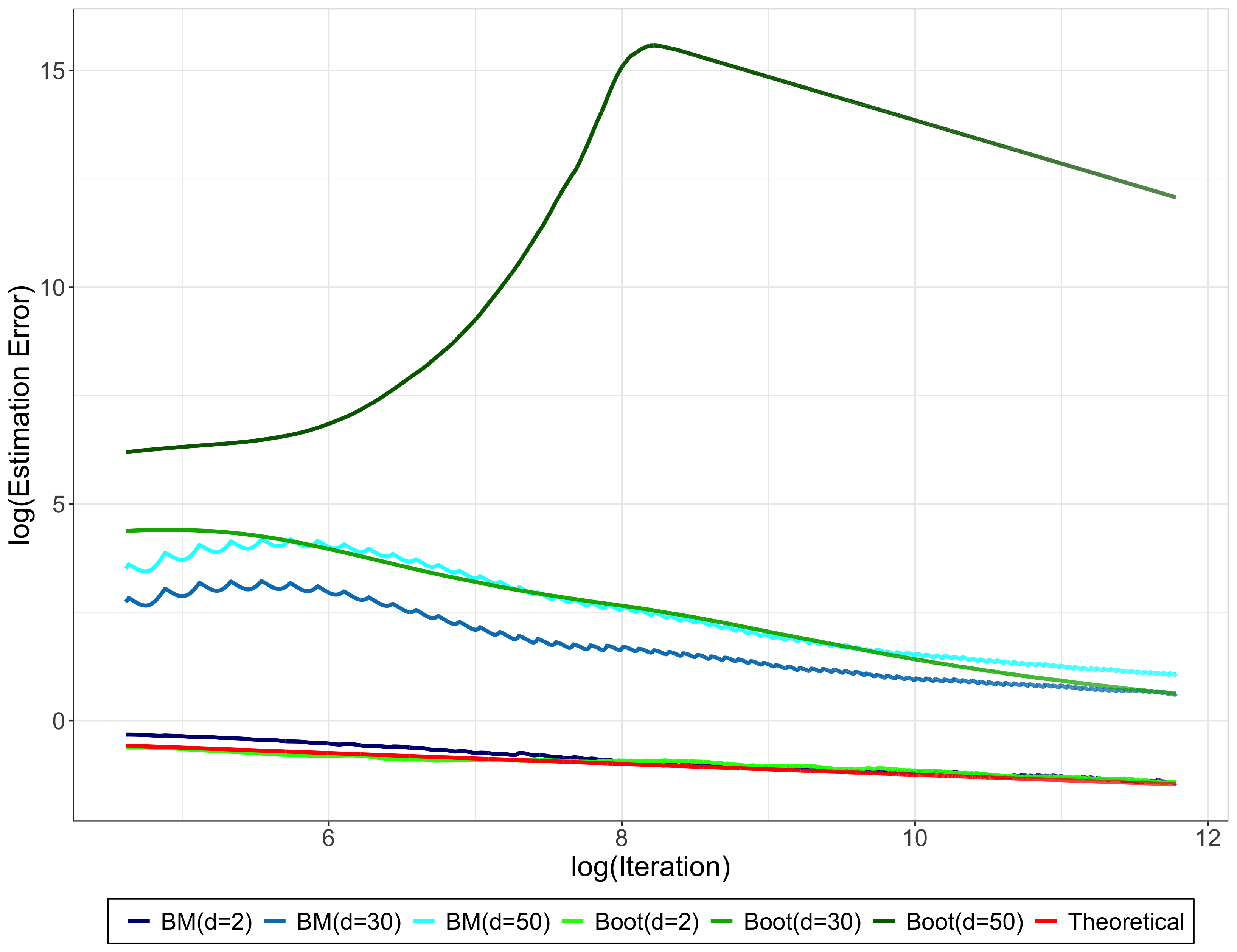} 
    \includegraphics[width=.245\textwidth,height=2in]{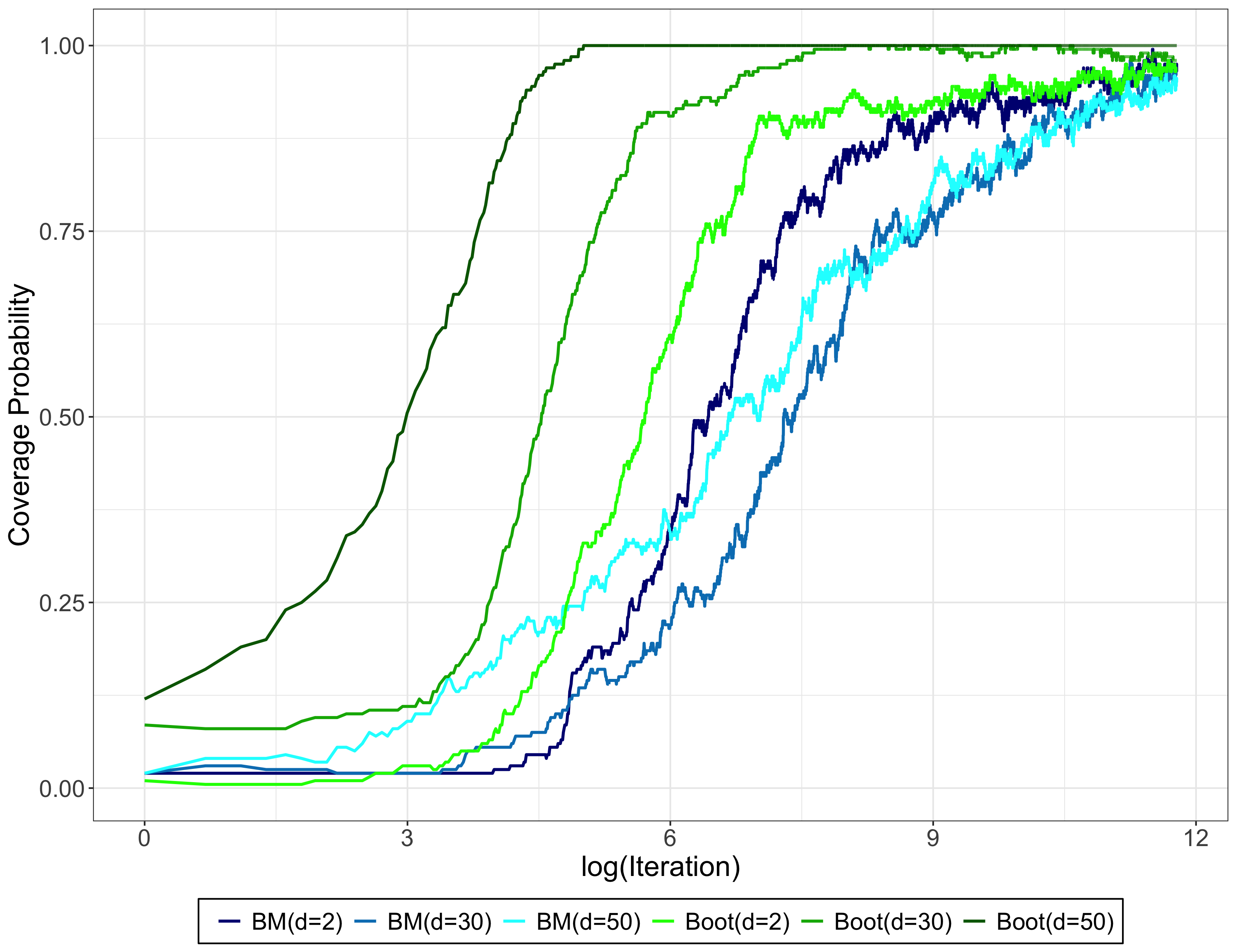}
    \includegraphics[width=.245\textwidth,height=2in]{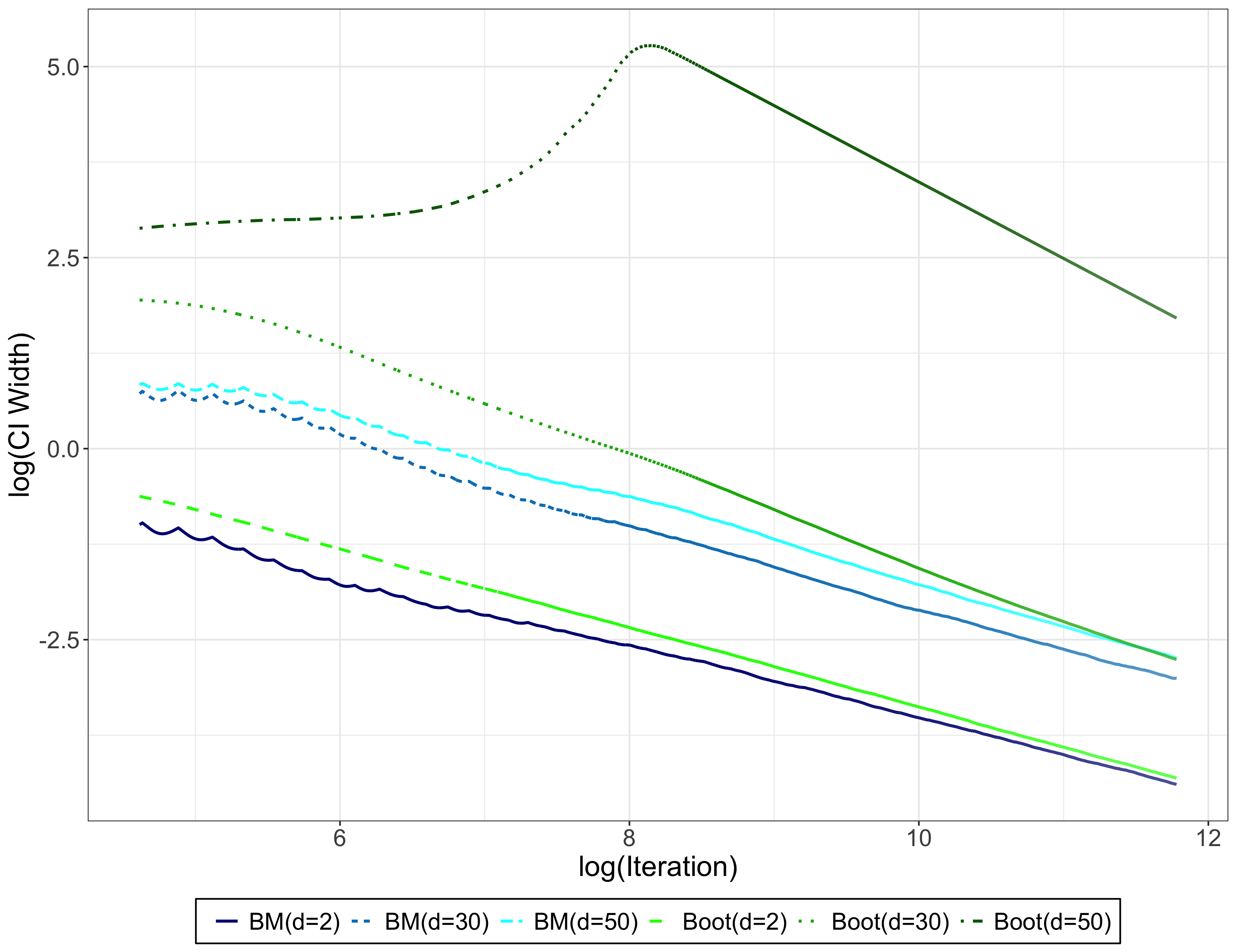}
    \includegraphics[width=.245\textwidth,height=2in]{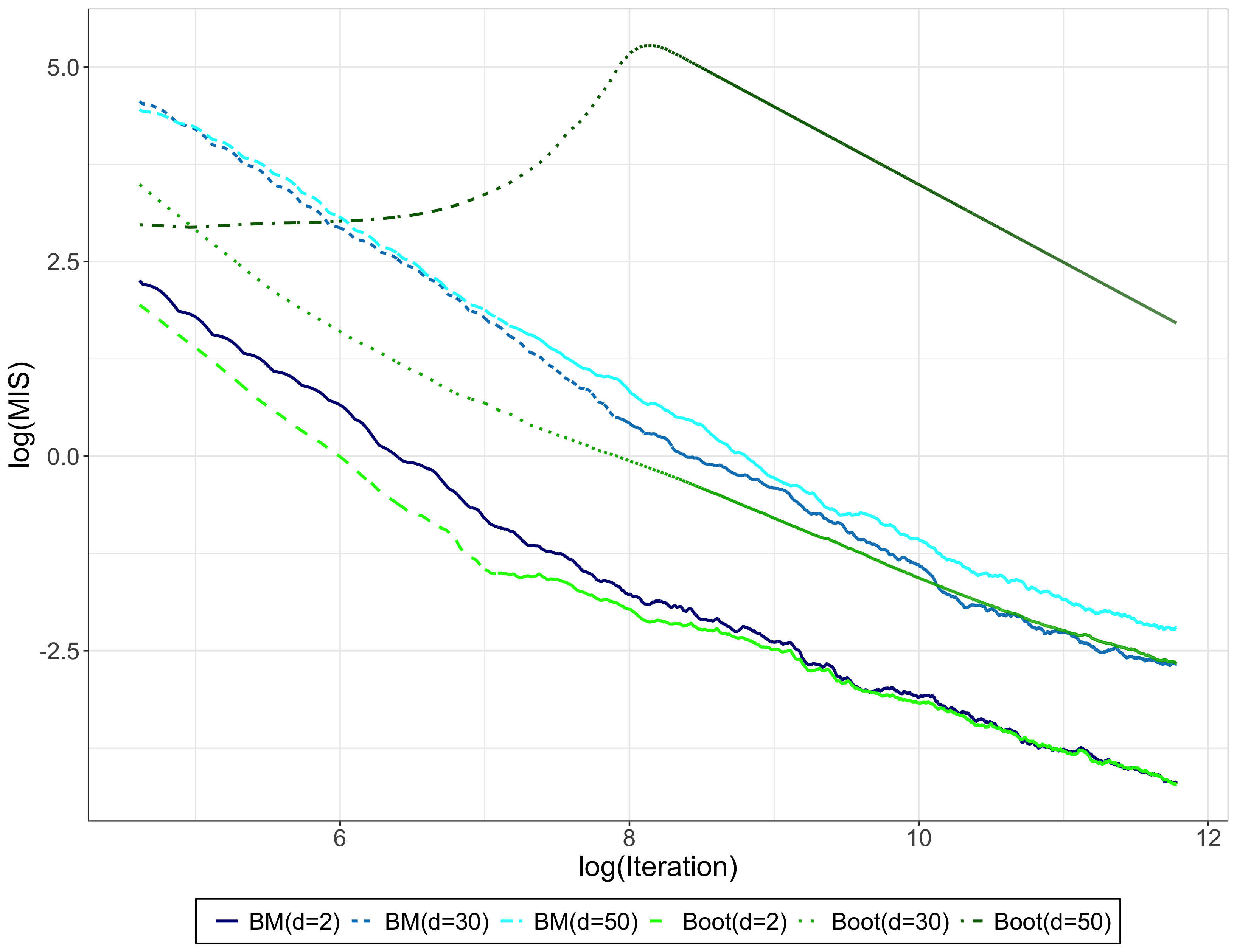}\\
    \includegraphics[width=.245\textwidth,height=2in]{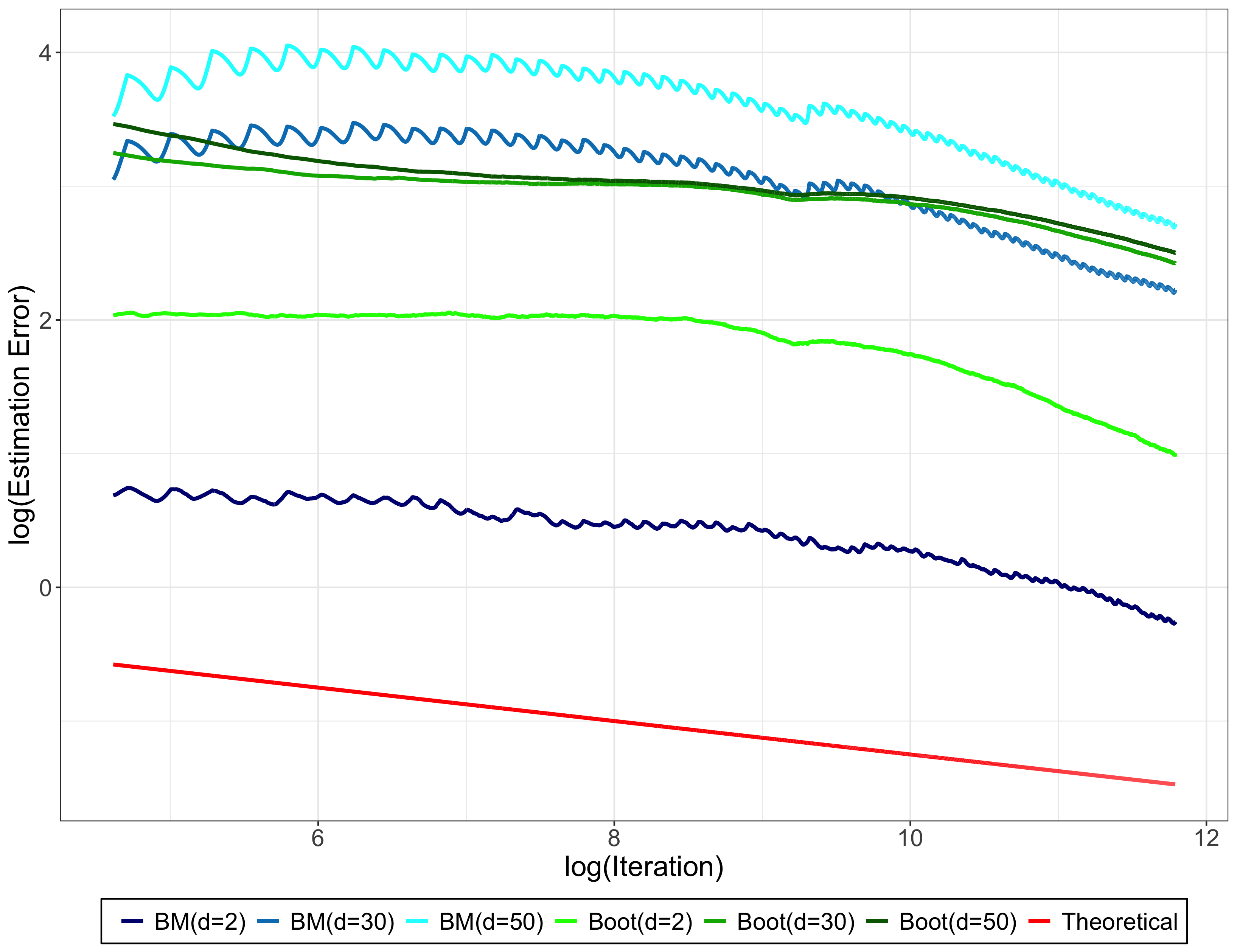}   
    \includegraphics[width=.245\textwidth,height=2in]{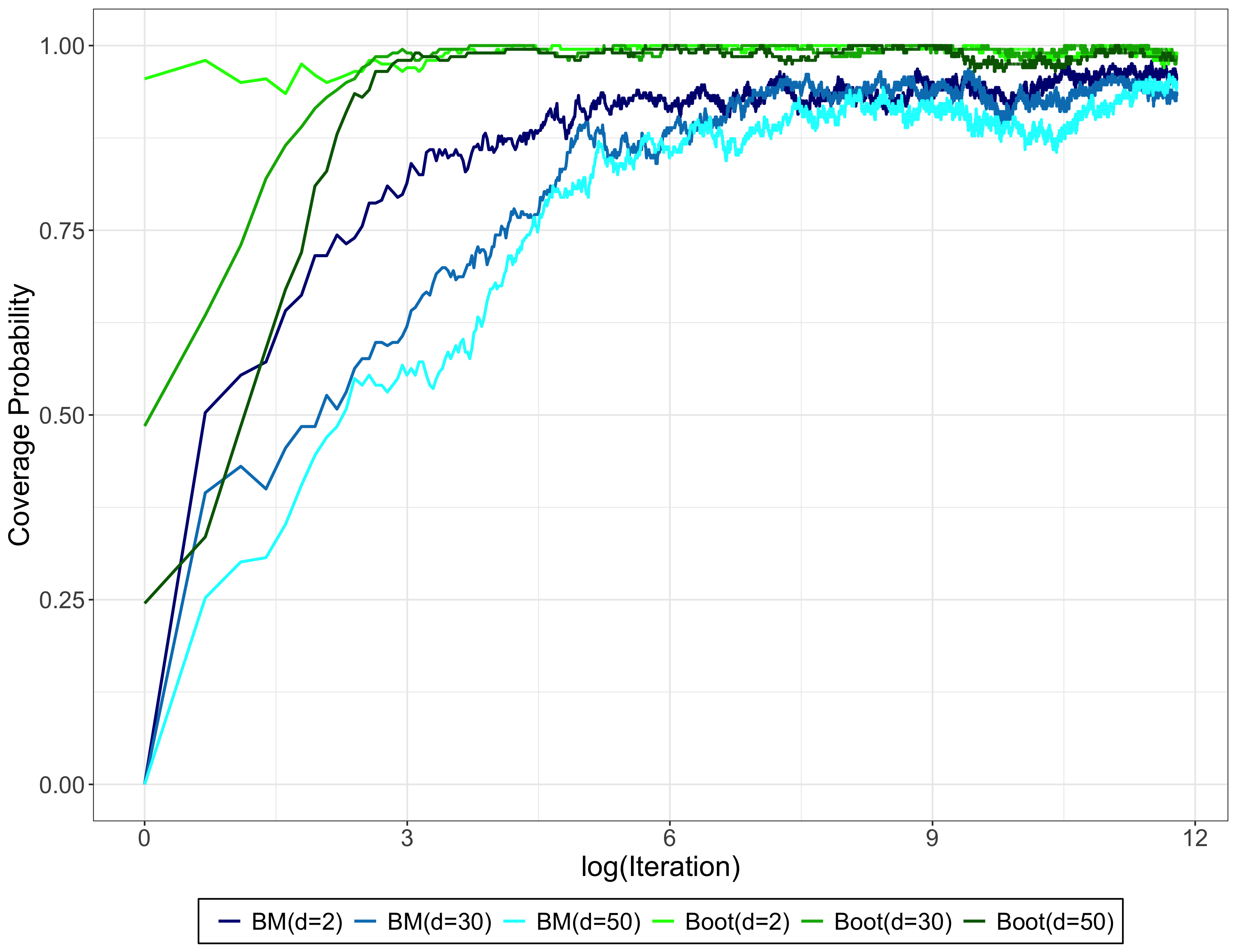}   
    \includegraphics[width=.245\textwidth,height=2in]{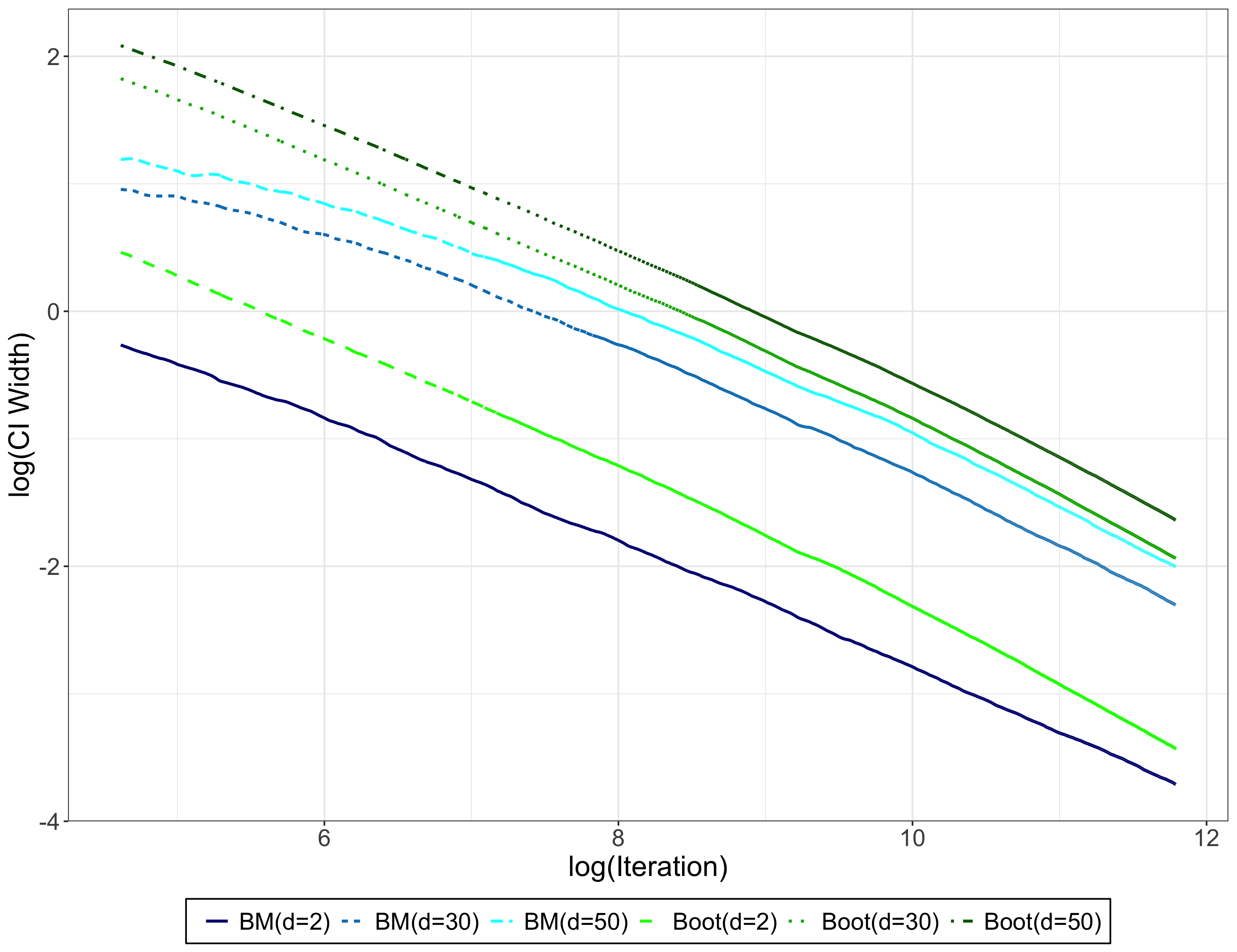}
    \includegraphics[width=.245\textwidth,height=2in]{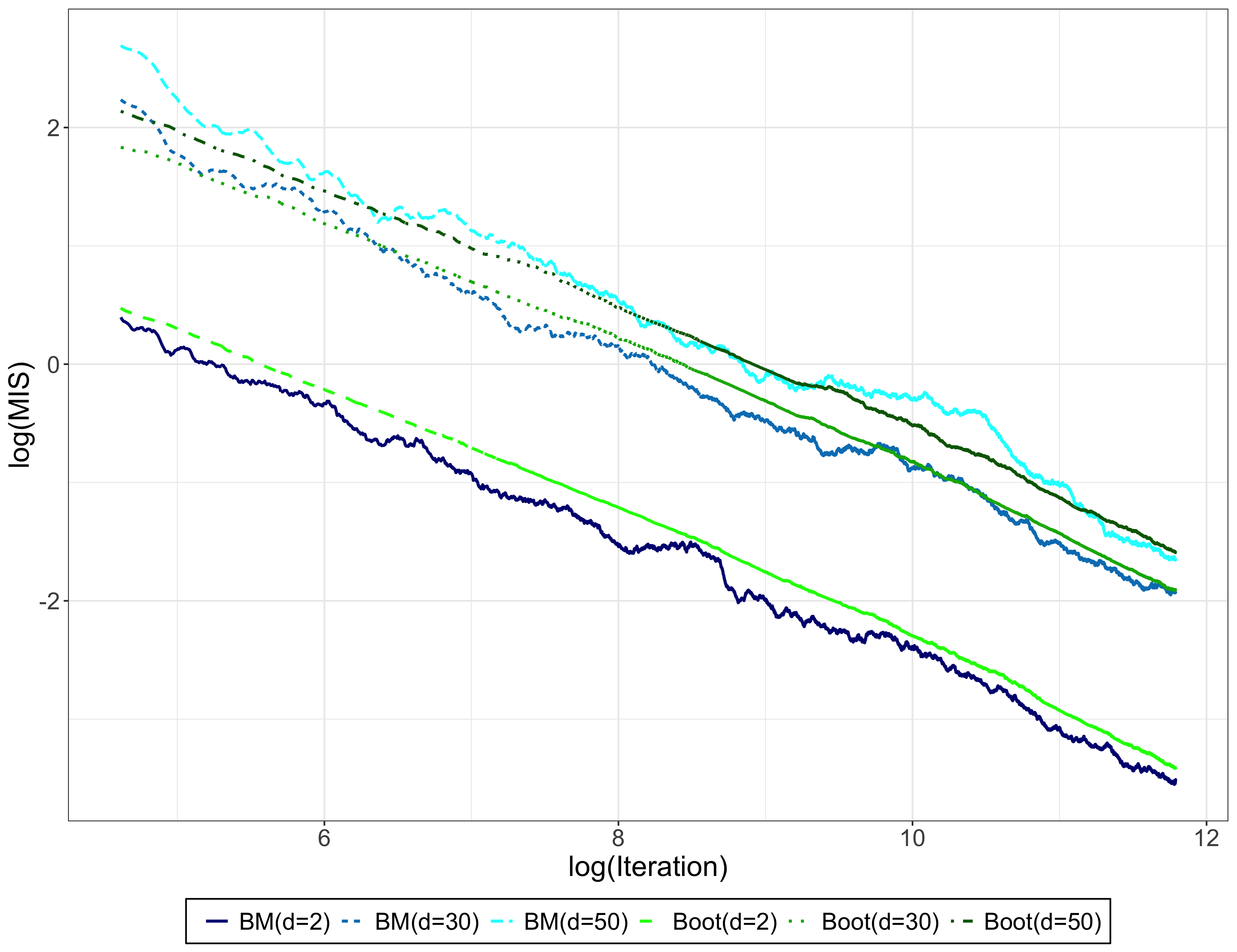}
    \caption{Linear regression (upper) and logistic regression (lower) with synthetic state-independent Markovian data: In each row, from left to right columns are the plots of $\log(\text{Estimation Error})$, coverage probability, $\log(\text{width of Confidence Interval (CI)})$, and $\log(\text{MIS})$ respectively. The blue (green) lines correspond to \texttt{BM} (\texttt{Boot}) for dimensions $d=2,30,50$. The red line in the plots of $\log(\text{Estimation Error})$ corresponds to the theoretical rate obtained in Theorem~\ref{pf:mainthm}. Similar to state-dependent Markovian data setting, the results show that the bootstrap-based method achieves a higher coverage probability by constructing wider confidence interval. The plot of $\log(\text{MIS})$ shows that, \texttt{BM} performs better than \texttt{Boot}, especially in high dimensional linear regression. }
    \label{fig:boot_bm_linlogregsi}
\end{sidewaysfigure}
\clearpage
\section{Comparison of Computation Time of \texttt{BM} with \texttt{Boot}}
In each iteration, \texttt{Boot} requires several gradient computations which can be expensive in practice. In Table~\ref{tab:timetable}, we present a comparison of the average computation time of the \texttt{BM} and \texttt{Boot} estimators in each iteration for various applications. In higher dimension, \texttt{Boot} requires more bootstrap replicates increasing the computation time which can be $\sim 8$ times as large as the computation time of \texttt{BM}. 
\begin{table}[t]
\begin{centering}
\begin{tabular}{|ll|ll|ll|}
\hline
\multicolumn{2}{|l|}{\multirow{2}{*}{}}                           & \multicolumn{2}{l|}{Linear Regression}                                             & \multicolumn{2}{l|}{Logistic Regression}                                           \\ \cline{3-6} 
\multicolumn{2}{|l|}{}                                            & \multicolumn{1}{l|}{\texttt{BM}} & \texttt{Boot} & \multicolumn{1}{l|}{\texttt{BM}} & \texttt{Boot} \\ \hline
\multicolumn{1}{|l|}{\multirow{3}{*}{State Dependent}}   & $d=2$  & \multicolumn{1}{l|}{0.28}                         & 0.63                           & \multicolumn{1}{l|}{0.29}                         & 0.76                           \\ \cline{2-6} 
\multicolumn{1}{|l|}{}                                   & $d=30$ & \multicolumn{1}{l|}{0.49}                         & 1.37                           & \multicolumn{1}{l|}{0.51}                         & 1.63                           \\ \cline{2-6} 
\multicolumn{1}{|l|}{}                                   & $d=50$ & \multicolumn{1}{l|}{0.70}                         & 3.39                           & \multicolumn{1}{l|}{0.68}                         & 3.88                           \\ \hline
\multicolumn{1}{|l|}{\multirow{3}{*}{State Independent}} & $d=2$  & \multicolumn{1}{l|}{0.12}                         & 0.23                           & \multicolumn{1}{l|}{0.076}                        & 0.18                           \\ \cline{2-6} 
\multicolumn{1}{|l|}{}                                   & $d=30$ & \multicolumn{1}{l|}{0.18}                         & 0.59                           & \multicolumn{1}{l|}{0.12}                         & 0.47                           \\ \cline{2-6} 
\multicolumn{1}{|l|}{}                                   & $d=50$ & \multicolumn{1}{l|}{0.25}                         & 1.70                           & \multicolumn{1}{l|}{0.18}                         & 1.38                           \\ \hline
\end{tabular}
\caption{Comparison of the average per-iteration  computation time (ms) of \texttt{BM} and \texttt{Boot}.}
\label{tab:timetable}
\end{centering}
\end{table}